\documentclass[12pt]{amsart}
\usepackage{graphics,wrapfig, graphicx, mathrsfs,xcolor}
\usepackage{mathtools}
\usepackage{amssymb,amscd,graphicx,xcolor,subfig,tikz}
\usepackage{graphics}
\usepackage{indentfirst}
\usepackage{bm, enumerate,xcolor}
\usepackage[dvips]{epsfig}
\usepackage{latexsym}
\usepackage{float}
\usepackage[colorlinks]{hyperref}
\AtBeginDocument{
   \hypersetup{
    linkcolor=blue,
    citecolor=blue,
 }
}

\usepackage{amsmath}
\usepackage{amssymb}
\usepackage[applemac]{inputenc}
\usepackage[T1]{fontenc}
\newtheorem{lemma}{Lemma}[section]
\newtheorem{theorem}{Theorem}[section]
\newtheorem{corollary}{Corollary}[section]

\newtheorem{definition}{Definition}[section]
\newtheorem{remark}{Remark}[section]

\newtheorem{example}{Example}[section]

\numberwithin{equation}{section} \numberwithin{theorem}{section}
\numberwithin{example}{section} \numberwithin{remark}{section}
\numberwithin{figure}{section} \numberwithin{algorithm}{section}
\mathtoolsset{showonlyrefs}

\title[Schauder estimate in large scale for planar cavity flows]{Schauder estimate in large scale for compressible planar cavity flows with fast-decaying vorticity}
\author{Chilin Zhang}
\address{School of Mathematical Sciences, Fudan University, Shanghai 200433, China}\email{zhangchilin@fudan.edu.cn}

\begin{document}
\begin{abstract}
    In this paper, we study planar compressible cavity flows fast-decaying vorticity and establish a Schauder estimate in large scale. As an application, we obtain an a priori asymptotic growth rate estimate of the cavity, that $s^{\frac{1}{2}-\delta}\lesssim h(s)\lesssim s^{\frac{1}{2}+\delta}$. This is by far the most accurate asymptotic estimate for compressible cavity flows and is probably the first paper that takes vorticity into account in the asymptotic analysis.
\end{abstract}
\maketitle
2020 {\bf MSC.} 76B10, 35R35, 35B40, 35B65.

\section{Introduction}
Cavity and jet flows are classical free boundary problems in fluid mechanics, sharing the same physical origin. In both cases, a free surface separates the fluid from vacuum or a low-pressure region behind a rigid body. Their mathematical theory originated from the hodograph and conformal mapping methods in Helmholtz \cite{H1868}, Kirchhoff \cite{K1869}, Leray \cite{L35} and Berg \cite{B62}, and was later reformulated in the variational framework of Alt-Caffarelli-Friedman \cite{ACF82,ACF83,ACF85}. Since our focus is on the cavity problem, we refer the reader to \cite{ACF85,CDW20,LSTX25} for the jet literature and proceed directly to the cavity setting.

The well-posedness of the 2D cavity problem was original established by Leray \cite{L35} for the incompressible case, and his result was extended to the compressible case in Berg \cite{B62}. For the axially symmetric setting, early results on infinite cavities were obtained by Caffarelli-Friedman \cite{CF82}, while Alt-Caffarelli-Friedman \cite{ACF85} later established well-posedness for compressible axially symmetric cavity flows via the variational method, extending \cite{CF82} to non-monotone obstacles.

Later, Alt-Caffarelli-Friedman applied the variational method and established the well-posedness of compressible 3D axial symmetric cavity flow in \cite{ACF85}. More recently, Wang \cite{W21} established well-posedness for 2D compressible subsonic cavity flows in an infinitely long nozzle, with additional properties such as the positivity of the horizontal velocity and the far-field asymptotic behavior.

For a cavity flow, it is an interesting problem to determine the asymptotic growth rate of the free boundary. In \cite{L46}, Levinson studied the growth rate of incompressible 3D axial symmetric cavity flow, and established the estimate
\begin{equation*}
    \frac{\sqrt{r}}{(\ln{r})^{\frac{1}{4}+\epsilon}}\leq h(r)\leq\frac{\sqrt{r}}{(\ln{r})^{\frac{1}{4}-\epsilon}}.
\end{equation*}
However, his proof is based on a power-law ansatz with a slowly varying correction at the level of formal asymptotics. Besides, the convergence of certain integrals in his derivation was later shown to be problematic in \cite{LLWX26}. In \cite{LLWX26}, Li-Lin-Weiss-Xie rigorously established the Levinson asymptotics by a potential theoretic reduction to ODE approach, proving both the rigidity of the asymptotic profile and the existence of a solution with this behavior. Both \cite{L46} and \cite{LLWX26} applied the integral equation method, relying heavily on the linear structure. Recently, in \cite{KW26}, Kriventsov and Weiss developed a robust potential analysis method that applies to equations with quasi-linear principal part, which in their regime is not treated as a perturbation of a linear operator.

For compressible flows, the result is considerably less accurate: the best estimate so far was obtained in \cite{ACF85}, where the authors showed that $h(r)=o(r)$ as $r\to\infty$. Moreover, to the best of our knowledge, the well-posedness of compressible cavity flows with vorticity remains open. The only available result for rotational cavity flows is Cheng-Du-Zhang \cite{CDZ17}, who established existence and uniqueness for the two-dimensional incompressible cavity flow with non-zero vorticity in an infinitely long channel; however, their analysis does not cover the compressible case, nor does it address the asymptotic growth rate of the free boundary. On the compressible side, the only relevant work is Li-Shi-Tang-Xie \cite{LSTX25}, where well-posedness is established for 2D compressible rotational jet flows, rather than cavity flows.

In this paper, we study the asymptotic growth rate of the free boundary for 2D compressible cavity flows with small vorticity. We assume that the velocity field is symmetric about the $x_{1}$-axis and that there is no stagnation point. Since the well-posedness of compressible cavity flows with vorticity remains open, our growth rate estimate is necessarily a priori: it does not rely on the existence of a global solution. Our main result (Theorem~\ref{thm. growth rate of free boundary}) states that, under the assumptions (A1)-(A6), the free boundary satisfies
\begin{equation*}
    s^{\frac{1}{2}-\delta}\leq h(s)\leq s^{\frac{1}{2}+\delta}\quad\mbox{for all }\delta>0\mbox{ and }s\geq\underline{s}(\delta),
\end{equation*}
This is, to the best of our knowledge, the first power-type growth rate estimate for compressible cavity flows with non-zero vorticity. Moreover, the estimate is highly localized: it does not require the flow to be globally defined in $\mathbb{R}^{2}_{+}$, in contrast to \cite{L46,LLWX26}.

\subsection{Problem formulation} The cavity problem consists of fluids being separated by a nose. In this paper, we assume that the nose is described by a function $N(x_{1})$ defined for $x_{1}\in[-L,0]$, such that
\begin{equation*}
    N(-L)=0<N(0),\quad\mbox{and }N'(x_{1})\geq0\mbox{ in }[-L,0].
\end{equation*}
The nose, or the obstacle region, is defined by
\begin{equation*}
    \mbox{obstacle}=\{(x_{1},x_{2}):-L\leq x_{1}\leq0,\ |x_{2}|\leq N(x_{1})\},
\end{equation*}
For convenience, we extend $N(x_{1})$ trivially in $(-\infty,L]$. To the left (namely, the $-\vec{e}_{1}$ direction) of the nose, a shear flow is flowing in the $+\vec{e}_{1}$ direction, and the velocity field is symmetric about the $x_{1}$-axis. The cavity region caused by the nose is assumed to be
\begin{equation*}
    \mbox{cavity region}=\{(x_{1},x_{2}):x_{1}\geq0,\ |x_{2}|\leq h(x_{1})\},
\end{equation*}
where $h:\mathbb{R}_{+}\to\mathbb{R}_{+}$ is assumed to be an increasing function with $h(0)=N(0)$. In the cavity region, there is no fluid and can be approximated by the vacuum.

By assuming the symmetry of the incoming flow about the $x_{1}$-axis, one can consider a problem in $\mathbb{R}^{2}_{+}$ as follows: Let $\Gamma$ be the border (a special stream line) of the flow, with
\begin{equation*}
    \Gamma:=\begin{cases}
        x_{2}=N(x_{1})=0,&x_{1}\leq-L,\\
        x_{2}=N(x_{1}),&-L\leq x_{1}\leq0\\
        x_{2}=h(x_{1}),&x_{1}\geq0.
    \end{cases}
\end{equation*}
The fluid is flowing above $\Gamma$, namely in the region $\Gamma^{+}$ given below:
\begin{equation*}
    \Gamma^{+}:=\begin{cases}
        x_{2}>N(x_{1})=0,&x_{1}\leq-L,\\
        x_{2}>N(x_{1}),&-L\leq x_{1}\leq0\\
        x_{2}>h(x_{1}),&x_{1}\geq0.
    \end{cases}
\end{equation*}
In the region $\Gamma^{+}$, the flow is described by a compressible steady Euler flow
\begin{equation}\label{eq. initial main}
    \left\{\begin{aligned}
        &\nabla\cdot(\rho\mathbf{u})=0,\\
        &\rho\mathbf{u}\cdot\nabla\mathbf{u}+\nabla p(\rho)=0.
    \end{aligned}\right.
\end{equation}
When there is no stagnation point (i.e. $\mathbf{u}=0$), then there is a stream function representation $\psi$, such that $\rho\mathbf{u}=(\psi_{y},-\psi_{x})$, and the stream function satisfies the following quasi-linear equation:
\begin{equation}\label{eq. quasilinear structure}
    \nabla\cdot\Big(g(|\nabla\psi|^{2},\psi)\nabla\psi\Big)=\frac{\mathcal{B}'(\psi)}{g(|\nabla\psi|^{2},\psi)}.
\end{equation}
In \eqref{eq. quasilinear structure}, the function $g$ represents the $1/\rho$, and the right-hand side corresponds to the vorticity of the flow. Using the Bernoulli law, one has the following equation:
\begin{equation}\label{eq. physics structure}
    \frac{T}{2}\cdot g(T,z)^{2}+h\Big(\frac{1}{g(T,z)}\Big)=\mathcal{B}(z),\quad\mbox{where }h(\rho)=\int_{0}^{\rho}\frac{1}{\widetilde{\rho}}\frac{dp}{d\widetilde{\rho}}d\widetilde{\rho}+C.
\end{equation}
A brief derivation of the stream function representation will be give in Section~\ref{sec. Basis on steady compressible planar Euler flows}.

Since $\Gamma$ is a special stream line, $\psi$ must be constant on $\Gamma$. without loss of generality, we assume that $\psi\equiv0$ on $\Gamma$. Since the pressure at the cavity is constant, one can imagine that $p(\rho)\equiv constant$ on the free boundary $\{(x_{1},x_{2}):x_{1}\geq0,\ x_{2}=h(x_{1})\}$. Then applying \eqref{eq. physics structure} gives that $|\nabla\psi|$ is a positive constant on the free boundary, as long as the flow is sub-sonic. In this paper, we let this positive constant be $\Lambda$.

Then, in the first quadrant where $\Gamma$ coincides with the free boundary, the Euler equation \eqref{eq. quasilinear structure} together with the free boundary assumption can be simplified into the following standard Bernoulli-type one-phase problem:
\begin{equation}\label{eq. main simple}
    \left\{\begin{aligned}
    &\nabla\cdot\Big(g(|\nabla\psi|^{2},\psi)\nabla\psi\Big)=\frac{\mathcal{B}'(\psi)}{g(|\nabla\psi|^{2},\psi)}&\mbox{in }&\{\psi>0\},\\
    &|\nabla\psi|=\Lambda&\mbox{on }&\{\psi=0\}\cap\partial\{\psi>0\}.
    \end{aligned}\right.
\end{equation}
Among all solutions to \eqref{eq. main simple}, shear flow solutions, namely solutions depending only on $x_{2}$, are the simplest and can be analyzed via ODE methods. See the definition below:
\begin{definition}\label{def. shear flow}
    We let $\Psi(x)=\Psi(x_{2})$ be a shear flow solution to \eqref{eq. quasilinear structure} satisfying the following properties
    \begin{equation}\label{eq. 1D solution Psi}
    \left\{\begin{aligned}
        &\Big(g(|\Psi_{t}|^{2},\Psi)\Psi_{t}\Big)_{t}=\frac{\mathcal{B}'(\Psi)}{g(|\Psi_{t}|^{2},\Psi)}\mbox{ and }\Psi>0\mbox{ for }t>0,\\
        &\Psi\equiv0\mbox{ for }t\leq0,\quad\Psi_{t}(0+)=\Lambda,\\
        &\Psi_{t}\geq C^{-1}\mbox{ in }\mathbb{R}_{+},\quad\mbox{and }\|\Psi_{t}\|_{L^{\infty}(\mathbb{R}_{+})}<Q_{sound}.
    \end{aligned}\right.
    \end{equation}
\end{definition}

In the upper plane $\mathbb{R}^{2}_{+}$, we assume that the cavity flow is sub-sonic and is asymptotically flat. Precisely, we give the following definition:
\begin{definition}\label{def. asymptotically flat}
    We say that the stream function $\psi$ for a sub-sonic cavity flow without stagnation is asymptotically flat, if $0<|\nabla\psi|<Q_{sound}$ everywhere (see Subsection~\ref{subset. sub-sonic} for $Q_{sound}$) in $\Gamma^{+}$, and if there exists a function $\epsilon_{R}$ with $\displaystyle\lim_{R\to\infty}\epsilon_{R}\to0$, such that
\begin{equation}\label{eq. cavity main problem}
    \left\{\begin{aligned}
        &\nabla\cdot\Big(g(|\nabla\psi|^{2},\psi)\nabla\psi\Big)=\frac{\mathcal{B}'(\psi)}{g(|\nabla\psi|^{2},\psi)}\mbox{ and }\psi>0&\mbox{in }&\Gamma^{+},\\
        &\psi=0&\mbox{on }&\Gamma,\\
        &|\nabla\psi|=\Lambda&\mbox{on }&\{x_{2}=h(x_{1})\},\\
        &\Psi(x_{2}-\epsilon_{R}R)\leq\psi(x)\leq\Psi(x_{2})&\mbox{in }&[-R,R]^{2}.
    \end{aligned}\right.
\end{equation}
In practice, for convenience, we would like to choose $\epsilon_{R}$ as small as possible, such that for any $\epsilon'<\epsilon_{R}$, we always have
\begin{equation*}
    \psi(x)<\Psi(x_{2}-\epsilon'R)\quad\mbox{for some }x\in[-R,R]^{2}.
\end{equation*}
\end{definition}

\subsection{Main results}
In this paper, we are interested in determining the growth rate of the free boundary $x_{2}=h(x_{1})$. Since the well-posedness of 2D compressible cavity flow with vorticity is by far unknown, the growth rate estimate in this paper has to be an a priori one. Nevertheless, since it was shown in Berg \cite{B62} the existence of irrotational flows, the main results in this paper is definitely valid as least in the irrotational case.

In this paper, we make the following assumptions:
\begin{itemize}
    \item[(A1)] The functions $p(\rho)$ and $h(\rho)$ depend only on the properties of the fluid (come from experimental data or theoretical computation). In this paper, we only assume that they are second order differentiable for all $\rho\in(0,+\infty)$
    \item[(A2)] The physical identity \eqref{eq. physics structure} holds, and so do its compatibility conditions.
    \item[(A3)] There is a shear flow solution $\Psi(x_{2})$ satisfying Definition~\ref{def. shear flow}.
    \item[(A4)] The shear flow solution satisfies $\|\Psi_{t}\|_{L^{\infty}(\mathbb{R}_{+})}<Q_{sound}$ and
    \begin{equation*}
        \frac{\partial}{\partial T}\Big(T\cdot g(T,z)^{2}\Big)\in[C^{-1},+\infty),\quad\mbox{for all }(T,z)\in[0,Q_{sound}^{2})\times[0,+\infty).
    \end{equation*}
    See Subsection~\ref{subset. sub-sonic} for the discussion of sub-sonic flows and $Q_{sound}$.
    \item[(A5)] The inverse of density, i.e. $g(T,z)$, is a $C^{2}$ function, such that
    \begin{equation*}
        |D^{2}g|\leq C\quad\mbox{for }(T,z)\in[0,L^{2}]\times[0,+\infty).
    \end{equation*}
    \item[(A6)] The function $\mathcal{B}(z)$ satisfies
    \begin{equation*}
        \mathcal{B}(0)=\frac{\Lambda^{2}}{2}g(\Lambda^{2},0)^{2}\quad\mbox{and }\lim_{z\to+\infty}z\cdot\mathcal{B}'(z)=0.
    \end{equation*}
    Moreover, $\mathcal{B}(+\infty)$ is strictly positive, such that
    \begin{equation*}
        \mathcal{B}(+\infty)=\lim_{\substack{T\to V\\z\to\infty}}\frac{T}{2}\cdot g(T,z)^{2},\quad\mbox{for some }V\in(0,Q_{sound}).
    \end{equation*}
\end{itemize}
Throughout this paper (including the Appendices), all theorems, lemma, corollaries take these assumptions without further mentioning.

Our first main result is the following Schauder estimate in large scale.
\begin{theorem}\label{thm. Schauder mixed}
    There exist an uniform constants $C>0$ and two functions $\overline{\epsilon}(\delta,L,N),M(\delta,L,N)>0$, such that the following fact holds for all $\delta>0$, $\epsilon\leq\overline{\epsilon}\big(\delta,L,N(0)\big)$, and all $\displaystyle R\geq M\big(\delta,L,N(0)\big)/\epsilon$: Assume that $\psi$ is a sub-sonic solution satisfying
    \begin{equation}\label{eq. cavity local problem}
    \left\{\begin{aligned}
        &\nabla\cdot\Big(g(|\nabla\psi|^{2},\psi)\nabla\psi\Big)=\frac{\mathcal{B}'(\psi)}{g(|\nabla\psi|^{2},\psi)}\mbox{ and }\psi>0&\mbox{in }&\Gamma^{+},\\
        &\psi=0&\mbox{on }&\Gamma,\\
        &|\nabla\psi|=\Lambda&\mbox{on }&\{x_{2}=h(x_{1})\}\cap[-R,R]^{2},\\
        &\Psi(x_{2}-\epsilon R)\leq\psi(x)\leq\Psi(x_{2})&\mbox{in }&[-R,R]^{2}.
    \end{aligned}\right.
    \end{equation}
    Let
    \begin{equation*}
        v_{R}(y)=\frac{\Psi^{-1}(\psi(x))-x_{2}}{\epsilon R}\quad\mbox{for }y\in[-1,1]^{2}\cap\{\psi(Ry)>0\},\quad\mbox{where }x=Ry.
    \end{equation*}
    Then, there exists a real number $-K\leq A_{R}\leq0$, such that
    \begin{equation}\label{eq. Schauder mixed error}
        \Big|v_{R}(y)-A_{R}\cdot Re(\sqrt{y_{1}+i\mu y_{2}})\Big|\leq K\cdot|A_{R}|\cdot Re(\sqrt{y_{1}+i\mu y_{2}})\cdot|y|+\delta
    \end{equation}
    in $B_{1/2}\cap\{\psi(Ry)>0\}$. Here, $\mu$ is given explicitly as follows:
    \begin{equation}\label{eq. vertical mach deform}
        \mu=\Big(\frac{g(V^{2},+\infty)}{g(V^{2},+\infty)+2V^{2}\cdot g_{T}(V^{2},+\infty)}\Big)^{1/2}.
    \end{equation}
\end{theorem}
\begin{remark}
    An alternative way to express $Re(\sqrt{y_{1}+i\mu y_{2}})$ is
    \begin{equation*}
        Re(\sqrt{y_{1}+i\mu y_{2}})=\sqrt{\frac{y_{1}+\sqrt{y_{1}^{2}+\mu^{2}y_{2}^{2}}}{2}}.
    \end{equation*}
\end{remark}
\begin{remark}
    In fact, the constant $\mu$ in \eqref{eq. vertical mach deform} is highly related to the mach number. Precisely speaking, $\sqrt{1-\mu^2}$ is equal to the limiting mach number of the flow as $x_2\to+\infty$.
\end{remark}
Notice that when $Ry\in\Gamma$, it must hold that $\Psi^{-1}\circ\psi(Ry)=0$ in the limit sense. As a result, one has $\displaystyle v_{R}(y)=-\frac{x_{2}}{\epsilon R}=-\frac{h(x_{1})}{\epsilon R}$. In other words, the horizontal shifting function $v_{R}(y)$ documents the location of the free boundary. Therefore, we are able to obtain the second main result: the a priori growth rate estimate of the free boundary.
\begin{theorem}\label{thm. growth rate of free boundary}
    Assume that $\psi$ is an asymptotically flat sub-sonic solution without stagnation (see Definition~\ref{def. asymptotically flat}) and assume that $h(x_{1})$ is strictly increasing. Then, one has
    \begin{equation*}
        \lim_{\substack{\frac{1}{2}\leq\frac{s}{t}\leq2\\s,t\to\infty}}\frac{h(s)/h(t)}{\sqrt{s/t}}=1.
    \end{equation*}
    Specifically, for every $\delta>0$, there exists some $\underline{s}(\delta)>0$, such that
    \begin{equation*}
        s^{\frac{1}{2}-\delta}\leq h(s)\leq s^{\frac{1}{2}+\delta}\quad\mbox{for all }s\geq\underline{s}(\delta).
    \end{equation*}
\end{theorem}
To summarize, the main contributions of this paper are threefold. First, we obtain the a priori $\frac{1}{2}\pm\delta$ growth rate estimate for compressible cavity flows, which is almost sharp as a power law. Second, this is the first result on the asymptotic behavior of compressible cavity flows with vorticity; the only previous work on rotational cavity flows is \cite{CDZ17}, which addresses well-posedness in the incompressible case and does not study the growth rate. Third, our estimate is highly localized, requiring neither the global existence of the flow nor the almost-power-growth ansatz used in \cite{L46}.

\subsection{Key ideas}
The proof is based on the ``improvement of flatness'' method of De Silva \cite{D11} (see also Velichkov's textbook \cite{V23}), but with two key differences specific to the quasilinear setting with vorticity.

First, for the quasi-linear problem with vorticity, the original linearization methods is invalid for analysis in large scale. To meet this difficulty, we apply Savin's ODE method \cite{S09} to construct radial barriers by deforming the shear flow solution $\Psi$. Luckily, the structure of quasilinear equation \eqref{eq. quasilinear structure} allows us to do so. Using Savin's ODE method, there is no much distinction between compressible rotational flows and incompressible irrotational flows.

Second, the ``improvement of flatness'' can only be used to prove the regularity of the interface (as in \cite{D11}), or to prove Liouville type results for entire solutions (as in \cite{AS22,S09}). To my knowledge, this is the first paper to use the ``improvement of flatness'' method to predict the asymptotic behavior of an entire solution, instead of proving the Liouville theorem. Since the free boundary is expected to grow like $x_{1}^{\frac{1}{2}\pm\delta}$, the term ``improvement of flatness'' would be misleading; we therefore refer to the method as ``Schauder estimate in large scale''.

Unlike in \cite{D11}, the blow-down limit $v_{\infty}$ of the horizontal shifting function $v_{R}(y)$ does not take the Neumann boundary condition on the whole bottom $\{y_{2}=0\}$. Instead, on the bottom, $v_{\infty}$ satisfies the Dirichlet boundary condition on $[-1,0]\times\{0\}$ and takes the Neumann boundary condition on $[0,1]\times\{0\}$. See details in Lemma~\ref{lem. blow-down discreption mixed}. For such a mixed boundary data \eqref{eq. mixed limit fractional problem}, De Silva-Savin \cite{DS16} has already obtained a boundary Harnack principle, which explicitly describes the asymptotic behavior of the blow-down limit $v_{\infty}$. We also realize that indeed \eqref{eq. mixed limit fractional problem} is closely related to the Caffarelli-Silvestre extension (see \cite{CS07}) of the fractional operator $(-\Delta)^{1/2}$ in $\mathbb{R}^{1}$. Then, the asymptotic estimate in Theorem~\ref{thm. growth rate of free boundary} coincides the fact that positive $\frac{1}{2}$-harmonic functions in $\mathbb{R}_{+}$ with zero exterior data must be a multiple of $(x_{+})^{1/2}$ (see e.g. \cite{ZLCC16}).

\subsection{Organization of this paper}
The proof is divided into three main steps:
\begin{itemize}
    \item In Section~\ref{sec. Basis on steady compressible planar Euler flows}, we derive the quasi-linear equation \eqref{eq. quasilinear structure} and discussed its ellipticity. Besides, we derive the physics equation \eqref{eq. physics structure}, as well as its compatibility conditions.
    \item In Section~\ref{sec. Harnack}, we prove the Harnack principles, see Theorem~\ref{thm. Harnack pure} and Theorem~\ref{thm. Harnack mixed}.
    \item In Section~\ref{sec. Schauder}, we prove the Schauder estimates, including the first main result Theorem~\ref{thm. Schauder mixed}, as well as Theorem~\ref{thm. Schauder pure}.
    \item In Section~\ref{sec. growth rate}, we apply Theorem~\ref{thm. Schauder mixed} to give an a priori estimate of the free boundary $\{x_{2}=h(x_{1})\}$, thus proving the second main result Theorem~\ref{thm. growth rate of free boundary}.
\end{itemize}
To avoid too much tedious computation in the main text, a lot of details are written in the Appendices:
\begin{itemize}
    \item In Appendix~\ref{sec. appendix ODE}, we prove an ODE analysis and deformation akin to Savin \cite{S09}.
    \item In Appendix~\ref{sec. appendix maximum principle}, we prove several maximum principles in the horizontal direction.
    \item In Appendix~\ref{sec. appendix Small-perturbation}, we study the linearized equation of \eqref{eq. quasilinear structure}, obtaining the interior Harnack principle and the asymptotic profile.
\end{itemize}

\section{Basis on steady compressible planar Euler flows}\label{sec. Basis on steady compressible planar Euler flows}
In this section, we give a brief derivation of the stream function representation.
\subsection{Bernoulli Law}
Multiplying the momentum equation in \eqref{eq. initial main} by $\mathbf{u}$ and using the continuity equation yields the Bernoulli Law
\begin{equation}\label{eq. Bernoulli law}
    \mathbf{u}\cdot\nabla H=0,\quad H=\frac{|\mathbf{u}|^{2}}{2}+h(\rho),
\end{equation}
where $h(\rho)$ satisfies $\frac{dh}{d\rho}=\frac{1}{\rho}\frac{dp}{d\rho}$.
\begin{example}
    For an ideal gas, the pressure function $p(\rho)$ is given as $p(\rho)=\rho^{\gamma}/\gamma$ for some $\gamma>1$. In this case, we correspondingly have that $h(\rho)=\frac{\rho^{\gamma-1}}{\gamma-1}$.
\end{example}

\subsection{Vorticity}
Let the vorticity $\omega=\nabla\times\mathbf{u}$ be the curl of $\mathbf{u}$. Since $\mathbf{u}$ is a planar vector, we see that $\omega\parallel\vec{e}_{3}$ where $\vec{e}_{3}$ is perpendicular to the plane. Therefore, one can treat $\omega$ as a scalar, defined as
\begin{equation*}
    \omega=\nabla\times\mathbf{u}:=\partial_{x}u^{2}-\partial_{y}u^{1}.
\end{equation*}
If we take the curl (and treat the curl as a scalar) of the second line of \eqref{eq. initial main}, then
\begin{equation}\label{eq. curl of second line}
    \mathbf{u}\cdot\nabla(\frac{\omega}{\rho})=0.
\end{equation}

\subsection{Stream function representation}
Since the flow is divergence-free, there exists a stream function $\psi$ such that $\rho\mathbf{u}=(\psi_{y},-\psi_{x})$. Then, we have that
\begin{equation*}
    \mathbf{u}=\big(\frac{\psi_{y}}{\rho},-\frac{\psi_{x}}{\rho}\big),\quad\omega=-\nabla\cdot\big(\frac{\nabla\psi}{\rho}\big).
\end{equation*}
Then, we get from \eqref{eq. Bernoulli law} and \eqref{eq. curl of second line} that
\begin{equation*}
    \nabla\psi\parallel\nabla\Big(\frac{|\nabla\psi|^{2}}{2\rho^{2}}+h(\rho)\Big).
\end{equation*}
Consequently, if there is no stagnation, i.e. $|\nabla\psi|\neq0$ everywhere, then
\begin{equation}\label{eq. H depends only on psi}
    H=\frac{|\mathbf{u}|^{2}}{2}+h(\rho)=\frac{|\nabla\psi|^{2}}{2\rho^{2}}+h(\rho)=\mathcal{B}(\psi).
\end{equation}
Reversing this equation means that $\rho$ depends on $(|\nabla\psi|^{2},\psi)$. Let us denote
\begin{equation*}
    g(|\nabla\psi|^{2},\psi)=\frac{1}{\rho}.
\end{equation*}
Set $T=|\nabla\psi|^{2}$ and $z=\psi$, then we have verified \eqref{eq. physics structure}.

\subsection{Quasi-linear equation structure}
Taking the gradient on both sides on \eqref{eq. H depends only on psi} gives
\begin{equation*}
    \mathbf{u}\cdot\nabla\mathbf{u}+\mathbf{u}\times\omega+\nabla h=\mathcal{B}'(\psi)\nabla\psi,
\end{equation*}
where $\mathbf{u}\times\omega$ is understood as $\mathbf{u}\times\omega\vec{e}_{3}=u^{2}\omega\vec{e}_{1}-u^{1}\omega\vec{e}_{2}$.
Since $\nabla h=h'(\rho)\nabla\rho=\frac{1}{\rho}\nabla p$, we get $\mathbf{u}\times(\omega\vec{e}_{3})=\mathcal{B}'(\psi)\nabla\psi$. In other words, using $\mathbf{u}=-\frac{1}{\rho}\nabla^{\perp}\psi$, one has
\begin{equation*}
    -\frac{\omega}{\rho}\nabla\psi=\mathcal{B}'(\psi)\nabla\psi.
\end{equation*}
If there is no stagnation, i.e. $|\nabla\psi|\neq0$ everywhere, then
\begin{equation*}
    -\frac{\omega}{\rho}=\mathcal{B}'(\psi),\quad\mbox{i.e. }\nabla\cdot\Big(g(|\nabla\psi|^{2},\psi)\nabla\psi\Big)=\frac{\mathcal{B}'(\psi)}{g(|\nabla\psi|^{2},\psi)}.
\end{equation*}
In other words, the vorticity $\omega$ has the following expression:
\begin{equation*}
    \omega=-\frac{\mathcal{B}'(\psi)}{g(|\nabla\psi|^{2},\psi)}.
\end{equation*}
Till now, we have seen that if $|\nabla\psi|\neq0$ (i.e. no stagnation), then \eqref{eq. initial main} implies \eqref{eq. physics structure} and \eqref{eq. quasilinear structure}. Then, we see that the stream function $\psi$ satisfies a quasi-linear equation.

\subsection{Compatibility conditions}
Let us derive the compatibility conditions of \eqref{eq. physics structure}. Taking the $T$-derivative of \eqref{eq. physics structure} yields
\begin{equation}\label{eq. compatibility T}
    \frac{g^{2}}{2}+T\cdot g g_{T}-h'(\frac{1}{g})\frac{g_{T}}{g^{2}}=0.
\end{equation}
Taking the $z$-derivative of \eqref{eq. physics structure} yields
\begin{equation*}
    T\cdot g g_{z}-h'(\frac{1}{g})\frac{g_{z}}{g^{2}}=\mathcal{B}'(z).
\end{equation*}
Applying \eqref{eq. compatibility T} gives the following further simplification:
\begin{equation}\label{eq. compatibility Z}
    g_{z}=-\frac{2g_{T}}{g^{2}}\cdot\mathcal{B}'(z).
\end{equation}
Taking the mixed order derivative of \eqref{eq. physics structure} gives that
\begin{equation*}
    g g_{z}+T\cdot\frac{\partial(g g_{z})}{\partial T}+h''(\frac{1}{g})\frac{g_{z}g_{T}}{g^{4}}-h'(\frac{1}{g})\frac{(gg_{z})_{T}-3g_{z}g_{T}}{g^{3}}=0.
\end{equation*}
Using \eqref{eq. compatibility T} and reordering the above identity gives that
\begin{equation*}
    \frac{g}{2g_{T}}\cdot\frac{\partial(g g_{z})}{\partial T}=\Big(\frac{1}{g^{3}}h'(\frac{1}{g})-T\Big)\cdot\frac{\partial(g g_{z})}{\partial T}=g g_{z}+h''(\frac{1}{g})\frac{g_{z}g_{T}}{g^{4}}+3h'(\frac{1}{g})\frac{g_{z}g_{T}}{g^{3}}.
\end{equation*}
Substituting the term $g_{z}$ by \eqref{eq. compatibility Z} gives that
\begin{equation}\label{eq. compatibility TZ}
    \frac{\partial(g g_{z})}{\partial T}=-\Big(4\frac{(g_{T})^{2}}{g^{2}}+4h''(\frac{1}{g})\cdot\frac{(g_{T})^{3}}{g^{7}}+12h'(\frac{1}{g})\cdot\frac{(g_{T})^{3}}{g^{6}}\Big)\cdot\mathcal{B}'(z).
\end{equation}

\subsection{Sub-sonic flow and ellipticity}\label{subset. sub-sonic}
Related to the fluid and the function $g(T,z)$ there is a constant called $Q_{sound}$, such that
\begin{equation*}
    \frac{\partial}{\partial T}\Big(T\cdot g(T,z)^{2}\Big)>0,\quad\mbox{for }(T,z)\in[0,Q_{sound})\times\mathbb{R}_{+}.
\end{equation*}
For incompressible fluids, then one has $Q_{sound}=+\infty$.

We say that the flow $(\rho,\mathbf{u})$ is sub-sonic near a point $x$, if the stream function $\psi$ satisfies
\begin{equation*}
    |\nabla\psi|<Q_{sound}\quad\mbox{near }x.
\end{equation*}
In this case, the linearized equation of \eqref{eq. quasilinear structure} is elliptic near $x$. Moreover, if
\begin{equation*}
    \lambda^{-1}\leq\frac{\partial}{\partial T}\Big(T\cdot g(T,z)^{2}\Big)\leq\lambda,\quad\mbox{for all }(T,z)\in Range\Big(|\nabla\psi|^{2},\psi\Big),
\end{equation*}
then the linearized equation of \eqref{eq. quasilinear structure} is uniformly elliptic.

\section{Harnack principles}\label{sec. Harnack}
In this section, we study the Harnack principle in large scale for the free boundary problem. The following two theorems are for the purely free boundary problem and for the mixed fixed-and-free boundary problem, respectively.
\begin{theorem}\label{thm. Harnack pure}
    There exist two sufficiently small constants $\overline{\epsilon},\sigma>0$ and a sufficiently large constant $\underline{R}$, such that the following fact holds for all $\epsilon\leq\overline{\epsilon}$ and all $R\geq\underline{R}$. Assume that
    \begin{equation*}
        \psi\mbox{ solves \eqref{eq. main simple} and satisfies }\Psi(x_{2}-\epsilon R)\leq\psi(x)\leq\Psi(x_{2}+\epsilon R)\mbox{ in }[-R,R]^{2}.
    \end{equation*}
    Then there exist $a,b\in[-\epsilon R,\epsilon R]$ satisfying $0\leq b-a\leq(1-\sigma)\cdot2\epsilon R$, such that
    \begin{equation*}
        \Psi(x_{2}+a)\leq\psi(x)\leq\Psi(x_{2}+b)\mbox{ in }[-\sigma R,\sigma R]^{2}.
    \end{equation*}
\end{theorem}
\begin{theorem}\label{thm. Harnack mixed}
    There exist two sufficiently small constants $\overline{\epsilon},\sigma>0$ and a sufficiently large constant $\underline{R}$, such that the following fact holds for all $\epsilon\leq\overline{\epsilon}$ and all $R\geq\underline{R}$ with $\epsilon R\geq2\cdot N(0)=2\cdot h(0)$: Assume that $\psi$ satisfies \eqref{eq. cavity local problem}, then, one has
    \begin{equation*}
        \Psi\big(x_{2}-(1-\sigma)\epsilon R\big)\leq\psi(x)\leq\Psi(x_{2})\quad\mbox{in }[-\sigma R,\sigma R]^{2}.
    \end{equation*}
\end{theorem}

\subsection{Harnack principle for the purely free boundary problem}
\begin{proof}[Proof of Theorem~\ref{thm. Harnack pure}]
    Let $\lambda$ be the constants produced in Lemma~\ref{lem. maximum principle for Harnack}.
    
    \textbf{Step 1: Interior Harnack principle.} We consider two cases:
    \begin{equation}\label{eq. alternative of Harnack}
        \mbox{Either }\psi(\frac{9}{10}R\vec{e}_{2})\geq\Psi(\frac{9}{10}R\vec{e}_{2}),\quad\mbox{or }\psi(\frac{9}{10}R\vec{e}_{2})\leq\Psi(\frac{9}{10}R\vec{e}_{2}).
    \end{equation}
    If the first case of \eqref{eq. alternative of Harnack} holds, then let us apply Lemma~\ref{lem. interior Harnack} (with $h=\frac{\lambda}{20}$) to the difference
    \begin{equation*}
        v_{R}(y)=\frac{\Psi^{-1}\circ\psi(Ry)-y_{2}+\epsilon R}{\epsilon R},\quad y\in[-1,1]\times[\frac{\lambda}{20},1].
    \end{equation*}
    As a result, for sufficiently small $\epsilon\leq\overline{\epsilon}$, there exists a constant $\kappa(\lambda)$, such that
    \begin{equation*}
        v_{R}(y)\geq\kappa(\lambda)\cdot v_{R}(\frac{9}{10}\vec{e}_{2})\geq\kappa(\lambda)\mbox{ in }[-\frac{9}{10},\frac{9}{10}]\times[\frac{\lambda}{10},\frac{9}{10}].
    \end{equation*}
    In other words,
    \begin{equation}\label{eq. pure alternative case 1, also useful in mixed}
        \psi(x)\geq\Psi\Big(x_{2}-\big(1-\kappa(\lambda)\big)\epsilon R\Big)\mbox{ in }\Big[-\frac{9R}{10},\frac{9R}{10}\Big]\times\Big[\frac{\lambda R}{10},\frac{9R}{10}\Big].
    \end{equation}

    Similarly, if the second case of \eqref{eq. alternative of Harnack} holds, then
    \begin{equation*}
        \psi(x)\leq\Psi\Big(x_{2}+\big(1-\kappa(\lambda)\big)\epsilon R\Big)\mbox{ in }\Big[-\frac{9R}{10},\frac{9R}{10}\Big]\times\Big[\frac{\lambda R}{10},\frac{9R}{10}\Big].
    \end{equation*}

    \textbf{Step 2: Improvement of the free boundary.} Let $\overline{\epsilon}(\cdot)$ be the function produced in Corollary~\ref{cor. application A of maximum principle}. Now we choose
    \begin{equation*}
        \overline{\epsilon}:=\overline{\epsilon}\big(\kappa(\lambda)\big),\quad\sigma=\min\Big\{\frac{\lambda}{4},\frac{\kappa(\lambda)}{128}\Big\}.
    \end{equation*}
    If $\epsilon\leq\overline{\epsilon}$, then applying Corollary~\ref{cor. application A of maximum principle} to $\psi$ gives the following estimates:
    \begin{itemize}
        \item If the first case of \eqref{eq. alternative of Harnack} holds, then
        \begin{equation*}
            \psi(x)\geq\Psi\Big(x_{2}-\big(1-\frac{\kappa(\lambda)}{64}\big)\epsilon R\Big)\mbox{ in }\big[-\frac{\lambda}{4}R,\frac{\lambda}{4}R\big]^{2};
        \end{equation*}
        \item If the second case of \eqref{eq. alternative of Harnack} holds, then
        \begin{equation*}
            \psi(x)\leq\Psi\Big(x_{2}+\big(1-\frac{\kappa(\lambda)}{64}\big)\epsilon R\Big)\mbox{ in }\big[-\frac{\lambda}{4}R,\frac{\lambda}{4}R\big]^{2}.
        \end{equation*}
    \end{itemize}
    Then we have finished the proof of Theorem~\ref{thm. Harnack pure}. In fact, by choosing
    \begin{equation*}
        (a,b)=\left\{\begin{aligned}
            &\Big(-\epsilon R,\big(1-\frac{\kappa(\lambda)}{64}\big)\epsilon R\Big),&\mbox{if the first case of \eqref{eq. alternative of Harnack} holds},\\
            &\Big(-\big(1-\frac{\kappa(\lambda)}{64}\big)\epsilon R,\epsilon R\Big),&\mbox{if the second case of \eqref{eq. alternative of Harnack} holds},
        \end{aligned}\right.
    \end{equation*}
    all requirements of Theorem~\ref{thm. Harnack pure} are satisfied.
\end{proof}

\subsection{Harnack principle for the cavity problem}
\begin{proof}[Proof of Theorem~\ref{thm. Harnack mixed}]
    Similar to the proof of Theorem~\ref{thm. Harnack pure}, once \eqref{eq. pure alternative case 1, also useful in mixed} is verified, applying Corollary~\ref{cor. application C of maximum principle} (instead of Corollary~\ref{cor. application A of maximum principle}) gives the desired estimate. In fact, since
    \begin{equation*}
        \Big([-R,R]\times\big[\frac{\epsilon R}{2},R\big]\Big)\cap\partial\{\psi>0\}\subseteq\{x_{2}\geq h(0)\}\cap\partial\{\psi>0\}\subseteq\{x_{2}=h(x_{1})\},
    \end{equation*}
    we see that $|\nabla\psi|=\Lambda$ on $\Big([-R,R]\times\big[\frac{\epsilon R}{2},R\big]\Big)\cap\partial\{\psi>0\}$, hence the assumptions of Corollary~\ref{cor. application C of maximum principle} are all verified.

    It remains the prove \eqref{eq. pure alternative case 1, also useful in mixed} in the settings of Theorem~\ref{thm. Harnack mixed}. To do this, we first apply Lemma~\ref{lem. maximum principle for forced decay} in the region $[-\frac{3}{4}R,-\frac{1}{4}R]\times[0,\frac{1}{2}R]$, where $\psi$ satisfies \eqref{eq. quasilinear structure} without free boundary points. This implies that
    \begin{equation}\label{eq. force decay 1 time}
        \psi(x)\geq\Psi\big(x_{2}-(1-\sigma)\epsilon R\big)\quad\mbox{in }-\frac{R}{2}\vec{e}_{1}+[-\sigma R,\sigma R]^{2}.
    \end{equation}
    Consider the function
    \begin{equation*}
        v_{R}(y)=\frac{\Psi^{-1}(\psi(x))-x_{2}}{\epsilon R},\quad\mbox{where }x=Ry.
    \end{equation*}
    Clearly, $-1\leq v_{R}(y)\leq0$ in $[-1,1]^{2}\cap\{\psi(Ry)>0\}$. Moreover, \eqref{eq. force decay 1 time} implies that
    \begin{equation*}
        v_{R}(-\frac{1}{2}\vec{e}_{1}+\sigma\vec{e}_{1})+1\geq\sigma>0.
    \end{equation*}
    Applying the interior Harnack principle Lemma~\ref{lem. interior Harnack} to $v_{R}+1$ first in $[-\frac{3}{4},-\frac{1}{4}]\times[0,\frac{1}{2}]$ and then in $[-1,1]\times[\frac{\lambda}{20},1]$ implies that
    \begin{equation*}
        \inf_{[-\frac{9}{10}R,\frac{9}{10}R]\times[\frac{\lambda}{10}R,\frac{9}{10}R]}(v_{R}+1)\geq c(\lambda,\sigma)\cdot\Big(v_{R}(-\frac{1}{2}\vec{e}_{1}+\sigma\vec{e}_{1})+1\Big)\geq c'(\lambda,\sigma).
    \end{equation*}
    Since $\sigma$ is a uniform constant, one can write $\kappa(\lambda)=c'(\lambda,\sigma)$, then \eqref{eq. pure alternative case 1, also useful in mixed} is verified by the definition of $v_{R}$. This completes the proof of Theorem~\ref{thm. Harnack mixed}.
\end{proof}
\begin{remark}\label{rmk. how to avoid h(infty)<infty}
    Even if the assumption $\epsilon R\geq2\cdot N(0)=2\cdot h(0)$ in Theorem~\ref{thm. Harnack mixed} is removed, the estimate \eqref{eq. pure alternative case 1, also useful in mixed} is still valid. Then, by applying Theorem~\ref{thm. Harnack pure} in the first quadrant, where $\partial\{\psi>0\}$ consists solely of the free boundary (i.e. $|\nabla\psi|=\Lambda$), we still have
    \begin{equation*}
        \psi(x)\geq\Psi\Big(x_{2}-(1-\sigma)\epsilon R\Big)\quad\mbox{in }[\frac{1}{2}-\frac{\sigma}{2}R,\frac{1}{2}+\frac{\sigma}{2}R]\times[0,\sigma R].
    \end{equation*}
    This does not improve the lower bound of $\psi$ near the origin, but it can at least lower the height of the free boundary in the first quadrant far from the origin.
\end{remark}

\subsection{H\"older regularity of the level sets}
As a consequence of Theorem~\ref{thm. Harnack pure}, we have the following H\"older estimate:
\begin{corollary}\label{cor. pure boundary Holder}
    There exist uniform constants $\alpha,C>0$ and a function $d(\epsilon,R)>0$ with $\displaystyle\lim_{\substack{\epsilon\to0\\R\to\infty}}d(\epsilon,R)=0$, such that the following statement holds: Let $\epsilon\leq\frac{1}{2}$ and let $R>0$ be arbitrary. Assume that
    \begin{equation*}
        \psi(x)\mbox{ solves \eqref{eq. main simple} and satisfies }\Psi(x_{2}-\epsilon R)\leq\psi(x)\leq\Psi(x_{2}+\epsilon R)\mbox{ in }[-R,R]^{2}.
    \end{equation*}
    Denote $v:[-1,1]^{2}\cap\{\psi(Ry)>0\}\to\mathbb{R}$ as
    \begin{equation*}
        v(y)=\frac{\Psi^{-1}(\psi(x))-x_{2}}{\epsilon R},\quad\mbox{where }x=Ry.
    \end{equation*}
    Then, one has
    \begin{equation*}
        [-1,1]\times[\epsilon,1]\subseteq[-1,1]^{2}\cap\{\psi(Ry)>0\}\subseteq[-1,1]\times[-\epsilon,1].
    \end{equation*}
    Moreover, for all $y^{1},y^{2}\in\big[-\frac{1}{2},\frac{1}{2}\big]^{2}\cap\{\psi(Ry)>0\}$, one has
    \begin{equation*}
        \big|v(y^{1})-v(y^{2})\big|\leq C\cdot\big(|y^{1}-y^{2}|^{\alpha}+d(\epsilon,R)^{\alpha}\big).
    \end{equation*}
\end{corollary}
\begin{proof}
    The domain inclusion fact is obvious, so we mainly focus on proving the H\"older regularity. In fact, let $d$ be arbitrary, we would like to show that for sufficiently small $\epsilon$ and large $R$ (depending on $d$), such that
    \begin{equation}\label{eq. another way of cor 4.1 estimate}
        \big|v(y^{1})-v(y^{2})\big|\leq C\cdot\big(|y^{1}-y^{2}|^{\alpha}+d^{\alpha}\big).
    \end{equation}
    The proof of \eqref{eq. another way of cor 4.1 estimate} is divided into two main steps:
    
    \textbf{Step 1: Oscillation reduction near the free boundary.} Let $\overline{\epsilon},\sigma>0$ be the uniform constants produced by Theorem~\ref{thm. Harnack pure}. Let
    \begin{equation*}
        R_{k}=\sigma^{k}\frac{R}{2},\quad0\leq k\leq K,\quad\mbox{where }K:=\Big\lfloor\log_{\sigma}(\frac{\epsilon}{2\overline{\epsilon}})\Big\rfloor.
    \end{equation*}
    If $\epsilon$ is sufficiently small, then $K$ is sufficiently large. Let
    \begin{equation*}
        R_{k}=\sigma^{k}\frac{R}{2}\mbox{ and }\epsilon_{k}=\big(\frac{1-\sigma}{\sigma}\big)^{k}\cdot2\epsilon,\quad\mbox{for }k\in\{0,\cdots,K\}.
    \end{equation*}
    Let $x_{1}\in\big[-\frac{R}{2},\frac{R}{2}\big]$ be fixed and let $\widetilde{\psi}(z):=\psi(z+x_{1}\vec{e}_{1})$. We then notice that
    \begin{equation*}
        \widetilde{\psi}(z)\mbox{ satisfies \eqref{eq. main simple} and }\Psi(z_{2}-\epsilon_{0}R_{0})\leq\widetilde{\psi}(z)\leq\Psi(z_{2}+\epsilon_{0}R_{0})\mbox{ in }[-R_{0},R_{0}]^{2}.
    \end{equation*}
    By inductively applying Theorem~\ref{thm. Harnack pure} $K$ times, we see that there exists a nested sequence of intervals $\{(a_{k},b_{k})\}_{k=0}^{K+1}$, i.e. $a_{k}\leq a_{k+1}\leq b_{k+1}\leq b_{k}$ for $0\leq k\leq K$, with $(a_{0},b_{0})=(-\epsilon_{0}R_{0},\epsilon_{0}R_{0})$, such that
    \begin{equation*}
        \frac{b_{k+1}-a_{k+1}}{b_{k}-a_{k}}\leq(1-\sigma)\mbox{ for }0\leq k\leq K,\quad\mbox{i.e. }b_{k}-a_{k}\leq\epsilon_{k}R_{k}\mbox{ for }0\leq k\leq K+1,
    \end{equation*}
    and that
    \begin{equation*}
        \Psi(z_{2}+a_{k})\leq\widetilde{\psi}(z)\leq\Psi(z_{2}+b_{k})\mbox{ in }[-R_{k},R_{k}]^{2},\quad\mbox{for }0\leq k\leq K+1.
    \end{equation*}
    This estimate, when expressed in terms of the function $v(y)$, means that
    \begin{equation*}
        \frac{a_{k}}{\epsilon R}\leq v(y+\frac{x_{1}}{R}\vec{e}_{1})\leq\frac{b_{k}}{\epsilon R},\quad\mbox{for }y\in\big[-\frac{\sigma^{k}}{2},\frac{\sigma^{k}}{2}\big]^{2}\cap\{\psi(Ry)>0\}\mbox{ with }0\leq k\leq K+1.
    \end{equation*}
    By letting $\alpha_{1}:=\log_{\sigma}(1-\sigma)>0$, one has
    \begin{equation*}
        \mathop{osc}_{[y_{1}-r,y_{1}+r]\times[-r,r]\cap\{\psi(Ry)>0\}}v\leq C\cdot r^{\alpha_{1}},\quad\mbox{for all }\frac{\sigma^{K+1}}{2}\leq r\leq\frac{1}{2}\mbox{ and }|y_{1}|\leq\frac{1}{2}.
    \end{equation*}
    For $\epsilon$ small enough so that $\frac{\sigma^{K+1}}{2}\leq d$, the above estimate implies In other words, 
    \begin{equation*}
        \mathop{osc}_{y_{1}\vec{e}_{1}+[-r,r]^{2}\cap\{\psi(Ry)>0\}}v\leq C\big(r^{\alpha_{1}}+d^{\alpha_{1}}\big),\quad\mbox{for all }r\leq\frac{1}{2}\mbox{ and }|y_{1}|\leq\frac{1}{2}.
    \end{equation*}

    \textbf{Step 2: Interior H\"older regularity.} Now, choose $h=d$ in Lemma~\ref{lem. interior Harnack}. Then, for $R$ large and $\epsilon$ small, the Harnack principle (and thus the H\"older estimate) in Lemma~\ref{lem. interior Harnack} holds in $\frac{1}{r}\Omega=[-1,1]\times[d,1]$.
    
    Let $y^{1},y^{2}\in\big[-\frac{1}{2},\frac{1}{2}\big]^{2}\cap\{\psi(Ry)>0\}$, and denote $\delta=|y^{1}-y^{2}|$. Without loss of generality, we assume that $\delta\geq d$. If not, then using an intermediate point $y^{3}$ such that $|y^{1}-y^{3}|\sim|y^{2}-y^{3}|\sim d$ gives the difference between $v(y^{1})$ and $v(y^{2})$.
    \begin{itemize}
        \item Case 1: If $dist(y^{1},y^{3})\leq3\delta$ for some $y^{3}\in[-\frac{1}{2},\frac{1}{2}\big]^{2}\cap\partial\{\psi(Ry)>0\}$, then we have
        \begin{equation*}
            dist(y^{2},y^{3})\leq4\delta.
        \end{equation*}
        Using the oscillation reduction in \textbf{Step 1}, we see that
        \begin{equation*}
            |v(y^{1})-v(y^{2})|\leq|v(y^{1})-v(y^{3})|+|v(y^{2})-v(y^{3})|\leq C\big(\delta^{\alpha_{1}}+d^{\alpha_{1}}\big).
        \end{equation*}
        Similarly, if $dist(y^{2},y^{3})\leq3\delta$, then the same argument holds.
        \item Case 2: If it holds that
        \begin{equation*}
            dist(y^{1},\partial\{\psi(Ry)>0\}),dist(y^{2},\partial\{\psi(Ry)>0\})\geq3\delta,
        \end{equation*}
        then let us assume that
        \begin{equation*}
            dist(y^{1},\partial\{\psi(Ry)>0\})=dist(y^{1},y^{3})=r
        \end{equation*}
        for some $y^{3}\in[-\frac{1}{2},\frac{1}{2}\big]^{2}\cap\partial\{\psi(Ry)>0\}$. Clearly, one has $r\geq3\delta$, and
        \begin{equation*}
            \frac{r}{2}\leq dist(y^{2},y^{3})\leq2r,\quad1\leq\frac{dist(y^{2},y^{3})}{dist(y^{2},\partial\{\psi(Ry)>0\})}\leq2.
        \end{equation*}
        By \textbf{Step 1}, we see that
        \begin{equation*}
            \mathop{osc}_{y^{3}+[-4r,4r]^{2}\cap\{\psi(Ry)>0\}}v\leq C\big(r^{\alpha_{1}}+d^{\alpha_{1}}\big).
        \end{equation*}
        Moreover, it follows that
        \begin{equation*}
            dist\Big(y^{i},\partial(y^{3}+[-4r,4r]^{2}\cap\{\psi(Ry)>0\})\Big)\geq r,\quad i=1,2.
        \end{equation*}
        Therefore, using the interior H\"older estimate (following from the interior Harnack principle Lemma~\ref{lem. interior Harnack}), we see that
        \begin{equation*}
            |v(y^{1})-v(y^{2})|\leq C\cdot(\frac{\delta}{r})^{\alpha_{2}}\cdot\mathop{osc}_{y^{3}+[-4r,4r]^{2}\cap\{\psi(Ry)>0\}}v\leq C(\delta^{\alpha}+d^{\alpha}).
        \end{equation*}
        Here, $\alpha$ is chosen as $\alpha=\frac{1}{2}\min\{\alpha_{1},\alpha_{2}\}$.
    \end{itemize}
    Then, we have finished the proof of \eqref{eq. another way of cor 4.1 estimate} when $R\gg1\gg\epsilon$, hence the proof of Corollary~\ref{cor. pure boundary Holder} is completed.
\end{proof}

\begin{corollary}\label{cor. mixed boundary Holder}
    There exist uniform constants $\alpha,C>0$ and a function $d(\epsilon,R)>0$ with
    \begin{equation*}
        \lim_{\substack{\epsilon\to0\\\epsilon R\to\infty}}d(\epsilon,R)=0,
    \end{equation*}
    such that the following statement holds: Let $\epsilon\leq\frac{1}{2}$ and let $R>0$ be arbitrary. Assume that
    \begin{equation*}
        \psi(x)\mbox{ solves \eqref{eq. cavity local problem} and satisfies }\Psi(x_{2}-\epsilon R)\leq\psi(x)\leq\Psi(x_{2})\mbox{ in }[-R,R]^{2}.
    \end{equation*}
    Denote $v:[-1,1]^{2}\cap\{\psi(Ry)>0\}\to\mathbb{R}$ as
    \begin{equation*}
        v(y)=\frac{\Psi^{-1}(\psi(x))-x_{2}}{\epsilon R},\quad\mbox{where }x=Ry.
    \end{equation*}
    Then, one has
    \begin{equation*}
        [-1,1]\times[\epsilon,1]\subseteq[-1,1]^{2}\cap\{\psi(Ry)>0\}\subseteq[-1,1]\times[-\epsilon,1].
    \end{equation*}
    Moreover, for all $y^{1},y^{2}\in\big[-\frac{1}{2},\frac{1}{2}\big]^{2}\cap\{\psi(Ry)>0\}$, one has
    \begin{equation*}
        \big|v(y^{1})-v(y^{2})\big|\leq C\cdot\big(|y^{1}-y^{2}|^{\alpha}+d(\epsilon,R)^{\alpha}\big).
    \end{equation*}
    Furthermore, for all $y\in\big[-\frac{1}{2},\frac{1}{2}\big]^{2}\cap\{\psi(Ry)>0\}$, it holds that
    \begin{equation}\label{eq. uniform vanishing on -x axis}
        |v(y)|\leq C\cdot\big(dist(y,\mathbb{R}_{-}\times\{0\})^{\alpha}+d(\epsilon,R)^{\alpha}\big).
    \end{equation}
\end{corollary}
\begin{proof}
    Fix an arbitrary and small $d>0$. Similar to \textbf{Step 1} in Corollary~\ref{cor. pure boundary Holder}, we can apply Theorem~\ref{thm. Harnack mixed} and get that for $\epsilon\ll1\ll R$,
    \begin{equation*}
        -C\big(r^{\alpha_{1}}+d^{\alpha_{1}}\big)\leq v\leq0,\quad\mbox{for all }y\in[-r,r]^{2}\cap\{\psi(Ry)>0\}\mbox{ with }r\leq\frac{1}{2},
    \end{equation*}
    Then, in each $\mathcal{R}_{r}=\big([-4r,4r]^{2}\setminus[-r/2,r/2]^{2})\cap\{\psi(Ry)>0\}$ with $d\leq r\leq\frac{1}{8}$, we have
    \begin{equation*}
        osc_{\mathcal{R}_{r}}v\leq C\big(r^{\alpha_{1}}+d^{\alpha_{1}}\big)\quad\mbox{in }\mathcal{R}_{r}.
    \end{equation*}
    
    By applying Corollary~\ref{cor. pure boundary Holder} to $v$ in $\big([0,4r]^{2}\setminus[-r/2,r/2]^{2})\cap\{\psi(Ry)>0\}$, where the zero set of $\psi$ purely consists of free boundary points (i.e. $|\nabla\psi|=\Lambda$), we get that for all $y^{1},y^{2}\in\big([0,2r]^{2}\setminus[-r,r]^{2})\cap\{\psi(Ry)>0\}$,
    \begin{equation*}
        |v(y^{1})-v(y^{2})|\leq C\Big(|y^{1}-y^{2}|^{\alpha_{2}}+d^{\alpha_{2}}\Big),\quad\mbox{as long as }\epsilon\ll1\ll R.
    \end{equation*}
    
    Moreover, since $v=0$ on $[-4r,-r/2]\times\{0\}$, by applying Lemma~\ref{lem. maximum principle for forced decay} in $\big([-4r,0]^{2}\setminus[-r/2,r/2]^{2})$ for $O\big(\min\big\{\ln{(rR)},\ln{(\epsilon/r)}\big\}\big)$ many times, we have
    \begin{equation*}
        \|v\|_{L^{\infty}([-2r,r]\times[0,t]^{2})}\leq C\Big(t^{\alpha_{3}}+d^{\alpha_{3}}\Big).
    \end{equation*}
    Then by applying Lemma~\ref{lem. interior Harnack} afterwards (similar to \textbf{Step 2} in the proof of Corollary~\ref{cor. pure boundary Holder}), we obtain that
    \begin{equation*}
        |v(y^{1})-v(y^{2})|\leq C\Big(|y^{1}-y^{2}|^{\alpha_{4}}+d^{\alpha_{4}}\Big)\mbox{ in }[-2r,0]^{2}\setminus[-r,r]^{2}.
    \end{equation*}
    Combining all estimates obtained above and allowing $r$ to vary finishes the proof of Corollary~\ref{cor. mixed boundary Holder}.
\end{proof}

\section{Schauder estimates}\label{sec. Schauder}
In this section, we study the a priori Schauder estimate in large scale for the free boundary problem. The following theorem and Theorem~\ref{thm. Schauder mixed} are for the purely free boundary problem and for the mixed fixed-and-free boundary problem, respectively.
\begin{theorem}\label{thm. Schauder pure}
    Then there exist five uniform constants $\overline{\epsilon},\underline{R},\rho,\alpha,C>0$, such that the following fact holds for all $\epsilon\leq\overline{\epsilon}$ and all $R\geq\underline{R}$. Assume that
    \begin{equation*}
        \psi\mbox{ solves \eqref{eq. main simple} and satisfies }\Psi(x_{2}-\epsilon R)\leq\psi(x)\leq\Psi(x_{2}+\epsilon R)\quad\mbox{in }B_{R}.
    \end{equation*}
    Then, there exist a unit vector $\vec{\nu}$ and two real numbers $a,b$ satisfying
    \begin{equation}\label{eq. Schauder pure result line 1}
        \angle{(\vec{\nu},\vec{e}_{2})}\leq C\epsilon,\quad\mbox{and }a,b\in[-C\epsilon R,C\epsilon R]\mbox{ with }0\leq b-a\leq\rho^{1+\alpha}\cdot2\epsilon R,
    \end{equation}
    such that
    \begin{equation}\label{eq. Schauder pure result line 2}
        \Psi(x\cdot\vec{\nu}+a)\leq\psi(x)\leq\Psi(x\cdot\vec{\nu}+b)\quad\mbox{in }B_{\rho R}.
    \end{equation}
\end{theorem}
\subsection{Blow-down analysis}
\begin{lemma}\label{lem. limit description}
    Assume that $\epsilon_{k}\to0$ and $R_{k}\to\infty$. Assume that
    \begin{equation*}
        \psi_{k}\mbox{ solves \eqref{eq. main simple} and satisfies }\Psi(x_{2}-\epsilon_{k}R_{k})\leq\psi(x)\leq\Psi(x_{2}+\epsilon_{k}R_{k})\mbox{ in }[-R_{k},R_{k}]^{2}.
    \end{equation*}
    Denote $\Omega_{k}:=[-1,1\big]^{2}\cap\{\psi_{k}(R_{k}y)>0\}$ and denote $v_{k}:\Omega_{k}\to\mathbb{R}$ as
    \begin{equation*}
        v_{k}(y)=\frac{\Psi^{-1}(\psi_{k}(x))-x_{2}}{\epsilon_{k}R_{k}},\quad\mbox{where }x=R_{k}y.
    \end{equation*}
    Then, there exists a subsequence of $\{v_{k}\}$ (still labeled by $k$) and a function $v:[-\frac{1}{2},\frac{1}{2}]\times[0,\frac{1}{2}]\to\mathbb{R}$, such that the followings hold:
    \begin{itemize}
        \item[(1)] For every $k$, one has $[-1,1]\times[\epsilon_{k},1]\subseteq\Omega_{k}\subseteq[-1,1]\times[-\epsilon_{k},1]$.
        \item[(2)] As $k\to\infty$, one has
        \begin{equation}\label{eq. v_k uniformly converges}
            \lim_{k\to\infty}\|v_{k}-v\|_{L^{\infty}([-\frac{1}{2},\frac{1}{2}]\times[0,\frac{1}{2}]\cap\Omega_{k})}=0.
        \end{equation}
        Besides, for all fixed $\epsilon>0$, one has
        \begin{equation*}
            \lim_{k\to\infty}\|v_{k}-v\|_{L^{\infty}([-\frac{1}{2},\frac{1}{2}]\times[\epsilon,\frac{1}{2}])}=0.
        \end{equation*}
        \item[(3)] For $\alpha$ being the constant in Corollary~\ref{cor. pure boundary Holder}, one has $v\in C^{\alpha}([-\frac{1}{2},\frac{1}{2}]\times[0,\frac{1}{2}])$;
    \end{itemize}
    Moreover, in the viscosity sense, one has
    \begin{equation*}
        \mu^{2}v_{11}+v_{22}=0\quad\mbox{in }(-\frac{1}{2},\frac{1}{2})\times(0,\frac{1}{2})\mbox{ for }\mu\mbox{ defined in \eqref{eq. vertical mach deform}},
    \end{equation*}
    and $\partial_{x_{2}}v(t\vec{e}_{1})=0$ for $x\in(-\frac{1}{2},\frac{1}{2})$. More precisely, we mean:
    \begin{itemize}
        \item[(4)] For every $y^{*}\in(-\frac{1}{2},\frac{1}{2})\times(0,\frac{1}{2})$ and every quadratic polynomial $P(y)$, if
        \begin{equation*}
            v(y^{*})-P(y^{*})=\sup_{B_{\rho}(y^{*})}(v-P)=0\quad\mbox{for some }\rho>0,
        \end{equation*}
        then $(\mu^{2}\partial_{11}+\partial_{22})P(y^{*})\geq0$. Conversely, if the ``sup'' is replaced with ``inf'', then $(\mu^{2}\partial_{11}+\partial_{22})P(y^{*})\leq0$.
        \item[(5)] For every $y^{*}\in(-\frac{1}{2},\frac{1}{2})\times\{0\}$ and every quadratic polynomial $P(y)$, if
        \begin{equation*}
            v(y^{*})-P(y^{*})=\sup_{B_{\rho}(y^{*})^{+}}(v-P)=0\quad\mbox{for some }\rho>0,
        \end{equation*}
        then $\partial_{2}P(y^{*})\geq0$. Conversely, if the ``sup'' is replaced with ``inf'', then $\partial_{2}P(y^{*})\leq0$.
    \end{itemize}
\end{lemma}
Since the proof is really involved, it will be broken into three shorter parts. We will prove part (1)-(3) first, and then part (4), and finally prove part (5).
\begin{proof}[Proof of Lemma~\ref{lem. limit description} (1)-(3)]
    These statements follow directly from Corollary~\ref{cor. pure boundary Holder}.
\end{proof}
\begin{proof}[Proof of Lemma~\ref{lem. limit description} (4)]
    Fix an arbitrary point $y^{*}\in(-\frac{1}{2},\frac{1}{2})\times(0,\frac{1}{2})$. Let us show that $(\mu^{2}\partial_{11}+\partial_{22})v(\zeta^{*})\leq0$ in the viscosity sense (the other side can be shown similarly). Suppose that
    \begin{equation*}
        P(y)=(y-y^{*})^{T}M(y-y^{*})+\vec{b}\cdot(y-y^{*})+v(y^{*})
    \end{equation*}
    touches the graph of $v(y)$ from below in $B_{\\rho}(y^{*})$ for some $\rho>0$ and satisfies $(\mu^{2}\partial_{11}+\partial_{22})P>0$. Since $P$ is a strict sub-solution, we can let $0<a\ll1$ and set
    \begin{equation*}
        \widetilde{P}(y)=(y-y^{*})^{T}\widetilde{M}(y-y^{*})+\vec{b}\cdot(y-y^{*})+v(y^{*}),\quad\mbox{where }\widetilde{M}=M-a\delta^{ij}.
    \end{equation*}
    Then, we still have $(\mu^{2}\partial_{11}+\partial_{22})\widetilde{P}>0$. Furthermore,
    \begin{equation}\label{eq. subharmonic test function}
        v-\widetilde{P}\geq(v-\widetilde{P})(y^{*})=0\mbox{ in }B_{\rho}(y^{*}),\quad v-\widetilde{P}\geq\frac{a\rho^{2}}{4}\mbox{ in }B_{\rho}(y^{*})\setminus B_{\rho/2}(y^{*}).
    \end{equation}
    Without loss of generality, we can assume that $\rho$ is so small that
    \begin{equation*}
        B_{\rho}(y^{*})\subseteq(-\frac{1}{2},\frac{1}{2})\times(0,\frac{1}{2})\mbox{ and }B_{\rho}(y^{*})\subseteq\Omega_{k}\quad\mbox{for sufficiently large }k.
    \end{equation*}
    Therefore, it follows from Lemma~\ref{lem. limit description} (1)-(3) and from \eqref{eq. subharmonic test function} that there exist $h_{k}\to0$ and $y^{k}\in B_{\rho/2}(y^{*})$, such that
    \begin{equation*}
        \widetilde{P}_{k}(y)=\widetilde{P}(y)+h_{k}\mbox{ touches }v_{k}\mbox{ from below at }y^{k}.
    \end{equation*}
    Let us rewrite $\widetilde{P}_{k}(y)$ such that it centers at $\zeta^{k}$, i.e.:
    \begin{equation*}
        \widetilde{P}_{k}(y)=(y-y^{k})^{T}\widetilde{M}(y-y^{k})+\vec{b}^{k}\cdot(y-y^{k})+v_{k}(y^{k}),\quad\mbox{where }\vec{b}^{k}=\vec{b}+2M(y^{k}-y^{*}).
    \end{equation*}
    Since $y^{*}_{2}>0$, we see that $y^{k}_{2}\geq\frac{1}{2}y^{*}_{2}>0$ for large $k$. By Remark~\ref{rmk. limiting mach number equation} and Lemma~\ref{lem. when does b converge to zero?}, we see that for sufficiently large $k$,
    \begin{equation*}
        (\mu^{2}\partial_{11}+\partial_{22})\widetilde{P}_{k}(y^{k})+o(1)\leq0,
    \end{equation*}
    where $o(1)\to1$ as $R_{k}\to\infty$ as $\epsilon_{k}\to0$. This results in a contradiction.
\end{proof}
\begin{proof}[Proof of Lemma~\ref{lem. limit description} (5)]
    Let us show that $\partial_{2}v\leq0$ on $(-\frac{1}{2},\frac{1}{2})\times\{0\}$ in the viscosity sense. The other direction that $\partial_{2}v\geq0$ on $(-\frac{1}{2},\frac{1}{2})\times\{0\}$ can be argued similarly.
    
    \textbf{Step 1: Fixing a test function.} Suppose that the condition $\partial_{y_{2}}v\leq0$ fails at $y^{*}\in(-\frac{1}{2},\frac{1}{2})\times\{0\}$, then there exists a quadratic test function $Q$ such that
    \begin{equation*}
        v(y^{*})-Q(y^{*})=\inf_{B_{\rho_{0}}(y^{*})^{+}}(v-Q)=0\mbox{ for some }\rho_{0}>0,\quad\mbox{and }\partial_{2}Q(y^{*})>0.
    \end{equation*}
    Since subtracting $Q$ with $(y-y^{*})^{T}A(y-y^{*})$ for a positive-definite matrix $A$ preserves $\partial_{2}Q(y^{*})$, and the new quadratic function still touches the graph of $v$ from below at $y^{*}$. Therefore, let us assume for simplicity that the quadratic test function takes the form
    \begin{equation*}
        Q(y)=-2M|y-y^{*}|^{2}+2\vec{b}\cdot(y-y^{*})+v(y^{*})\quad\mbox{with }\min\{M,\vec{b}\cdot\vec{e}_{2}\}\geq a>0.
    \end{equation*}
    Now, let $\lambda$ and $\delta_{0}(\beta)$ be the uniform constant and function produced by Corollary~\ref{cor. application B of maximum principle}. Then, let us set
    \begin{equation*}
        \rho=\min\Big\{\frac{1}{2}-|y^{*}|,\frac{\lambda\cdot\langle\vec{b},\vec{e}_{2}\rangle}{2048M},\rho_{0}\Big\},\quad\beta=\frac{\langle\vec{b},\vec{e}_{2}\rangle}{2}\quad\overline{\epsilon}=\delta_{0}(\beta).
    \end{equation*}
    Set
    \begin{equation*}
        \widetilde{Q}(y)=-M|y-y^{*}|^{2}+\vec{b}\cdot(y-y^{*})+v(y^{*}),
    \end{equation*}
    then for every $t\in(0,\rho)$ one has
    \begin{equation}\label{eq. large derivative test function}
        v-\widetilde{Q}\geq(v-\widetilde{Q})(y^{*})=0\mbox{ in }B_{\rho}(y^{*})^{+},\quad\mbox{and }v-\widetilde{Q}\geq a\cdot t^{2}\mbox{ in }B_{\rho}(y^{*})^{+}\setminus B_{t}(y^{*})^{+}.
    \end{equation}
    
    \textbf{Step 2: Analysis of $v_{k}$.} By \eqref{eq. v_k uniformly converges}, one has
    \begin{equation*}
        v_{k}-\widetilde{Q}\geq-o(1)\mbox{ in }B_{\rho}(\zeta^{*})\cap\Omega_{k},
    \end{equation*}
    and
    \begin{equation*}
        v_{k}-\widetilde{Q}\geq a\cdot t^{2}-o(1)\mbox{ in }\big(B_{\rho}(y^{*})\setminus B_{t}(y^{*})\big)\cap\Omega_{k}.
    \end{equation*}
    Since one can let $\rho$ arbitrarily small, we see that as $k\to\infty$, there exist two sequences $\{y^{(*,k)}\}\subseteq\mathbb{R}^{2}$ and $\{h_{k}\}\subseteq\mathbb{R}$, such that
    \begin{equation}\label{eq. sequence of contact points}
        y^{(*,k)}\in\overline{B_{\rho}(y^{*})\cap\{\psi_{k}(R_{k}y)>0\}},\quad y^{(*,k)}\to y^{*},\quad h_{k}\to v(y^{*}),
    \end{equation}
    and that
    \begin{equation}\label{eq. v_k minimal assumption}
        \liminf_{\substack{y\in B_{\rho}(y^{*})\cap\Omega_{k}\\y\to y^{(*,k)}}}(v_{k}-\widetilde{Q}_{k})(y)=\inf_{B_{\rho}(y^{*})\cap\Omega_{k}}(v_{k}-\widetilde{Q}_{k})=0,
    \end{equation}
    where
    \begin{equation}\label{eq. widetilde Q_k form}
        \widetilde{Q}_{k}(y)=-M|y-y^{*}|^{2}+\vec{b}\cdot(y-y^{*})+h_{k}.
    \end{equation}

    \textbf{Step 3: Analysis of $\psi_{k}$.} From \eqref{eq. v_k minimal assumption} and by the continuity of $v_{k}$ in $\Omega_{k}$, we actually have that
    \begin{equation*}
        \liminf_{\substack{\eta\in B_{\rho}(y^{*})\cap\Omega_{k}\\\eta\to y}}v_{k}(\eta)\geq\widetilde{Q}_{k}(y)\quad\mbox{for all }y\in\overline{B_{\rho}(y^{*})\cap\Omega_{k}}.
    \end{equation*} 
    Since
    \begin{equation*}
        \psi_{k}(R_{k}y)=0\quad\mbox{on }[-\frac{1}{2},\frac{1}{2}]^{2}\cap\partial\{\psi_{k}(r_{k}\zeta)>0\}=[-\frac{1}{2},\frac{1}{2}]^{2}\cap\partial\Omega_{k},
    \end{equation*}
    we then infer from the definition of $v_{k}$ that
    \begin{equation}\label{eq. Q_k behavior on free boundary}
        \widetilde{Q}_{k}(y)\leq\liminf_{\substack{\eta\in B_{\rho}(y^{*})\cap\Omega_{k}\\\eta\to y}}v_{k}(\eta)=-\frac{y_{2}}{\epsilon_{k}}\mbox{ on }\big[-\frac{1}{2},\frac{1}{2}\big]^{2}\cap\partial\Omega_{k}.
    \end{equation}
    Besides, it follows from \eqref{eq. v_k minimal assumption} that
    \begin{equation*}
        \psi_{k}(x)\geq\Psi\Big(x_{2}+\widetilde{Q}_{k}(\frac{x}{R_{k}})\epsilon_{k}R_{k}\Big)\mbox{ in }B_{\rho R_{k}}(R_{k}y^{*})\cap\{\psi_{k}(x)>0\}.
    \end{equation*}
    We claim that indeed
    \begin{equation}\label{eq. claim lower bound extends to contact region}
        \psi_{k}(x)\geq\Psi\Big(x_{2}+\widetilde{Q}_{k}(\frac{x}{R_{k}})\epsilon_{k}R_{k}\Big)\mbox{ in }B_{\rho_{0}R_{k}}(R_{k}y^{*}).
    \end{equation}

    In fact, suppose that \eqref{eq. claim lower bound extends to contact region} is false. Then there exists $z\in B_{\rho R_{k}}(R_{k}y^{*})$, such that
    \begin{equation}\label{eq. if (4.6) is wrong}
        \psi_{k}(z)=0,\quad z_{2}+\widetilde{Q}_{k}(\frac{z}{R_{k}})\epsilon_{k}R_{k}>0.
    \end{equation}
    As we are assuming $\psi(x)\geq\Psi(x_{2}-\epsilon_{k}R_{k})$ in $[-R_{k},R_{k}]^{2}$, we must have
    \begin{equation*}
        z\in B_{\rho R_{k}}(R_{k}y^{*})\cap\{z_{2}\leq\epsilon_{k}R_{k}\}.
    \end{equation*}
    Recall that $R_{k}y^{*}\in(-\frac{R_{k}}{2},\frac{R_{k}}{2})\times\{0\}$, we see that for $k$ sufficiently large, one always have $\epsilon_{k}\leq\rho /\sqrt{2}$. Therefore, there exists a unit vector $\vec{e}$ with $\angle{(\vec{e},\vec{e}_{2})}\leq\frac{\pi}{4}$, such that the ray $z+t\vec{e}$ with $t\geq0$ enters the region $\{x_{2}\geq\epsilon_{k}R_{k}\}$ before leaving the ball $B_{\rho R_{k}}(R_{k}y^{*})$. Precisely, we mean
    \begin{equation*}
        z+t\vec{e}\in B_{\rho R_{k}}(R_{k}y^{*}),\quad\mbox{for all }0\leq t\leq\frac{\epsilon_{k}R_{k}-x_{2}}{\langle\vec{e},\vec{e}_{2}\rangle}.
    \end{equation*}
    Since $\psi_{k}(x)\geq\Psi(x_{2}-\epsilon_{k}R_{k})$ in $B_{\rho R_{k}}(R_{k}y^{*})\subset\subset[-R_{k},R_{k}]^{2}$, we see that
    \begin{equation*}
        \psi_{k}(z+t\vec{e})>0,\quad\mbox{for any }t\mbox{ slightly larger than }\frac{\epsilon_{k}R_{k}-x_{2}}{\langle\vec{e},\vec{e}_{2}\rangle}.
    \end{equation*}
    Therefore, we let
    \begin{equation*}
        \overline{t}=\sup\Big\{t\geq0,\quad\psi_{k}(z+\tau\vec{e})=0\mbox{ for all }0\leq\tau\leq t\Big\},
    \end{equation*}
    then it must holds that
    \begin{equation*}
        0\leq\overline{t}\leq\frac{\epsilon_{k}R_{k}-x_{2}}{\langle\vec{e},\vec{e}_{2}\rangle},\quad\mbox{and }z+\overline{t}\vec{e}\in B_{\rho R_{k}}(R_{k}y^{*})\cap\partial\{\psi_{k}>0\}.
    \end{equation*}
    By letting $\displaystyle y=\frac{z+\overline{t}\vec{e}}{R_{k}}$ in \eqref{eq. Q_k behavior on free boundary}, one has
    \begin{equation*}
        \langle z+\overline{t}\vec{e},\vec{e}_{2}\rangle+\widetilde{Q}_{k}(\frac{z+\overline{t}\vec{e}}{R_{k}})\epsilon_{k}R_{k}=R_{k}y_{2}+\widetilde{Q}_{k}(y)\epsilon_{k}R_{k}\leq0.
    \end{equation*}
    However, on the other hand, given that $\epsilon_{k}\ll1$, one has
    \begin{equation*}
        \frac{d}{dt}\Big\{\langle z+t\vec{e},\vec{e}_{2}\rangle+\widetilde{Q}_{k}(\frac{z+t\vec{e}}{R_{k}})\epsilon_{k}R_{k}\Big\}\geq\frac{1}{\sqrt{2}}-\epsilon_{k}\|\nabla_{y}\widetilde{Q}\|_{L^{\infty}([-1,1]^{2})}>\frac{1}{10}.
    \end{equation*}
    Integrating the above estimate contradicts the assumption \eqref{eq. if (4.6) is wrong}, hence \eqref{eq. claim lower bound extends to contact region} is verified.

    \textbf{Step 4: Coordinate change.} Consider the following isometry $E_{k}$:
    \begin{equation*}
        x=E_{k}(\widetilde{x})=\frac{1}{|\vec{e}_{2}+\epsilon_{k}\vec{b}|}\begin{bmatrix}
            1+\epsilon_{k}b_{2}&\epsilon_{k}b_{1}\\-\epsilon_{k}b_{1}&1+\epsilon_{k}b_{2}
        \end{bmatrix}\cdot\begin{bmatrix}
            \widetilde{x}_{1}\\\widetilde{x}_{2}
        \end{bmatrix}+R_{k}y^{*}.
    \end{equation*}
    Let $\widetilde{\psi}_{k}(\widetilde{z}):B_{\rho R_{k}}\to\mathbb{R}$ be defined as follows:
    \begin{equation*}
        \widetilde{\psi}_{k}(\widetilde{x}):=(\psi_{k}\circ E_{k})(\widetilde{x})=\psi(x).
    \end{equation*}
    By \eqref{eq. claim lower bound extends to contact region}, we have that
    \begin{equation*}
        \widetilde{\psi}_{k}(\widetilde{x})\geq\Psi\Big(|\vec{e}_{2}+\epsilon_{k}\vec{b}|\cdot\widetilde{x}_{2}-M\frac{\epsilon_{k}}{R_{k}}|\widetilde{x}|^{2}+h_{k}\epsilon_{k}R_{k}\Big)\mbox{ in }B_{\rho R_{k}}.
    \end{equation*}
    For sufficiently large $k$, we always have $|h_{k}|\ll\rho^{2}$ and
    \begin{equation*}
        |\vec{e}_{2}+\epsilon_{k}\vec{b}|\cdot\widetilde{x}_{2}-M\frac{\epsilon_{k}}{R_{k}}|\widetilde{x}|^{2}+h_{k}\epsilon_{k}R_{k}\geq|\vec{e}_{2}+\epsilon_{k}\vec{b}|\cdot\widetilde{x}_{2}-(M\rho^{2}-h_{k})\epsilon_{k}R_{k}\quad\mbox{for }\widetilde{x}\in B_{\rho R_{k}}.
    \end{equation*}
    Let $\iota=2M\rho$, then $0<\iota\leq\frac{\lambda\beta}{512}$ and $\mu$ is independent of $k$. Furthermore, for sufficiently large $k$, we have
    \begin{equation*}
        \beta\leq\frac{|\vec{e}_{2}+\epsilon_{k}\vec{b}|-1}{\epsilon_{k}},\quad\mbox{and }\iota\geq\frac{M\rho^{2}-h_{k}}{\rho}.
    \end{equation*}
    Then, one has
    \begin{equation*}
        \widetilde{\psi}_{k}(\widetilde{x})\geq\Psi\Big((1+\beta\epsilon_{k})\cdot\widetilde{x}_{2}-\iota\epsilon_{k}\widetilde{R}_{k}\Big)\mbox{ in }B_{\widetilde{R}_{k}},\quad\mbox{where }\widetilde{R}_{k}=\rho\cdot R_{k}.
    \end{equation*}

    \textbf{Step 5: Apply the maximum principle.} When $k$ is sufficiently large, we also have $\epsilon_{k}\leq\overline{\epsilon}=\delta_{0}(\beta)$. Since $0<\iota\leq\frac{\lambda\beta}{512}$, applying Corollary~\ref{cor. application B of maximum principle} to $\widetilde{\psi}_{k}$ in $B_{\widetilde{R}_{k}}$ yields that
    \begin{equation*}
        \widetilde{\psi}_{k}(\widetilde{x})\geq\Psi\big(\widetilde{x}_{2}+\frac{\lambda\beta}{512}\epsilon_{k}\widetilde{R}_{k}\big)=\Psi\big(\widetilde{x}_{2}+\frac{\lambda\beta}{512}\epsilon_{k}\rho R_{k}\big)\mbox{ in }B_{\frac{\lambda}{4}\widetilde{R}_{k}}.
    \end{equation*}
    In the $z$-coordinate (via the map $E_{k}$), we have
    \begin{equation*}
        \psi_{k}(x)\geq\Psi\Big(\frac{\epsilon_{k}b_{1}(x_{1}-R_{k}y^{*}_{1})+(1+\epsilon_{k}b_{2})x_{2}}{|\vec{e}_{2}+\epsilon_{k}\vec{b}|}+\frac{\lambda\beta}{512}\epsilon_{k}\rho R_{k}\Big)\mbox{ in }B_{\frac{\lambda}{4}\rho R_{k}}(R_{k}y^{*}).
    \end{equation*}
    
    \textbf{Step 6: Derive a contradiction.} Now let us study $v_{k}$ again, where $k$ is sufficiently large. By the definition of $v_{k}$, one has
    \begin{align*}
        v_{k}(y)\geq&\frac{\epsilon_{k}b_{1}(x_{1}-R_{k}y^{*}_{1})+(1+\epsilon_{k}b_{2}-|\vec{e}_{2}+\epsilon_{k}\vec{b}|)x_{2}}{\epsilon_{k}R_{k}\cdot|\vec{e}_{2}+\epsilon_{k}\vec{b}|}+\frac{\lambda\beta\rho}{512}\\
        \geq&-\frac{|\epsilon_{k}b_{1}|+(1+\epsilon_{k}b_{2}-|\vec{e}_{2}+\epsilon_{k}\vec{b}|)}{\epsilon_{k}\cdot|\vec{e}_{2}+\epsilon_{k}\vec{b}|}\cdot|y-y^{*}|+\frac{\lambda\beta\rho}{512}\\
        \geq&-2|b_{1}|\cdot|y-y^{*}|+\frac{\lambda\beta\rho}{512},\quad\mbox{for all }y=\frac{x}{R_{k}}\in B_{\frac{\lambda\rho}{4}}(y^{*})\cap\Omega_{k}.
    \end{align*}
    By the continuity of $v_{k}$ in $\Omega_{k}$, we further have
    \begin{equation}\label{eq. v_k(z) is large}
        \liminf_{\substack{\eta\in B_{\rho}(y^{*})\cap\Omega_{k}\\\eta\to y}}v_{k}(\eta)\geq-2|b_{1}|\cdot|y-y^{*}|+\frac{\lambda\beta\rho}{512},\quad\mbox{for all }y\in\overline{B_{\frac{\lambda\rho}{4}}(y^{*})\cap\Omega_{k}}.
    \end{equation}
    In particular, let $y=y^{(*,k)}$ in \eqref{eq. v_k(z) is large}. Since $|y^{(*,k)}-y^{*}|\to0$, we see that
    \begin{equation*}
        \liminf_{\substack{\eta\in B_{\rho}(y^{*})\cap\Omega_{k}\\\eta\to y^{(*,k)}}}v_{k}(\eta)\geq\frac{\lambda\beta\rho}{1024}\quad\mbox{for sufficiently large }k.
    \end{equation*}
    However, recall that it follows from \eqref{eq. sequence of contact points}, \eqref{eq. v_k minimal assumption}, and \eqref{eq. widetilde Q_k form} that
    \begin{equation*}
        \lim_{k\to\infty}\liminf_{\substack{\eta\in B_{\rho}(y^{*})\cap\Omega_{k}\\\eta\to y^{(*,k)}}}v_{k}(\eta)=\lim_{k\to\infty}\Big\{-M|y^{(*,k)}-y^{*}|^{2}+\vec{b}\cdot(y^{(*,k)}-y^{*})+h_{k}\Big\}=0,
    \end{equation*}
    causing a contradiction. Hence, the Neumann boundary condition is verified.
\end{proof}
\subsection{Schauder estimate for the purely free boundary problem} In this subsection, our goal is to prove Theorem~\ref{thm. Schauder pure}. Firstly, we need the following $C^{2}$ estimate:
\begin{lemma}\label{lem. C^2 estimate for harmonic functions}
    Let $B_{1}^{+}\subseteq\mathbb{R}^{2}$ be the two-dimensional unit half ball, and let $\mu$ be given in \eqref{eq. vertical mach deform}. Assume that $v\in C(\overline{B_{1}^{+}})$ satisfies
    \begin{equation}\label{eq. harmonic and even-symmetric problem}
        \left\{\begin{aligned}
            &(\mu^{2}\partial_{11}+\partial_{22})v(y)=0\mbox{ and }|v(y)|\leq1&\mbox{in }&B_{1}^{+},\\
            &\partial_{2}v=0&\mbox{on }&[-1,1]\times\{0\}.
        \end{aligned}\right.
    \end{equation}
    Then, one has $\|v\|_{C^{1,1}(B_{1/2}^{+})}\leq K$ for some uniform constant $K$. Moreover, there exists an affine function $l(y)$ (which is even in $y_{2}$), such that
    \begin{equation*}
        \|l\|_{L^{\infty}(B_{1})}\leq K,\quad\mbox{and }|v(y)-l(y)|\leq K|y|^{2}\mbox{ in }B_{1}^{+}.
    \end{equation*}
\end{lemma}
\begin{proof}
    After a coordinate dilation in the $x_{1}$-direction, $v$ becomes harmonic in the new coordinate. Then, the desired estimate follows from the Green representation formula.
\end{proof}
\begin{proof}[Proof of Theorem~\ref{thm. Schauder pure}]
    The proof is based on an contradiction argument. Let $\alpha\in(0,1)$ be arbitrary and fixed. Let $K$ be the uniform constant produced by Lemma~\ref{lem. C^2 estimate for harmonic functions}. Then, let us fix a $\rho$ as follows:
    \begin{equation*}
        \rho=\min\Big\{\big(\frac{1}{2K}\big)^{\frac{1}{1-\alpha}},1\Big\}.
    \end{equation*}
    
    \textbf{Step 1: A sequence of counter examples.} Suppose that no matter how small $\overline{\epsilon}$ is chosen, Theorem~\ref{thm. Schauder pure} always fails. Then, there exists a sequence of triples $\{(\epsilon_{k},R_{k},\psi_{k})\}_{k=1}^{\infty}$, such that $\epsilon_{k}\to0$, $R_{k}\to\infty$ and that
    \begin{equation*}
        \psi_{k}\mbox{ solves \eqref{eq. main simple} with }\Psi(x_{2}-\epsilon_{k}R_{k})\leq\psi(x)\leq\Psi(x_{2}+\epsilon_{k}R_{k})\quad\mbox{in }B_{R_{k}}.
    \end{equation*}
    But, there does not exist $(\vec{\nu},a,b)$ satisfying \eqref{eq. Schauder pure result line 1} and \eqref{eq. Schauder pure result line 2}.

    \textbf{Step 2: Analysis of $v_{k}$.} Define $v_{k}(y)$ and $\Omega_{k}$ as in Lemma~\ref{lem. limit description}. By Lemma~\ref{lem. limit description}, we know that there exists a function $v$ satisfying \eqref{eq. harmonic and even-symmetric problem}, such that up to a subsequence (still labeled by $k$), one has
    \begin{equation*}
        \|v_{k}-v\|_{L^{\infty}(B_{1/2}\cap\Omega_{k})}\to0,\quad\mbox{and }B_{1/2}\cap\{x_{2}\geq\epsilon_{k}\}\subseteq B_{1/2}\cap\Omega_{k}\subseteq B_{1/2}\cap\{x_{2}\geq-\epsilon_{k}\}.
    \end{equation*}
    Here, we have evenly extended $v$ to $B_{1}^{-}$, such that $\|v_{k}-v\|_{L^{\infty}(B_{1/2}\cap\Omega_{k})}$ is well-defined. By Lemma~\ref{lem. C^2 estimate for harmonic functions}, there exists an affine function
    \begin{equation*}
        l(y)=l_{0}+b\cdot y_{1}\mbox{ with }|l_{0}|,|b|\leq K,
    \end{equation*}
    such that
    \begin{equation*}
        |v_{k}(y)-l(y)|\leq|v(y)-l(y)|+|v_{k}(y)-v(y)|\leq K|y|^{2}+o(1)\quad\mbox{in }B_{1/2}\cap\Omega_{k}.
    \end{equation*}

    \textbf{Step 3: Analysis of $\psi_{k}$.} By the definition of $v_{k}$, one has
    \begin{equation*}
        \Psi\Big(x_{2}+\epsilon_{k}R_{k}\cdot\big(l(\frac{x}{R})-K\rho^{2}\big)\Big)\leq\psi_{k}(x)\leq\Psi\Big(x_{2}+\epsilon_{k}R_{k}\cdot\big(l(\frac{x}{R})+K\rho^{2}\big)\Big)\mbox{ in }B_{\rho R_{k}}\cap\{\psi_{k}>0\}.
    \end{equation*}
    Let
    \begin{equation*}
        \vec{\nu}=\frac{\vec{e}_{2}+\epsilon_{k}b\vec{e}_{1}}{|\vec{e}_{2}+\epsilon_{k}b\vec{e}_{1}|},\quad\left\{\begin{aligned}
            a=&l_{0}\epsilon_{k}R_{k}-(\frac{\epsilon_{k}\cdot K^{2}}{2}+K\rho)\cdot\rho\epsilon_{k}R_{k},\\
            b=&l_{0}\epsilon_{k}R_{k}+(\frac{\epsilon_{k}\cdot K^{2}}{2}+K\rho)\cdot\rho\epsilon_{k}R_{k}.
        \end{aligned}\right.
    \end{equation*}
    Using the fact
    \begin{equation*}
        \Big|\big(x_{2}+\epsilon_{k}R_{k}\cdot l(\frac{x}{R_{k}})\big)-\big(l_{0}\epsilon_{k}R_{k}+x\cdot\vec{\nu}\big)\Big|=|x\cdot\vec{\nu}|\cdot\Big||\vec{e}_{2}+\epsilon_{k}b\vec{e}_{1}|-1\Big|\leq\frac{\epsilon_{k}^{2}\cdot|b|^{2}}{2}\cdot|x|,
    \end{equation*}
    one gets that
    \begin{equation*}
        \Big|\big(x_{2}+\epsilon_{k}R_{k}\cdot l(\frac{x}{R_{k}})\big)-\big(l_{0}\epsilon_{k}R_{k}+x\cdot\vec{\nu}\big)\Big|\leq\frac{\epsilon_{k}\cdot K^{2}}{2}\rho\epsilon_{k}R_{k}\mbox{ in }B_{\rho R_{k}}.
    \end{equation*}
    Therefore, one has
    \begin{equation*}
        \Psi(x\cdot\vec{\nu}+a)\leq\psi_{k}(x)\leq\Psi(x\cdot\vec{\nu}+b)\quad\mbox{in }B_{\rho R_{k}}\cap\{\psi_{k}>0\}.
    \end{equation*}
    
    We claim that indeed
    \begin{equation}\label{eq. claim of Schauder from positive set to contact set}
        \Psi(x\cdot\vec{\nu}+a)\leq\psi_{k}(x)\leq\Psi(x\cdot\vec{\nu}+b)\quad\mbox{in }B_{\rho R_{k}}.
    \end{equation}
    The proof of \eqref{eq. claim of Schauder from positive set to contact set} is similar to that of \eqref{eq. claim lower bound extends to contact region}. In fact, the second inequality of \eqref{eq. claim of Schauder from positive set to contact set} is clearly still valid even if $\psi_{k}(x)=0$, so the only possibility that \eqref{eq. claim of Schauder from positive set to contact set} can go wrong is that there exists a point $z\in B_{\rho R_{k}}$, such that
    \begin{equation*}
        \psi_{k}(z)=0\quad\mbox{and }z\cdot\vec{\nu}+a>0.
    \end{equation*}
    The assumption $\psi_{k}(z)=0$ together with the second inequality of \eqref{eq. claim of Schauder from positive set to contact set} implies that $z\cdot\vec{\nu}+b\leq0$. Then, there exists a unit vector $\vec{e}$ with $\angle{(\vec{\nu},\vec{e})}<\frac{\pi}{2}$, such that the ray $z+t\vec{e}$ with $t\geq0$ enters the region where $\psi_{k}>0$ for the first time prior to its leaving $B_{\rho R_{k}}$. Let
    \begin{equation*}
        \overline{t}:=\sup\Big\{\tau\geq0:\psi_{k}(z+t\vec{e})=0\mbox{ for all }0\leq t\leq\tau\Big\},
    \end{equation*}
    then one has that $\displaystyle\lim_{t\to\overline{t}+}\psi_{k}(z+t\vec{e})=0$, while
    \begin{equation*}
        \langle z+\overline{t}\vec{e},\nu\rangle+a>z\cdot\vec{\nu}+a>0,
    \end{equation*}
    contradicting that $\Psi(x\cdot\vec{\nu}+a)\leq\psi_{k}(x)$ in $B_{\rho R_{k}}\cap\{\psi_{k}>0\}$. This verifies \eqref{eq. claim of Schauder from positive set to contact set}.

    \textbf{Step 4: Derivation of a contradiction.} By the choice of $(\vec{\nu},a,b)$, we see that $|\vec{\nu}-\vec{e}_{2}|\leq b\cdot\epsilon_{k}$. Moreover, by the choice of $\rho$, one has $K\rho\leq\frac{1}{2}\rho^{\alpha}$. For sufficiently large $k$, $\epsilon_{k}$ is so small that $\epsilon_{k}\cdot K^{2}\leq\rho^{\alpha}$, so we have
    \begin{equation*}
        0\leq(\frac{\epsilon_{k}\cdot K^{2}}{2}+K\rho)\cdot\rho\epsilon_{k}R_{k}\leq\rho^{1+\alpha}\epsilon_{k}R_{k}.
    \end{equation*}
    Then, we have verified that \eqref{eq. claim of Schauder from positive set to contact set} holds for $\psi_{k}$, with $(\vec{\nu},a,b)$ satisfying \eqref{eq. Schauder pure result line 1}. This violates our assumption in \textbf{Step 1} that $\{(\epsilon_{k},R_{k},\psi_{k})\}_{k=1}^{\infty}$ is a sequence of counter examples. Hence, we have finished the proof of Theorem~\ref{thm. Schauder pure} via the a priori method.
\end{proof}
\subsection{Schauder estimate for the cavity problem}
In this subsection, we prove the first main result (Theorem~\ref{thm. Schauder mixed}). Similar to the proof of Theorem~\ref{thm. Schauder mixed}, we also need a discription of the blow-down limit, see below.
\begin{lemma}\label{lem. blow-down discreption mixed}
    Assume that $\epsilon_{k}\to0$ and $\epsilon_{k}R_{k}\to\infty$. Assume that
    \begin{equation*}
        \left\{\begin{aligned}
            &\nabla\cdot\Big(g(|\nabla\psi|^{2},\psi)\nabla\psi\Big)=\frac{\mathcal{B}'(\psi)}{g(|\nabla\psi|^{2},\psi)}\mbox{ and }\psi>0&\mbox{in }&\Gamma^{+}\cap[-R_{k},R_{k}]^{2},\\
            &\psi=0&\mbox{on }&\Gamma\cap[-R_{k},R_{k}]^{2},\\
            &|\nabla\psi|=\Lambda&\mbox{on }&\{x_{2}=h(x_{1})\}\cap[-R_{k},R_{k}]^{2},\\
            &\Psi(x_{2}-\epsilon_{k}R_{k})\leq\psi(x)\leq\Psi(x_{2})&\mbox{in }&[-R_{k},R_{k}]^{2}.
        \end{aligned}\right.
    \end{equation*}
    Denote $\Omega_{k}:=[-1,1\big]^{2}\cap\{\psi_{k}(R_{k}y)>0\}$ and denote $v_{k}:\Omega_{k}\to\mathbb{R}$ as
    \begin{equation*}
        v_{k}(y)=\frac{\Psi^{-1}(\psi_{k}(x))-x_{2}}{\epsilon_{k}R_{k}},\quad\mbox{where }x=R_{k}y.
    \end{equation*}
    Then, there exists a subsequence of $\{v_{k}\}$ (still labeled by $k$) and a function $v:[-\frac{1}{2},\frac{1}{2}]\times[0,\frac{1}{2}]\to\mathbb{R}$, such that the followings hold:
    \begin{itemize}
        \item[(1)] For every $k$, one has $[-1,1]\times[\epsilon_{k},1]\subseteq\Omega_{k}\subseteq[-1,1]\times[-\epsilon_{k},1]$.
        \item[(2)] As $k\to\infty$, one has
        \begin{equation*}
            \lim_{k\to\infty}\|v_{k}-v\|_{L^{\infty}([-\frac{1}{2},\frac{1}{2}]\times[0,\frac{1}{2}]\cap\Omega_{k})}=0.
        \end{equation*}
        Besides, for all fixed $\epsilon>0$, one has
        \begin{equation*}
            \lim_{k\to\infty}\|v_{k}-v\|_{L^{\infty}([-\frac{1}{2},\frac{1}{2}]\times[\epsilon,\frac{1}{2}])}=0.
        \end{equation*}
        \item[(3)] For $\alpha$ being the constant in Corollary~\ref{cor. pure boundary Holder}, one has $v\in C^{\alpha}([-\frac{1}{2},\frac{1}{2}]\times[0,\frac{1}{2}])$;
    \end{itemize}
    Moreover, in the viscosity sense (similar to that in Lemma~\ref{lem. limit description}), the limiting function $v$ satisfies
    \begin{equation}\label{eq. mixed limit fractional problem}
        \left\{\begin{aligned}
            &\mu^{2}v_{11}+v_{22}=0\mbox{ and }-1\leq v(y)\leq0&\mbox{in }&B_{1}^{+},\\
            &v=0&\mbox{on }&[-1,0]\times\{0\},\\
            &\partial_{2}v=0&\mbox{on }&[0,1]\times\{0\}.
        \end{aligned}\right.
    \end{equation}
    See definition of $\mu$ in \eqref{eq. vertical mach deform}.
\end{lemma}
\begin{proof}
    Using Corollary~\ref{cor. mixed boundary Holder}, the proof of Lemma~\ref{lem. blow-down discreption mixed} (1)-(3) is the same as Lemma~\ref{lem. limit description}. Using the same argument as in Lemma~\ref{lem. limit description} (4), we can prove that $(\mu^{2}\partial_{11}+\partial_{22})v(y)=0$ in $B_{1}^{+}$.
    
    Moreover, since in the first quadrant, $\partial\{\psi>0\}$ consists solely of free boundary points (i.e. $|\nabla\psi|=\Lambda$), applying the same argument as in Lemma~\ref{lem. limit description} (5) gives that $\partial_{2}v=0$ on $[0,1]\times\{0\}$.
    
    Finally, applying the estimate \eqref{eq. uniform vanishing on -x axis} implies that $v=0$ on $[-1,0]\times\{0\}$. Then, we have finished the proof of Lemma~\ref{lem. blow-down discreption mixed}.
\end{proof}
We need the following boundary Harnack principle.
\begin{lemma}\label{lem. BHP slit}
    Let $B_{1}^{+}\subseteq\mathbb{R}^{2}$ be the two-dimensional unit half ball. There exists a uniform constant $K$ such that the following fact holds: Assume that $v\in C(\overline{B_{1}^{+}})$ satisfies
    \begin{equation*}
        \left\{\begin{aligned}
            &(\mu^{2}\partial_{11}+\partial_{22})v(y)=0\mbox{ and }-1\leq v(y)\leq0&\mbox{in }&B_{1}^{+},\\
            &v=0&\mbox{on }&[-1,0]\times\{0\},\\
            &\partial_{2}v=0&\mbox{on }&[0,1]\times\{0\}.
        \end{aligned}\right.
    \end{equation*}
    Then, either $v(y)\equiv0$ in $B_{1/2}^{+}$, or there exists a real number $-K\leq A<0$, such that
    \begin{equation*}
        \Big|\frac{v(y)}{A\cdot Re(\sqrt{y_{1}+i\mu y_{2}})}-1\Big|\leq K\cdot|y|\quad\mbox{in }B_{1/2}^{+}.
    \end{equation*}
    Here, the constant $\mu$ is given in \eqref{eq. vertical mach deform}.
\end{lemma}
\begin{proof}
    Assume that $v\not\equiv0$ in $B_{1/2}^{+}$, then it follows from the strong maximum principle that $v<0$ in $B_{1}^{+}$. Then, after a dilation in the $y_{1}$-coordinate, the operator $(\mu^{2}\partial_{11}+\partial_{22})$ becomes the standard Laplacian in the new coordinate, which allows us to apply the boundary Harnack principle in a slit domain (see \cite[Theorem 2.3]{DS16}) to $-v$. Mapping back to the original $y$-coordinate gives the disired estimate.
\end{proof}
\begin{proof}[Proof of Theorem~\ref{thm. Schauder mixed}]
    The proof is similar to that of Theorem~\ref{thm. Schauder pure}. Let $\alpha\in(0,1)$ be arbitrary and fixed. Let $K$ be the uniform constant produced by Lemma~\ref{lem. BHP slit}. Then, let us fix a $\rho$ as follows:
    \begin{equation*}
        \rho=\min\Big\{\big(\frac{1}{2K}\big)^{\frac{1}{1-\alpha}},1\Big\}.
    \end{equation*}
    
    \textbf{Step 1: A sequence of counter examples.} Suppose that no matter how small $\overline{\epsilon}$ and how large $M$ is chosen, Theorem~\ref{thm. Schauder mixed} always fails. Then, there exists a sequence of triples $\{(\epsilon_{k},R_{k},\psi_{k})\}_{k=1}^{\infty}$, such that $\epsilon_{k}\to0$, $\epsilon_{k}R_{k}\to\infty$ and that
    \begin{equation*}
        \psi_{k}\mbox{ solves \eqref{eq. main simple} with }\Psi(x_{2}-\epsilon_{k}R_{k})\leq\psi(x)\leq\Psi(x_{2})\quad\mbox{in }B_{R_{k}}.
    \end{equation*}
    But, there does not exist $A_{R}$ satisfying the desired approximation estimate.

    \textbf{Step 2: Analysis of $v_{k}$.} Define $v_{k}(y)$ and $\Omega_{k}$ as in Lemma~\ref{lem. limit description}. By Lemma~\ref{lem. blow-down discreption mixed}, we know that there exists a function $v$ satisfying \eqref{eq. mixed limit fractional problem}, such that up to a subsequence (still labeled by $k$), one has
    \begin{equation*}
        \|v_{k}-v\|_{L^{\infty}(B_{1/2}\cap\Omega_{k})}\to0,\quad\mbox{and }B_{1/2}\cap\{x_{2}\geq\epsilon_{k}\}\subseteq B_{1/2}\cap\Omega_{k}\subseteq B_{1/2}^{+}.
    \end{equation*}
    By Lemma~\ref{lem. BHP slit}, either $v\equiv0$, or there exists sone $-K\leq A<0$ such that
    \begin{equation*}
        l(y)=l_{0}+b\cdot y_{1}\mbox{ with }|l_{0}|,|b|\leq K,
    \end{equation*}
    such that
    \begin{equation*}
        \Big|\frac{v(y)}{A\cdot Re(\sqrt{y_{1}+i\mu y_{2}})}-1\Big|\leq K\cdot|y|\quad\mbox{in }B_{1/2}^{+}.
    \end{equation*}
    Using the triangle inequality, we get that
    \begin{align*}
        |v_{k}(y)-A\cdot Re(\sqrt{y_{1}+i\mu y_{2}})|\leq&|v(y)-A\cdot Re(\sqrt{y_{1}+i\mu y_{2}})|+|v_{k}(y)-v(y)|\\
        \leq&K\cdot|A|\cdot Re(\sqrt{y_{1}+i\mu y_{2}})\cdot|y|+o(1)\quad\mbox{in }B_{1/2}\cap\Omega_{k}.
    \end{align*}
    This contradicts our assumption, since $o(1)\leq\delta$ for sufficiently large $k$.
\end{proof}
\section{A priori growth rate estimate of the cavity}\label{sec. growth rate}
\subsection{Non-degeneracy of the Schauder estimate}
In this subsection, when studying the growth rate of the cavity $h(x_{2})$, we would like to exclude the possibility that the leading coefficient $A_{R}$ in Theorem~\ref{thm. Schauder mixed} is too close to zero. Therefore, we give the following non-degeneracy estimate.
\begin{lemma}[Non-degeneracy]\label{lem. Non-degeneracy}
    Let $\psi$ be a fixed global solution satisfying \eqref{eq. cavity local problem}. Let
    \begin{equation*}
        M_{R}=\|\Psi^{-1}(\psi(x))-x_{2}\|_{L^{\infty}(B_{R}\cap\{\psi>0\})}.
    \end{equation*}
    Assume that
    \begin{equation*}
        M_{R}\to\infty\mbox{ and }\epsilon_{R}=\frac{M_{R}}{R}\to0,\quad\mbox{as }R\to\infty.
    \end{equation*}
    Then, there exist two uniform constant $\underline{a},\overline{\delta}>0$ and a function $\underline{R}(\delta)$, such that the following fact holds for all $0<\delta\leq\overline{\delta}$: For all $R\geq\underline{R}(\delta)$, let $-K\leq A_{R}\leq0$ be then coefficient produced by Theorem~\ref{thm. Schauder mixed} with $\epsilon=\epsilon_{R}$, such that \eqref{eq. Schauder mixed error} holds. Then, we must have $A_{R}\leq-\underline{a}<0$.
\end{lemma}
\begin{proof}
    \textbf{Step 1: Choice of constants.} We first fix
    \begin{equation}\label{eq. non-degeneracy constant requirements}
        \rho=\min\Big\{\frac{1}{4K},\frac{1}{10}\Big\},\quad\underline{a}=\frac{2}{5}\rho^{3/2},\quad\overline{\delta}=\frac{1}{40}\rho^{5/2}.
    \end{equation}
    Furthermore, we require that $\underline{R}(\delta)$ is sufficiently large such that
    \begin{equation*}
        \epsilon_{R}\leq\frac{\rho^{100}}{100}\mbox{ and the estimate in Theorem~\ref{thm. Schauder mixed} holds},\quad\mbox{for all }R\geq\underline{R}(\delta).
    \end{equation*}
    We claim that
    \begin{equation}\label{eq. Claim of non-degenerate}
        A_{R}\leq-\underline{a},\quad\mbox{for all }R\geq\underline{R}(\delta).
    \end{equation}

    \textbf{Step 2: Outline of the induction.} We prove \eqref{eq. Claim of non-degenerate} by contradiction: suppose that \eqref{eq. Claim of non-degenerate} fails, i.e. there exists some $R\geq\underline{R}(\delta)$ such that $-\underline{a}\leq A_{R}\leq0$. We would like to inductively show the following two statements for all $k\geq0$:
    \begin{equation*}
        (\mathcal{A}_{k}): -\underline{a}\leq A_{\rho^{-k}R}\leq0,\qquad(\mathcal{B}_{k}): \frac{M_{\rho^{1-k}R}}{M_{\rho^{-k}R}}\leq\rho^{2}.
    \end{equation*}
    Here, the contradiction assumption means that $(\mathcal{A}_{0})$ is correct. Then our approach is to show $(\mathcal{A}_{k})\Rightarrow(\mathcal{B}_{k})$ and $(\mathcal{B}_{k})\Rightarrow(\mathcal{A}_{k+1})$ for all $k\geq0$.

    To simplify the writing later on, we make the following notations:
    \begin{equation*}
        \widetilde{v}_{k}:=v_{\rho^{-k}R},\quad\Omega_{k}:=B_{1}\cap\{\psi(\rho^{-k}Ry)>0\}.
    \end{equation*}

    \textbf{Step 3: Proof of $(\mathcal{A}_{k})\Rightarrow(\mathcal{B}_{k})$ for $k\geq0$.} Observe that
    \begin{equation*}
        \frac{M_{\rho^{1-k}R}}{M_{\rho^{-k}R}}=\|\widetilde{v}_{k}\|_{L^{\infty}(B_{\rho}\cap\{\psi(\rho^{-k}Ry)>0\})}.
    \end{equation*}
    Since $\rho^{-k}R\geq R\geq\underline{R}(\delta)$, the estimate \eqref{eq. Schauder mixed error} applies to $\widetilde{v}_{k}$. Then, applying the triangle inequality yields that
    \begin{equation}\label{eq. denominator bound in iteration}
        \|\widetilde{v}_{k}\|_{L^{\infty}(B_{\rho}\cap\Omega_{k})}\leq|A_{\rho^{-k}R}|\sqrt{\rho}+K\cdot|A_{\rho^{-k}R}|\cdot\rho^{3/2}+\delta\leq\frac{5}{4}|A_{\rho^{-k}R}|\sqrt{\rho}+\delta.
    \end{equation}
    By the assumption $(\mathcal{A}_{k})$ and by \eqref{eq. non-degeneracy constant requirements}, one has
    \begin{equation*}
        \frac{M_{\rho^{1-k}R}}{M_{\rho^{-k}R}}\leq\frac{5}{4}\cdot\underline{a}\cdot\sqrt{\rho}+\delta\leq\rho^{2}.
    \end{equation*}
    
    \textbf{Step 4: Proof of $(\mathcal{B}_{k})\Rightarrow(\mathcal{A}_{k+1})$ for $k\geq0$.} If $(\mathcal{B}_{k})$ is true, then
    \begin{equation*}
        \rho^{2}\geq\frac{M_{\rho^{1-k}R}}{M_{\rho^{-k}R}}=\frac{\|\widetilde{v}_{k+1}\|_{L^{\infty}(B_{\rho^{2}}\cap\Omega_{k+1})}}{\|\widetilde{v}_{k+1}\|_{L^{\infty}(B_{\rho}\cap\Omega_{k+1})}}.
    \end{equation*}
    Since $\rho^{-k}R\geq R\geq\underline{R}(\delta)$, the estimate \eqref{eq. Schauder mixed error} applies to $\widetilde{v}_{k+1}$. The numerator is bounded from below by:
    \begin{align*}
        \|\widetilde{v}_{k+1}\|_{L^{\infty}(B_{\rho^{2}}\cap\Omega_{k+1})}\geq&\big|\widetilde{v}_{k+1}(\rho^{2}\vec{e}_{2})\big|\geq|A_{\rho^{-k-1}R}|\cdot\sqrt{\frac{\rho^{2}}{2}}-K\cdot|A_{\rho^{-k-1}R}|\cdot|\rho^{2}|^{3/2}-\delta\\
        =&\Big(\frac{1}{\sqrt{2}}-K\rho^{2}\Big)\cdot|A_{\rho^{-k-1}R}|\cdot\rho-\delta\geq\frac{1}{4}|A_{\rho^{-k-1}R}|\cdot\rho-\delta.
    \end{align*}
    On the other hand, similar to \eqref{eq. denominator bound in iteration}, the denominator is bounded from above by:
    \begin{equation*}
        \|\widetilde{v}_{k+1}\|_{L^{\infty}(B_{\rho}\cap\Omega_{k+1})}\leq\frac{5}{4}\cdot|A_{\rho^{-k-1}R}|\cdot\sqrt{\rho}+\delta.
    \end{equation*}
    Combining all estimates above and using $\delta\leq\overline{\delta}=\frac{1}{40}\rho^{5/2}$ yields that:
    \begin{equation*}
        \frac{1}{4}|A_{\rho^{-k-1}R}|\cdot\rho\leq\frac{5}{4}\cdot|A_{\rho^{-k-1}R}|\cdot\rho^{5/2}+(1+\rho^{2})\delta\leq\frac{1}{8}|A_{\rho^{-k-1}R}|\cdot\rho+\frac{1}{20}\rho^{5/2}.
    \end{equation*}
    Therefore,
    \begin{equation*}
        -A_{\rho^{-k-1}R}=|A_{\rho^{-k-1}R}|\leq\frac{2}{5}\rho^{3/2}=\underline{a},
    \end{equation*}
    and thus $(\mathcal{A}_{k+1})$ is verified.
    
    \textbf{Step 5: Ending statement.} Now, we have proven the correctness of $(\mathcal{A}_{k})$'s and $(\mathcal{B}_{k})$'s. However, we infer from $(\mathcal{B}_{k})$'s that $M_{\rho^{-k}R}\geq\rho^{-2k}M_{R}$ for all $k\geq0$. In other words,
    \begin{equation*}
        \epsilon_{\rho^{-k}R}=\frac{M_{\rho^{-k}R}}{\rho^{-k}R}\geq\rho^{-k}\cdot\frac{M_{R}}{R}\to\infty,\quad\mbox{as }k\to\infty.
    \end{equation*}
    This violates the assumption that $\displaystyle\lim_{R\to\infty}\epsilon_{R}\to0$. Therefore, we have verified \eqref{eq. Claim of non-degenerate} and thus have finished the proof of Lemma~\ref{lem. Non-degeneracy}.
\end{proof}
\subsection{Proof of the a priori growth rate estimate}
\begin{proof}[Proof of Theorem~\ref{thm. growth rate of free boundary}]
    Since $\psi$ is asymptotically flat, there exists a function $\epsilon_{R}$ with $\lim_{R\to\infty}\epsilon_{R}\to0$, such that \eqref{eq. cavity main problem} holds. In \textbf{Step 1-4}, we prove Theorem~\ref{thm. growth rate of free boundary} under the additional assumption that $\displaystyle\limsup_{R\to\infty}\epsilon_{R}R=+\infty$ (akin to Savin's setting). Later, in \textbf{Step 5-6}, under the assumption that $h(x_{1})$ is strictly increasing, we rule out the possibility that $\displaystyle\limsup_{R\to\infty}\epsilon_{R}R<+\infty$ (akin to Gibbons' setting).

    \textbf{Step 1: Choice of constants.} Let $\underline{a},\overline{\delta}>0$ and $\underline{R}(\delta)$ be the constants and the function produced by Lemma~\ref{lem. Non-degeneracy}. For every given $\epsilon>0$, we choose
    \begin{equation*}
        \rho=\min\Big\{\frac{\epsilon}{100K},\frac{1}{10}\Big\},\quad\delta_{0}=\min\Big\{\frac{\epsilon\cdot\underline{a}\cdot\sqrt{\rho}}{100},\overline{\delta}\Big\}.
    \end{equation*}
    We now require $R_{0}$ to be so large, that
    \begin{equation*}
        R_{0}\geq\frac{1}{\rho}\underline{R}(\delta_{0}),\quad\mbox{and that }\frac{M_{R}}{R}\leq\frac{\epsilon}{100\mu}\mbox{for all }R\geq R_{0}.
    \end{equation*}

    \textbf{Step 2: Recover the information of the free boundary.} For every $s_{1},t_{1}\geq R_{0}$ with $\frac{1}{2}\leq\frac{s_{1}}{t_{1}}\leq2$, there exists some $R\geq\underline{R}(\delta_{0})$, such that
    \begin{equation*}
        \frac{s_{1}}{R},\frac{t_{1}}{R}\in[\rho,2\rho],\quad\frac{s_{1}}{R}\vec{e}_{1}+\frac{h(s_{1})}{R}\vec{e}_{2},\frac{t_{1}}{R}\vec{e}_{1}+\frac{h(t_{1})}{R}\vec{e}_{2}\in(B_{3\rho}\setminus B_{\rho})\cap\{0\leq\frac{y_{2}}{y_{1}}\leq\frac{\epsilon}{100\mu}\}.
    \end{equation*}
    By the definition of $v_{R}$ and $h(x_{1})$, one has that
    \begin{equation*}
        \lim_{y_{2}\to\frac{h(s_{1})}{R}+}v_{R}(\frac{s_{1}}{R}\vec{e}_{1}+y_{2}\vec{e}_{2})=\lim_{x_{2}\to h(s_{1})+}\frac{\Psi^{-1}\big(\psi(s_{1}\vec{e}_{1}+x_{2}\vec{e}_{2})\big)-x_{2}}{M_{R}}=-\frac{h(s_{1})}{M_{R}}.
    \end{equation*}
    A similar identity holds for $t_{1}$. Therefore, by assigning $v_{R}(\frac{s_{1}}{R}\vec{e}_{1}+\frac{h(s_{1})}{R}\vec{e}_{2})$ to be its limiting value, i.e. $\displaystyle\lim_{y_{2}\to\frac{h(s_{1})}{R}+}v_{R}(\frac{s_{1}}{R}\vec{e}_{1}+y_{2}\vec{e}_{2})$, and working similarly for $t_{1}$, we get the following identity, which recovers the information of $h(x_{1})$ from the horizontal shifting function $v_{R}$:
    \begin{equation}\label{eq. recovery}
        \frac{h(s_{1})}{h(t_{1})}=\frac{-v_{R}(\frac{s_{1}}{R}\vec{e}_{1}+\frac{h(s_{1})}{R}\vec{e}_{2})}{-v_{R}(\frac{t_{1}}{R}\vec{e}_{1}+\frac{h(t_{1})}{R}\vec{e}_{2})}.
    \end{equation}
    
    \textbf{Step 3: Apply the Schauder estimate.} By Lemma~\ref{lem. Non-degeneracy}, for every $R\geq\underline{R}(\delta_{0})$, there exists a coefficient $-K\leq A_{R}\leq-\underline{a}<0$, such that \eqref{eq. Schauder mixed error} holds. In particular,
    \begin{equation*}
        \Big|v_{R}(y)-A_{R}\cdot Re(\sqrt{y_{1}+i\mu y_{2}})\Big|\leq K\cdot|A_{R}|\cdot(3\rho)^{3/2}+\frac{\epsilon\cdot\underline{a}\cdot\sqrt{\rho}}{100}\quad\mbox{in }B_{3\rho}\cap\{\psi(Ry)>0\}.
    \end{equation*}
    Let $y=y_{1}\vec{e}_{1}+y_{2}\vec{e}_{2}\in B_{3\rho}\setminus B_{\rho}$ with $y_{1}=\frac{s_{1}}{R}\vec{e}_{1}\in[\rho,2\rho]$ and $y_{2}\to\frac{h(s_{1})}{R}+$, then
    \begin{equation*}
        \sqrt{y_{1}}\leq\frac{Re(\sqrt{y_{1}+i\mu y_{2}})}{\sqrt{y_{1}}}\leq\big(1+\frac{\epsilon}{100}\big)\cdot\sqrt{y_{1}}\quad\mbox{if }y_{2}\mbox{ is sufficiently close to }\frac{h(s_{1})}{R}+.
    \end{equation*}
    Therefore,
    \begin{align*}
        \Big|v_{R}\big(\frac{s_{1}}{R}\vec{e}_{1}+\frac{h(s_{1})}{R}\vec{e}_{2}\big)-A_{R}\sqrt{y_{1}}\Big|\leq&|A_{R}|\cdot\frac{\epsilon\cdot\sqrt{y_{1}}}{100}+K\cdot|A_{R}|\cdot(3\rho)^{3/2}+\frac{\epsilon\cdot\underline{a}\cdot\sqrt{\rho}}{100}\\
        \leq&|A_{R}|\cdot\sqrt{y_{1}}\cdot\Big(\frac{\epsilon}{100}+\frac{K\cdot(3\rho)^{3/2}}{\sqrt{y_{1}}}+\frac{\epsilon\sqrt{\rho}}{100\sqrt{y_{1}}}\Big)\\
        \leq&|A_{R}|\cdot\sqrt{y_{1}}\cdot\Big(\frac{\epsilon}{50}+9K\cdot\rho\Big)\leq|A_{R}|\cdot\sqrt{y_{1}}\cdot\frac{\epsilon}{5}.
    \end{align*}
    In other words, we have
    \begin{equation*}
        -A_{R}\sqrt{\frac{s_{1}}{R}}\cdot(1-\frac{\epsilon}{5})\leq-v_{R}\big(\frac{s_{1}}{R}\vec{e}_{1}+\frac{h(s_{1})}{R}\vec{e}_{2}\big)\leq-A_{R}\sqrt{\frac{s_{1}}{R}}\cdot(1+\frac{\epsilon}{5}).
    \end{equation*}
    Similarly, we have
    \begin{equation*}
        -A_{R}\sqrt{\frac{t_{1}}{R}}\cdot(1-\frac{\epsilon}{5})\leq-v_{R}\big(\frac{t_{1}}{R}\vec{e}_{1}+\frac{h(t_{1})}{R}\vec{e}_{2}\big)\leq-A_{R}\sqrt{\frac{t_{1}}{R}}\cdot(1+\frac{\epsilon}{5}).
    \end{equation*}
    
    \textbf{Step 4: Asymptotic behavior of the free boundary.} Using \eqref{eq. recovery}, we see that for $s_{1},t_{1}\geq R_{0}$ with $\frac{1}{2}\leq\frac{s_{1}}{t_{1}}\leq2$,
    \begin{equation*}
        \sqrt{\frac{s_{1}}{t_{1}}}\cdot\frac{1-\frac{\epsilon}{5}}{1+\frac{\epsilon}{5}}\leq\frac{h(s_{1})}{h(t_{1})}\leq\sqrt{\frac{s_{1}}{t_{1}}}\cdot\frac{1+\frac{\epsilon}{5}}{1-\frac{\epsilon}{5}}.
    \end{equation*}
    In other words,
    \begin{equation*}
        \Big|\frac{h(s_{1})}{h(t_{1})}-\sqrt{\frac{s_{1}}{t_{1}}}\Big|\leq\epsilon.
    \end{equation*}
    Since $\epsilon$ can be made arbitrarily small, this gives the desired growth rate estimate in Theorem~\ref{thm. growth rate of free boundary}.

    \textbf{Step 5: Location of $h(+\infty)$ under Gibbons' setting.} Now, let us rule out the possibility that $G:=\displaystyle\limsup_{R\to\infty}\epsilon_{R}R<+\infty$. Since $h(x_{1})$ is strictly increasing, we let $h(+\infty)$ be the limiting height, which is must satisfy $N(0)<h(+\infty)\leq G$. We claim that
    \begin{equation}\label{eq. claim that h(+infty)=G}
        h(+\infty)=G=\limsup_{R\to\infty}\epsilon_{R}R.
    \end{equation}
    To prove \eqref{eq. claim that h(+infty)=G}, we argue by contradiction. If it holds that $h(+\infty)<G$, then set
    \begin{equation*}
        v_{R}(y):=\frac{\Psi^{-1}\circ\psi(Ry)-Ry_{2}}{G},\quad\mbox{for all }y\in[-1,1]^{2}\cap\{\psi(Ry)>0\}.
    \end{equation*}
    Clearly, $-1\leq v_{R}(y)\leq0$. Applying Lemma~\ref{lem. limit description} to the first quadrant (where $|\nabla\psi|=\Lambda$ on $\partial\{\psi>0\}$) yields that up to a subsequence $R_{k}\to\infty$,
    \begin{equation*}
        \lim_{k\to\infty}\|v_{R_{k}}-v_{\infty}\|_{L^{\infty}([\frac{1}{4},\frac{3}{4}]\times[0,\frac{1}{4}]\cap\{\psi(R_{k}y)>0\})}=0.
    \end{equation*}
    Here, $v_{\infty}$ satisfies the equation $\mu^{2}v_{11}+v_{22}=0$ in $int([\frac{1}{4},\frac{3}{4}]\times[0,\frac{1}{4}])$, as well as the Neumann condition that $\partial_{2}v_{\infty}(t\vec{e}_{1})\equiv0$ for $t\in[\frac{1}{4},\frac{3}{4}]$. Furthermore, if $h(+\infty)<G$, then
    \begin{equation*}
        v_{\infty}(t\vec{e}_{1})\equiv-\frac{h(+\infty)}{G}\in(-1,0),\quad\mbox{for all }t\in[\frac{1}{4},\frac{3}{4}].
    \end{equation*}
    Then, the only possibility is that
    \begin{equation*}
        v_{\infty}\equiv-\frac{h(+\infty)}{G}\in(-1,0)\quad\mbox{everywhere in }[\frac{1}{4},\frac{3}{4}]\times[0,\frac{1}{4}].
    \end{equation*}
    By applying the interior Harnack principle (Lemma~\ref{lem. interior Harnack}) to $v_{R_{k}}$, we get the following observation: For every $\lambda>0$, as long as $k$ is sufficiently large, one has
    \begin{equation*}
        \psi(x)\geq\Psi\Big(x_{2}-(1-\sigma)G\Big),\quad\mbox{for all }x\in[-\frac{9}{10}R_{k},\frac{9}{10}R_{k}]\times[\frac{\lambda}{10}R_{k},\frac{9}{10}R_{k}].
    \end{equation*}
    Here, the constant $\sigma>0$ depends on the ratio $\frac{h(+\infty)}{G}$, and it is by default chosen so small that $\sigma G\leq\frac{G-N(0)}{100}$. Applying Lemma~\ref{lem. maximum variant} gives that
    \begin{equation*}
        \psi(x)\geq\Psi\Big(x_{2}-(1-\frac{\sigma}{640})G\Big)\quad\mbox{in }[-\frac{\lambda}{100}R_{k},\frac{\lambda}{100}R_{k}]^{2}\mbox{ for all sufficiently large }k.
    \end{equation*}
    Recall that we choose $\epsilon_{R}$ as small as possible in Definition~\ref{def. asymptotically flat}, then the above estimate yields that indeed $\epsilon_{r_{k}}r_{k}\leq(1-\frac{\sigma}{640})G<G$ for a subsequence $r_{k}=\frac{\lambda}{100}R_{k}\to\infty$, contradicting the definition of $G$. Then, we have verified \eqref{eq. claim that h(+infty)=G}.
    
    \textbf{Step 6: Contradicts the assumption that $h(x_{1})$ is increasing.} As we have verified $h(+\infty)=G<+\infty$ in \eqref{eq. claim that h(+infty)=G}, one can apply Remark~\ref{rmk. how to avoid h(infty)<infty}, and show that indeed there exists a sequence $s_{k}\to\infty$, such that
    \begin{equation*}
        h(s_{k})\leq(1-\sigma)\cdot G=(1-\sigma)\cdot h(+\infty).
    \end{equation*}
    This violates the assumption that $h(x_{1})$ is strictly increasing. Therefore, we have excluded the possibility that $G:=\displaystyle\limsup_{R\to\infty}\epsilon_{R}R<+\infty$, finishing the proof of Theorem~\ref{thm. growth rate of free boundary}.
\end{proof}

\appendix
\section{ODE analysis}\label{sec. appendix ODE}
Let us construct some non-negative sub-solutions or super-solutions to \eqref{eq. main simple}. The construction is based on the shear flow solution $\Psi(t)$ (see Definition~\ref{def. shear flow}).
\subsection{First integral of the shear flow}
In this subsection, study the shear flow $\Psi$ and study the properties of its deformations.
\begin{lemma}\label{lem. shear flow property}
    Let $\Psi(x_{2})$ be the shear flow solution satisfying Definition~\ref{def. shear flow}. Then, there is a function $H(z)>0$ with $H(+\infty)=V^{2}$, where $V$ coincides with that in (A6), such that
    \begin{equation}\label{eq. speed identity}
        \Psi_{t}=\sqrt{H(\Psi)}\mbox{ for }t>0.
    \end{equation}
\end{lemma}
\begin{proof}
    Multiplying $g(|\Psi_{t}|^{2},\Psi)\Psi_{t}$ on both sides of \eqref{eq. 1D solution Psi} implies
    \begin{equation*}
        \Big(g(|\Psi_{t}|^{2},\Psi)\Psi_{t}\Big)_{t}\cdot g(|\Psi_{t}|^{2},\Psi)\Psi_{t}=\mathcal{B}'(\Psi)\Psi_{t}.
    \end{equation*}
    Using the gradient assumption at the free boundary, one has
    \begin{equation*}
        g(|\Psi_{t}|^{2},\Psi)^{2}\cdot|\Psi_{t}|^{2}=2\mathcal{B}(\Psi)+C.
    \end{equation*}
    Since it was assumed in (A6) that $\mathcal{B}(0)=\frac{1}{2}\Lambda^{2}\cdot g(\Lambda^{2},0)^{2}$, we indeed have
    \begin{equation*}
        g(|\Psi_{t}|^{2},\Psi)^{2}\cdot|\Psi_{t}|^{2}=2\mathcal{B}(\Psi).
    \end{equation*}
    Denote $F(T,z):=g(T,z)^{2}\cdot T$. By (A4), at least in the sub-sonic regime, $F(T,z)$ is invertible when $z$ is fixed. Then, we let $\displaystyle\Big(F\Big|_{z}\Big)^{-1}$ be the inverse of $F(T,z)$ (mapped to sub-sonic region $[0,Q_{sound}^{2}]$) when the variable $z$ is fixed, i.e.
    \begin{equation*}
        \Big(F\Big|_{z}\Big)^{-1}(a)=t\quad\mbox{if and only if }F(t,z)=a\mbox{ and }t\in[0,Q_{sound}^{2}].
    \end{equation*}
    Let us assume that $H(\Psi)$ is given by
    \begin{equation}\label{eq. H(z) initial definition}
        H(z):=\Big(F\Big|_{z}\Big)^{-1}\big(2\mathcal{B}(z)\big),\quad\mbox{i.e. }g\big(H(z),z\big)^{2}H=2\mathcal{B}(z).
    \end{equation}
    Then, it follows that \eqref{eq. speed identity} holds and $H(0)=\Lambda^{2}$. Sending $z\to+\infty$ and using the assumption (A6) gives
    \begin{equation*}
        \lim_{z\to+\infty}g\big(H(z),z\big)^{2}\cdot H(z)=2\mathcal{B}(+\infty)=V^{2}\cdot g(V^{2},+\infty)^{2}.
    \end{equation*}
    By (A4) and the assumption that $\Psi$ is sub-sonic (see Definition~\ref{def. shear flow}), we must have $H(+\infty)=V^{2}$. This completes the proof of Lemma~\ref{lem. shear flow property}.
\end{proof}

\subsection{Deformation of type I} In the following Lemma, we construct two deformations of $\Psi(t)$.
\begin{lemma}\label{lem. barrier A}
    There exist two uniform constants $\mathfrak{c}_{1}\ll1\ll\mathfrak{C}_{2}$ such that the following holds for every $\rho>0$:
    
    \textbf{(A) Sub-solution:} There exists a function $\check{w}_{\rho}(t)$ such that:
    \begin{itemize}
        \item[(A1)] For $0\leq t\leq\mathfrak{c}_{1}\rho$, one has
        \begin{equation*}
            \frac{d}{dt}\Big(g\big(|\frac{d}{dt}\check{w}_{\rho}|^{2},\check{w}_{\rho}\big)\frac{d}{dt}\check{w}_{\rho}\Big)-\frac{1}{\rho-t}g\big(|\frac{d}{dt}\check{w}_{\rho}|^{2},\check{w}_{\rho}\big)\frac{d}{dt}\check{w}_{\rho}>\frac{\mathcal{B}'(\check{w}_{\rho})}{g\big(|\frac{d}{dt}\check{w}_{\rho}|^{2},\check{w}_{\rho}\big)};
        \end{equation*}
        \item[(A2)] For $0\leq t\leq\mathfrak{c}_{1}\rho$, one has
        \begin{equation*}
            \Psi\Big(t+\frac{t^{2}}{\mathfrak{C}_{2}\rho}\Big)\leq\check{w}_{\rho}(t)\leq\Psi\Big(t+\frac{\mathfrak{C}_{2}t^{2}}{\rho}\Big).
        \end{equation*}
    \end{itemize}
    
    \textbf{(B) Super-solution:} There exists a function $\hat{w}_{\rho}(t)$ such that:
    \begin{itemize}
        \item[(B1)] For $0\leq t\leq\mathfrak{c}_{1}\rho$, one has
        \begin{equation*}
            \frac{d}{dt}\Big(g\big(|\frac{d}{dt}\hat{w}_{\rho}|^{2},\hat{w}_{\rho}\big)\frac{d}{dt}\hat{w}_{\rho}\Big)+\frac{1}{\rho+t}g\big(|\frac{d}{dt}\hat{w}_{\rho}|^{2},\hat{w}_{\rho}\big)\frac{d}{dt}\hat{w}_{\rho}<\frac{\mathcal{B}'(\hat{w}_{\rho})}{g\big(|\frac{d}{dt}\hat{w}_{\rho}|^{2},\hat{w}_{\rho}\big)};
        \end{equation*}
        \item[(B2)] For $0\leq t\leq\mathfrak{c}_{1}\rho$, one has
        \begin{equation*}
            \Psi\Big(t-\frac{\mathfrak{C}_{2}t^{2}}{\rho}\Big)\leq\hat{w}_{\rho}(t)\leq\Psi\Big(t-\frac{t^{2}}{\mathfrak{C}_{2}\rho}\Big).
        \end{equation*}
    \end{itemize}
\end{lemma}
\begin{proof}
    \textbf{Step 1: Construction of $\check{w}_{\rho}(t)$ and $\hat{w}_{\rho}(t)$.} Let $\check{s}=\check{w}_{\rho}(t)$ or $\hat{s}=\hat{w}_{\rho}(t)$. We would like
    \begin{equation}\label{eq. w(check/hat) 1st ODE}
        \frac{d}{dt}\check{w}_{\rho}=\sqrt{\check{T}(\check{w}_{\rho})},\quad\frac{d}{dt}\hat{w}_{\rho}=\sqrt{\hat{T}(\hat{w}_{\rho})},
    \end{equation}
    where $\check{T}(s)$ and $\hat{T}(s)$ are defined as
    \begin{equation*}
        \check{T}(s)=F\Big|_{s}^{-1}\Big(2\mathcal{B}(s)+\mathcal{E}(s)\Big),\quad\hat{T}(s)=F\Big|_{s}^{-1}\Big(2\mathcal{B}(s)-\mathcal{E}(s)\Big),
    \end{equation*}
    where the correction term $\mathcal{E}(s)$ is given below for some sufficiently large constant $\mathfrak{C}_{3}$:
    \begin{equation*}
        \mathcal{E}(s):=\frac{\mathfrak{C}_{3}}{\rho}\int_{0}^{s}\sqrt{H(\sigma)}d\sigma.
    \end{equation*}
    Then, for $\check{w}_{\rho}$ and $\hat{w}_{\rho}$, we require
    \begin{equation}\label{eq. sub+sup 1st ODE}
    \left\{\begin{aligned}
        \frac{1}{2}\frac{d}{d\check{s}}\Big(g\big(\check{T}(\check{s}),\check{s}\big)^{2}\cdot\check{T}(\check{s})\Big)-\frac{2}{\rho}g\big(\check{T}(\check{s}),\check{s}\big)^{2}\sqrt{\check{T}(\check{s})}>&\frac{d}{d\check{s}}\mathcal{B}(\check{s}),\\
        \frac{1}{2}\frac{d}{d\hat{s}}\Big(g\big(\hat{T}(\hat{s}),\hat{s}\big)^{2}\cdot\hat{T}(\hat{s})\Big)+\frac{2}{\rho}g\big(\hat{T}(\hat{s}),\hat{s}\big)^{2}\sqrt{\hat{T}(\hat{s})}<&\frac{d}{d\hat{s}}\mathcal{B}(\hat{s}).
    \end{aligned}\right.
    \end{equation}
    When $\mathfrak{c}_{4}\leq\frac{1}{2}$ is sufficiently small, then
    \begin{equation}\label{eq. E and H ratio}
        \mathcal{E}(s)\leq\frac{1}{2}H(s),\quad\mbox{for all }0\leq s\leq\Psi(\mathfrak{c}_{4}\rho).
    \end{equation}
    Then, the first line of \eqref{eq. sub+sup 1st ODE} follows from
    \begin{equation*}
        \frac{1}{2}\frac{d}{d\check{s}}\Big(g\big(\check{T}(\check{s}),\check{s}\big)^{2}\cdot\check{T}(\check{s})\Big)-\frac{d}{d\check{s}}\mathcal{B}(\check{s})=\frac{1}{2}\frac{d}{d\check{s}}\mathcal{E}(s)=\frac{\mathfrak{C}_{3}}{2\rho}\sqrt{H(s)}>\frac{2}{\rho}g\big(\check{T}(\check{s}),\check{s}\big)^{2}\sqrt{\check{T}(\check{s})}.
    \end{equation*}
    The second line of \eqref{eq. sub+sup 1st ODE} can be proven similarly.

    \textbf{Step 2: Horizontal error estimate.} Let $\check{t}(\check{s})$ and $\hat{t}(\hat{s})$ be the inverse functions of $\check{w}_{\rho}$ and $\hat{w}_{\rho}$, respectively. Precisely:
    \begin{equation*}
        \check{w}_{\rho}\big(\check{t}(\check{s})\big)=\check{s},\quad\hat{w}_{\rho}\big(\hat{t}(\hat{s})\big)=\hat{s}.
    \end{equation*}
    It then follows from \eqref{eq. w(check/hat) 1st ODE} that
    \begin{equation*}
        \frac{d\check{t}}{ds}=\frac{1}{\sqrt{\check{T}(s)}}\quad\mbox{and }\frac{d\hat{t}}{ds}=\frac{1}{\sqrt{\hat{T}(s)}}.
    \end{equation*}
    By the construction of $\check{T}$, one has
    \begin{equation*}
        \frac{d\check{t}}{ds}-\frac{1}{\sqrt{H(s)}}\sim-\frac{\mathcal{E}(s)}{H(s)^{3/2}}\quad\mbox{and }\frac{d\hat{t}}{ds}-\frac{1}{\sqrt{H(s)}}\sim\frac{\mathcal{E}(s)}{H(s)^{3/2}}.
    \end{equation*}
    Integrating this inequality and using the fact $\displaystyle\frac{\mathcal{E}(s)}{H(s)^{3/2}}\sim\frac{s}{\rho}$ yields that
    \begin{equation}\label{eq. w and Psi accurate error}
        \check{w}_{\rho}^{-1}(s)-\Psi^{-1}(s)\sim-\frac{s^{2}}{\rho}\quad\mbox{and }\hat{w}_{\rho}^{-1}(s)-\Psi^{-1}(s)\sim\frac{s^{2}}{\rho}.
    \end{equation}

    \textbf{Step 3: Ending statement.} By \eqref{eq. E and H ratio}, one has
    \begin{equation*}
        \Psi(t)\leq\check{w}_{\rho}(t)\leq\Psi(\frac{3}{2}t)\mbox{ and }\Psi(\frac{1}{2}t)\leq\hat{w}_{\rho}(t)\leq\Psi(t),\quad\mbox{for }0\leq t\leq\frac{1}{2}\mathfrak{c}_{4}\rho.
    \end{equation*}
    By choosing $\mathfrak{c}_{1}\ll\mathfrak{c}_{4}$ small, one has
    \begin{equation*}
        \frac{1}{\rho\pm t}\leq\frac{2}{\rho},\quad\mbox{for all }0\leq t\leq\mathfrak{c}_{1}\rho,
    \end{equation*}
    so the inequalities (A1) and (B1) follow from \eqref{eq. sub+sup 1st ODE}. Moreover, the inequalities (A2) and (B2) follow from \eqref{eq. w and Psi accurate error} and the observation $s\sim t$. Then, we have finished the proof of Lemma~\ref{lem. barrier A}.
\end{proof}

\subsection{Deformation of type II} Now, let us consider another type of deformation.
\begin{lemma}\label{lem. barrier B}
    There exist two uniform constants $\mathfrak{c}_{3}\ll1\ll\mathfrak{C}_{6}$ such that the following holds for every $\rho>0$ and $m\in[0,\mathfrak{c}_{5}\rho]$: There exists a number $-m-\frac{\mathfrak{C}_{6}m^{2}}{\rho}\leq t_{-}\leq-m-\frac{m^{2}}{\mathfrak{C}_{6}\rho}$, and a function $w_{\rho,m}(t)$ defined for $t\in[t_{-},0]$, such that
    \begin{itemize}
        \item[(1)] One has $0<w_{\rho,m}(t)<\Psi(m)$ in $(t_{-},0)$, with $w_{\rho,m}(0)=\Psi(m)$ and $w_{\rho,m}(t_{-})=0$.
        \item[(2)] The following ODE inequality holds in $(t_{-},0)$:
        \begin{equation*}
            \frac{d}{dt}\Big(g\big(|\frac{d}{dt}w_{\rho,m}|^{2},w_{\rho,m}\big)\frac{d}{dt}w_{\rho,m}\Big)-\frac{1}{\rho-t}g\big(|\frac{d}{dt}w_{\rho,m}|^{2},w_{\rho,m}\big)\frac{d}{dt}w_{\rho,m}>\frac{\mathcal{B}'(w_{\rho,m})}{g\big(|\frac{d}{dt}w_{\rho,m}|^{2},w_{\rho,m}\big)};
        \end{equation*}
        \item[(3)] The following estimate holds in $[t_{-},0]$:
        \begin{equation*}
            \Psi\Big(m+t+\frac{t^{2}}{\mathfrak{C}_{6}\rho}\Big)\leq w_{\rho,m}(t)\leq\Psi\Big(m+t+\frac{\mathfrak{C}_{6}t^{2}}{\rho}\Big),\quad\mbox{for all }t_{-}\leq t\leq0.
        \end{equation*}
    \end{itemize}
\end{lemma}
\begin{proof}
    The method of Lemma~\ref{lem. barrier A} can be applied here with a slight modification. Let
    \begin{equation}\label{eq. widetilde Psi}
        \widetilde{\Psi}(\widetilde{t})=\Big(\Psi(m)-\Psi(m-\widetilde{t})\Big)\cdot\chi_{[0,m]}+\Big(\Psi(m)+\Lambda\cdot(\widetilde{t}-m)\Big)\cdot\chi_{(m,\infty)}.
    \end{equation}
    Using the function $H(\cdot)$ in \eqref{eq. H(z) initial definition}, one can verify that for all $\widetilde{t}\geq0$,
    \begin{equation*}
        \frac{d}{d\widetilde{t}}\widetilde{\Psi}=\sqrt{\widetilde{H}(\widetilde{\Psi})},\quad\mbox{where }\widetilde{H}(\widetilde{s})=H\big(\Psi(m)-\widetilde{s}\big)\cdot\chi_{[0,\Psi(m)]}+\Lambda^{2}\cdot\chi_{(\Psi(m),+\infty)}.
    \end{equation*}
    Let us apply Lemma~\ref{lem. barrier A} part \textbf{(B)} to $\widetilde{\Psi}(\widetilde{t})$ and let $\hat{\widetilde{w}}_{\rho}(\widetilde{t})$ be the corresponding super-solution, which is a deformation of $\widetilde{\Psi}$. Then,
    \begin{equation}\label{eq. up side down ODE}\begin{aligned}
        &\frac{d}{dt}\Big(g\big(|\frac{d}{dt}\hat{\widetilde{w}}_{\rho}|^{2},\Psi(m)-\hat{\widetilde{w}}_{\rho}\big)\frac{d}{dt}\hat{\widetilde{w}}_{\rho}\Big)+\frac{1}{\rho+t}g\big(|\frac{d}{dt}\hat{w}_{\rho}|^{2},\Psi(m)-\hat{\widetilde{w}}_{\rho}\big)\frac{d}{dt}\hat{w}_{\rho}\\
        <&\frac{\mathcal{B}'(\Psi(m)-\hat{\widetilde{w}}_{\rho})}{g\big(|\frac{d}{dt}\hat{\widetilde{w}}_{\rho}|^{2},\Psi(m)-\hat{\widetilde{w}}_{\rho}\big)},\quad\mbox{as long as }0<\hat{\widetilde{w}}_{\rho}<\Psi(m).
    \end{aligned}\end{equation}
    Moreover, we obtain from Lemma~\ref{lem. barrier A} part (B2) that
    \begin{equation}\label{eq. up side down deformation estimate}
        \widetilde{\Psi}\Big(\widetilde{t}-\frac{\mathfrak{C}_{2}\widetilde{t}^{2}}{\rho}\Big)\leq\hat{\widetilde{w}}_{\rho}(\widetilde{t})\leq\widetilde{\Psi}\Big(\widetilde{t}-\frac{\widetilde{t}^{2}}{\mathfrak{C}_{2}\rho}\Big),\quad\mbox{for all }0\leq\widetilde{t}\leq\mathfrak{c}_{1}\rho.
    \end{equation}
    Let $-t_{-}$ be the smallest positive number $\widetilde{t}$ such that $\hat{\widetilde{w}}_{\rho}(\widetilde{t})=\widetilde{\Psi}(m)=\Psi(m)$. By letting $\mathfrak{c}_{5}\ll\mathfrak{c}_{1}\ll1\ll\mathfrak{C}_{2}\ll\mathfrak{C}_{6}$, we see that as long as $m\in[0,\mathfrak{c}_{5}\rho]$, it must hold that
    \begin{equation*}
        m+\frac{m^{2}}{\mathfrak{C}_{6}\rho}\leq-t_{-}\leq m+\frac{\mathfrak{C}_{6}m^{2}}{\rho}.
    \end{equation*}
    Then, let us define $w_{\rho,m}(t)$ for $t\in[t_{-},0]$ as follows:
    \begin{equation*}
        w_{\rho,m}(t):=\Psi(m)-\hat{\widetilde{w}}_{\rho}(-t).
    \end{equation*}
    Then, the ODE inequality in part (2) follows from \eqref{eq. up side down ODE}. Since $\hat{\widetilde{w}}_{\rho}(-t)$ is a deformation of $\Psi$, the estimate in part (1) follows immediately. Finally, from \eqref{eq. widetilde Psi} and \eqref{eq. up side down deformation estimate}, we see that for all $t\in[t_{-},0]$,
    \begin{equation*}
        w_{\rho,m}(t)\leq\Psi(m)-\widetilde{\Psi}\Big(-t-\frac{\mathfrak{C}_{2}t^{2}}{\rho}\Big)\leq\Psi\Big(m+t+\frac{\mathfrak{C}_{2}t^{2}}{\rho}\Big)\leq\Psi\Big(m+t+\frac{\mathfrak{C}_{6}t^{2}}{\rho}\Big).
    \end{equation*}
    The other side of the estimate in part (3) can be argued similarly. Then, we have finished the proof of Lemma~\ref{lem. barrier B}.
\end{proof}

\section{Maximum principles}\label{sec. appendix maximum principle}
In this section, we prove the following two maximum principles and give several applications. The main difference with the traditional maximum principle is that we translate the barrier function in $\mathbb{R}^{2}$ rather than adding a constant to it.
\begin{lemma}\label{lem. maximum principle for Harnack}
    There exist uniform constants $\overline{\delta},\lambda\in(0,1)$, such that the following fact holds for all $r>0$ and $\delta\leq\overline{\delta}$:
    \begin{itemize}
        \item[(A)] \textbf{Lower bound increases.} Assume that
        \begin{equation*}
            \left\{\begin{aligned}
                &\psi\mbox{ solves \eqref{eq. main simple} and satisfies }\psi(x)\geq\Psi(x_{2})&\mbox{in }&[-r,r]^{2},\\
                &\psi(x)\geq\Psi(x_{2}+\delta r)&\mbox{in }&[-r,r]\times[\lambda r,r].
            \end{aligned}\right.
        \end{equation*}
        Then, one has
        \begin{equation*}
            \psi(x)\geq\Psi(x_{2}+\frac{\delta r}{64})\quad\mbox{in }[-\lambda r,\lambda r]^{2}.
        \end{equation*}
        \item[(B)] \textbf{Upper bound decreases.} Assume that
        \begin{equation*}
            \left\{\begin{aligned}
                &\psi\mbox{ solves \eqref{eq. main simple} and satisfies }\psi(x)\leq\Psi(x_{2})&\mbox{in }&[-r,r]^{2},\\
                &\psi(x)\leq\Psi(x_{2}-\delta r)&\mbox{in }&[-r,r]\times[\lambda r,r].
            \end{aligned}\right.
        \end{equation*}
        Then, one has
        \begin{equation*}
            \psi(x)\leq\Psi(x_{2}-\frac{\delta r}{64})\quad\mbox{in }[-\lambda r,\lambda r]^{2}.
        \end{equation*}
    \end{itemize}
\end{lemma}
\begin{lemma}\label{lem. maximum principle for forced decay}
    There exist uniform constants $\overline{\delta},\lambda\in(0,1)$, such that the following fact holds for all $r>0$ and $\delta\leq\widetilde{\delta}$: Assume that
    \begin{equation*}
        \psi\mbox{ satisfies \eqref{eq. quasilinear structure} and }\Psi(x_{2}-\delta r)\leq\psi(x)\leq\Psi(x_{2})\mbox{ in }(-r,r)\times(0,r).
    \end{equation*}
    Then, one has $\psi\leq\Psi(x_{2}-\frac{63}{64}\delta r)$ in $[-\lambda r,\lambda r]\times[0,\lambda r]$
\end{lemma}
\subsection{On the proof of Lemma~\ref{lem. maximum principle for Harnack}}
In this subsection, we give the proof of Lemma~\ref{lem. maximum principle for Harnack}.
\begin{proof}[Proof of Lemma~\ref{lem. maximum principle for Harnack}]
    We mainly prove part (A), leaving the proof of part (B) to the readers. The proof is divided into 5 steps.
    
    \textbf{Step 1: Geometric set-up.} Let $\mathfrak{c}_{1}$ and $\mathfrak{C}_{2}$ be the constants given in Lemma~\ref{lem. barrier A}. Denote:
    \begin{equation}\label{eq. maximum principle A, constants requirements}
        \lambda=(128\mathfrak{C}_{2})^{-1/2},\quad\overline{\delta}=\min\{\mathfrak{c}_{1},\lambda,1\},\quad\rho=\frac{r}{\delta}.
    \end{equation}
    Fix an $x_{1}\in[-\frac{r}{10},\frac{r}{10}]$, then we set
    \begin{equation*}
        q_{h}=x_{1}\vec{e}_{1}+h\vec{e}_{2}\in\mathbb{R}^{2}\mbox{ for each }h\in\mathbb{R},
    \end{equation*}
    and consider the following two domains:
    \begin{equation*}
        \left\{\begin{aligned}
            \check{\mathcal{D}}_{h}:=&\Big\{x\in\mathbb{R}^{2}:-2\lambda r\leq\rho-|x-q_{h}|\leq2\lambda r,\ \angle{\big(x-q_{h},-\vec{e}_{2}\big)}\leq\sin^{-1}(\frac{\delta}{2})\Big\},\\
            \hat{\mathcal{D}}_{h}:=&\Big\{x\in\mathbb{R}^{2}:-2\lambda r\leq|x-q_{h}|-\rho\leq2\lambda r,\ \angle{\big(x-q_{h},\vec{e}_{2}\big)}\leq\sin^{-1}(\frac{\delta}{2})\Big\}.
        \end{aligned}\right.
    \end{equation*}
    Let
    \begin{equation*}
        \left\{\begin{aligned}
            &\check{\phi}_{h}(x):=\check{w}_{\rho}\big(\rho-|x-q_{h}|\big)&\mbox{for }x\in\check{\mathcal{D}}_{h},\\
            &\hat{\phi}_{h}(x)=\hat{w}_{\rho}\big(|x-q_{h}|-\rho\big)&\mbox{for }x\in\hat{\mathcal{D}}_{h}.
        \end{aligned}\right.
    \end{equation*}
    It is easy to verify that when $|x_{1}|\leq\frac{r}{10}$, one always has
    \begin{equation*}
        \left\{\begin{aligned}
            &\check{\mathcal{D}}_{h}\subseteq[-r,r]^{2}&\mbox{for all }&\rho-\frac{r}{4}\leq h\leq\rho+\frac{r}{4},\\
            &\hat{\mathcal{D}}_{h}\subseteq[-r,r]^{2}&\mbox{for all }&-\rho-\frac{r}{4}\leq h\leq-\rho+\frac{r}{4}.
        \end{aligned}\right.
    \end{equation*}
    From now on, we only consider case (A). In the following process, we set $h=\rho+\frac{r}{4}$ at the beginning and decrement $h$ until $h=\rho-\frac{\delta r}{64}$. For simplicity, denote
    \begin{equation*}
        t=\rho-|x-q_{h}|\in[-2\lambda r,2\lambda r],\quad\theta=\angle{\big(x-q_{h},-\vec{e}_{2}\big)}\leq\sin^{-1}(\frac{\delta}{2}).
    \end{equation*}
    Clearly, $\check{\phi}_{h}(x)=\check{w}_{\rho}(t)$ in $\check{\mathcal{D}}_{h}$. By \eqref{eq. maximum principle A, constants requirements}, one has $|t|\leq2\lambda r\leq\mathfrak{c}_{1}\rho$ in $\check{\mathcal{D}}_{h}$. Then, by Lemma~\ref{lem. barrier A}, one has
    \begin{equation}\label{eq. what lemma 3.1 implies in Maximum principle A}
        \Psi\big(t+\frac{t^{2}}{\mathfrak{C}_{2}\rho}\big)\leq\check{\phi}_{h}(x)\leq\Psi\big(t+\frac{\mathfrak{C}_{2}t^{2}}{\rho}\big)\quad\mbox{in }\check{\mathcal{D}}_{h}.
    \end{equation}
    
    \textbf{Step 2: Initial comparison.} When $h=\rho+\frac{r}{4}$, one has
    \begin{equation*}
        t=\rho-|x-q_{h}|\leq\rho-(h-x_{2})=x_{2}-\frac{r}{4}.
    \end{equation*}
    Since $|t|\leq2\lambda r$ in $\check{\mathcal{D}}_{h}$, it follows from \eqref{eq. maximum principle A, constants requirements} that
    \begin{equation*}
        x_{2}-t-\frac{\mathfrak{C}_{2}t^{2}}{\rho}\geq\frac{r}{4}-\frac{\mathfrak{C}_{2}\cdot(2\lambda r)^{2}}{\rho}\geq\frac{r}{4}-\frac{\delta r}{32}\geq\frac{r}{8}>0.
    \end{equation*}
    This estimate, together with \eqref{eq. what lemma 3.1 implies in Maximum principle A}, implies
    \begin{equation*}
        \check{\phi}_{h}(x)\leq\Psi\Big(t+\frac{\mathfrak{C}_{2}t^{2}}{\rho}\Big)\leq\Psi(x_{2})\quad\mbox{in }\check{\mathcal{D}}_{h}.
    \end{equation*}
    In other words,
    \begin{equation*}
        \check{\phi}_{h}(x)\leq\Psi(x_{2})\leq\psi(x)\mbox{ in }\check{\mathcal{D}}_{h},\quad\mbox{when }h=\rho+\frac{r}{4}.
    \end{equation*}

    \textbf{Step 3: Boundary comparison.} When $\rho-\frac{\delta r}{64}\leq h\leq\rho+\frac{r}{4}$, we decompose $\partial\check{\mathcal{D}}_{h}$ into three parts:
    \begin{itemize}
        \item Outer boundary: $\partial_{out}\check{\mathcal{D}}_{h}:=\partial\check{\mathcal{D}}_{h}\cap\{t=-2\lambda r\}$,
        \item Inner boundary: $\partial_{in}\check{\mathcal{D}}_{h}:=\partial\check{\mathcal{D}}_{h}\cap\{t=2\lambda r\}$,
        \item Side boundary: $\partial_{side}\check{\mathcal{D}}_{h}:=\partial\check{\mathcal{D}}_{h}\cap\{\theta=\sin^{-1}(\frac{\delta}{2})\}$.
    \end{itemize}
    Our goal is to show that there exists a function $p(\delta,r)>0$, such that
    \begin{equation}\label{eq. maximum principle A boudnary comparison}
        \left\{\begin{aligned}
            &\psi(x)\geq\check{\phi}_{h}(x)&\mbox{on }&\partial\check{\mathcal{D}}_{h},\\
            &\psi-\check{\phi}_{h}\geq p(\delta,r)>0&\mbox{on }&\partial\check{\mathcal{D}}_{h}\cap\{\check{\phi}_{h}>0\},
        \end{aligned}\right.\quad\mbox{for all }\rho-\frac{\delta r}{64}\leq h\leq\rho+\frac{r}{4}.
    \end{equation}
    The estimate \eqref{eq. maximum principle A boudnary comparison} is to be verified on three types of boundary points separately.
    \begin{itemize}
        \item On the outer boundary $\partial_{out}\check{\mathcal{D}}_{h}=\partial\check{\mathcal{D}}_{h}\cap\{t=-2\lambda r\}$, we have
        \begin{equation*}
            \check{\phi}_{h}(x)=\check{w}_{\rho}(-2\lambda t)=0\leq\psi(x),
        \end{equation*}
        and thus \eqref{eq. maximum principle A boudnary comparison} is verified on $\partial_{out}\check{\mathcal{D}}_{h}$.
        \item On the inner boundary $\partial_{in}\check{\mathcal{D}}_{h}=\partial\check{\mathcal{D}}_{h}\cap\{t=2\lambda r\}$, we first notice that as $\rho-\frac{\delta r}{64}\leq h\leq\rho+\frac{r}{4}$, it holds that
        \begin{equation}\label{eq. maximum principle A inner boundary}
            2\lambda r=t=\rho-|x-q_{h}|\leq\rho-(h-x_{2})\leq x_{2}+\frac{\delta r}{64}\quad\mbox{on }\partial_{in}\check{\mathcal{D}}_{h}.
        \end{equation}
        As a result, $\partial_{in}\check{\mathcal{D}}_{h}\subseteq[-r,r]\times[\lambda r,r]$, and thus
        \begin{equation*}
            \psi(x)\geq\Psi(x_{2}+\delta r)>0\mbox{ on }\partial_{in}\check{\mathcal{D}}_{h}.
        \end{equation*}
        Moreover, it follows from \eqref{eq. maximum principle A, constants requirements} and \eqref{eq. maximum principle A inner boundary} that
        \begin{equation*}
            x_{2}+\delta r-2\lambda r-\frac{\mathfrak{C}_{2}}{\rho}(2\lambda r)^{2}\geq\frac{63}{64}\delta r-4\mathfrak{C}_{2}\lambda^{2}\delta r=\frac{61}{64}\delta r>0.
        \end{equation*}
        Applying \eqref{eq. what lemma 3.1 implies in Maximum principle A} yields that for all $x\in\partial_{in}\check{\mathcal{D}}_{h}$, it holds that
        \begin{align*}
            \check{\phi}_{h}(x)=&\check{w}_{\rho}(\lambda r)\leq\Psi\Big(\lambda r+\frac{\mathfrak{C}_{2}}{\rho}(\lambda r)^{2}\Big)\leq\Psi(x_{2}+\delta r-\frac{61}{64}\delta r)\\
            \leq&\Psi(x_{2}+\delta r)-p(\delta,r)\leq\psi(x)-p(\delta,r),
        \end{align*}
        and thus \eqref{eq. maximum principle A boudnary comparison} is verified on $\partial_{in}\check{\mathcal{D}}_{h}$.
        \item On the side boundary $\partial_{side}\check{\mathcal{D}}_{h}=\partial\check{\mathcal{D}}_{h}\cap\{\theta=\sin^{-1}(\frac{\delta}{2})\}$, one has
        \begin{equation*}
            (h-x_{2})=(\rho-t)\cdot\cos{\theta}.
        \end{equation*}
        Using the fact $h\geq\rho-\frac{\delta r}{64}$, one has
        \begin{equation*}
            x_{2}\geq(1-\cos{\theta})\rho+t\cdot\cos{\theta}-\frac{\delta r}{64}.
        \end{equation*}
        Moreover, for $x\in\partial_{side}\check{\mathcal{D}}_{h}$, one has
        \begin{equation*}
            |t|\leq2\lambda r\mbox{ and }1-\cos{\theta}\geq\frac{1}{2}(\sin{\theta})^{2}=\frac{\delta^{2}}{8}.
        \end{equation*}
        Consequently, for $x\in\partial_{side}\check{\mathcal{D}}_{h}$, it follows from \eqref{eq. maximum principle A, constants requirements} that
        \begin{align*}
            x_{2}-t-\frac{\mathfrak{C}_{2}t^{2}}{\rho}\geq&(1-\cos{\theta})\cdot(\rho-t)-\frac{\delta r}{64}-\frac{\mathfrak{C}_{2}(2\lambda r)^{2}}{\rho}\\
            \geq&\frac{\delta^{2}}{8}\cdot(\frac{r}{\delta}-2\lambda r)-\frac{\delta r}{64}-\frac{\delta r}{32}.
        \end{align*}
        By \eqref{eq. maximum principle A, constants requirements}, one has $\frac{r}{\delta}-2\lambda r\geq\frac{r}{2\delta}$ and thus
        \begin{equation*}
            x_{2}-t-\frac{\mathfrak{C}_{2}t^{2}}{\rho}\geq\frac{\delta r}{64}>0\mbox{ on }\partial_{side}\check{\mathcal{D}}_{h}.
        \end{equation*}
        Applying \eqref{eq. what lemma 3.1 implies in Maximum principle A} yields that
        \begin{equation*}
            \check{\phi}_{h}(x)\leq\Psi\Big(t+\frac{\mathfrak{C}_{2}t^{2}}{\rho}\Big)\leq\Psi(x_{2}-\frac{\delta r}{64})\leq\psi(x).
        \end{equation*}
        Furthermore, if $\check{\phi}_{h}(x)>0$, i.e. $t>0$, then we even have $x_{2}-\frac{\delta r}{64}>0$. Then,
        \begin{equation*}
            \check{\phi}_{h}(x)\leq\Psi(x_{2}-\frac{\delta r}{64})\leq\Psi(x_{2})-p(\delta,r)\leq\psi(x)-p(\delta,r),
        \end{equation*}and thus \eqref{eq. maximum principle A boudnary comparison} is verified on $\partial_{side}\check{\mathcal{D}}_{h}$.
    \end{itemize}
    In conclusion, we have verified \eqref{eq. maximum principle A boudnary comparison}.

    \textbf{Step 4: Open-and-closed argument.} Let
    \begin{equation*}
        h^{*}=\inf\Big\{\chi\in\big[\rho-\frac{\delta r}{64},\rho+\frac{r}{2}\big]:\check{\phi}_{h}\leq\psi\mbox{ in }\check{\mathcal{D}}_{h}\mbox{ for all }\chi\leq h\leq\rho+\frac{r}{2}\Big\}.
    \end{equation*}
    By \textbf{Step 2}, we know that $h^{*}$ exists. Moreover, by the continuity of $\psi$, the condition $\check{\phi}_{h}\leq\psi$ in $\check{\mathcal{D}}_{h}$ is a closed condition in $h$, meaning that
    \begin{equation}\label{eq. closed condition}
        \check{\phi}_{h^{*}}\leq\psi\mbox{ in }\check{\mathcal{D}}_{h^{*}}.
    \end{equation}
    We intend to show $h^{*}=\rho-\frac{\delta r}{64}$ by contradiction.
    
    Suppose that $h^{*}>\rho-\frac{\delta r}{64}$, then there exists a sequence $\{(h^{k},x^{k})\}_{k=0}^{\infty}$ satisfying
    \begin{equation*}
        \rho-\frac{\delta r}{64}\leq h^{k}<h^{*},\quad h^{k}\to h^{*},\quad x^{k}\in\check{\mathcal{D}}_{h^{k}},\quad\check{\phi}_{h^{k}}(x^{k})>\psi(x^{k}).
    \end{equation*}
    Clearly, up to a subsequence (still labeled by $k$), one has
    \begin{equation}\label{eq. x^* as limit for open-and-closed}
        x^{k}\to x^{*}\in\check{\mathcal{D}}_{h^{*}},\quad\mbox{with }\check{\phi}_{h^{*}}(x^{*})\geq\psi(x^{*})\mbox{ and }x^{*}\in\overline{\check{\mathcal{D}}_{h^{*}}\cap\{\check{\phi}_{h^{*}}>0\}}.
    \end{equation}
    This, combined with \eqref{eq. closed condition}, implies that
    \begin{equation}\label{eq. open-and-closed local max}
        \sup_{x\in\check{\mathcal{D}}_{h^{*}}}(\check{\phi}_{h^{*}}-\psi)=\check{\phi}_{h^{*}}(x^{*})-\psi(x^{*})=0.
    \end{equation}
    
    We first claim that $x^{*}\notin\partial\check{\mathcal{D}}_{h^{*}}$. In fact, suppose that $x^{*}\notin\partial\check{\mathcal{D}}_{h^{*}}$, then it follows from \eqref{eq. x^* as limit for open-and-closed} that $x^{*}\in\overline{\{\check{\phi}_{h^{*}}>0\}}\cap\partial\check{\mathcal{D}}_{h^{*}}$. By the continuity of $\psi$ and by \eqref{eq. maximum principle A boudnary comparison}, we have that
    \begin{equation*}
        \check{\phi}_{h^{*}}(x^{*})-\psi(x^{*})\leq-p(\delta,r)<0,
    \end{equation*}
    which is a contradiction. Therefore, $x^{*}\in int(\check{\mathcal{D}}_{h^{*}})$. Using \eqref{eq. closed condition}, we further have
    \begin{equation*}
        x^{*}\in int(\check{\mathcal{D}}_{h^{*}})\cap\overline{\{\check{\phi}_{h^{*}}>0\}}\cap\overline{\{\psi>0\}}.
    \end{equation*}
    Next, by Lemma~\ref{lem. barrier A}, we see that $\check{\phi}_{h^{*}}$ is a strict sub-solution to \eqref{eq. main simple} in the region $\check{\mathcal{D}}_{h^{*}}\cap\{\check{\phi}_{h^{*}}>0\}$, then one can exclude the possibility that $\check{\phi}_{h^{*}}(x^{*})>0$ by applying the second derivative test to \eqref{eq. open-and-closed local max}. Consequently, $\psi(x^{*})=\check{\phi}_{h^{*}}(x^{*})=0$, and thus
    \begin{equation}\label{eq. open and close x* final possibility}
        x^{*}\in int(\check{\mathcal{D}}_{h^{*}})\cap\partial\big(\{\check{\phi}_{h^{*}}>0\}\big)\cap\partial\big(\{\psi>0\}\big).
    \end{equation}
    
    Let $\vec{H}$ be the mean curvature vector of $\partial\big(\{\psi>0\}\big)$. By \eqref{eq. open-and-closed local max} and \eqref{eq. open and close x* final possibility}, $\{\psi>0\}$ allows a interior tangential ball of radius $\rho$ at $x^{*}$, so
    \begin{equation*}
        \vec{H}(x^{*})\cdot\vec{\nu}\leq\frac{1}{\rho}\quad\mbox{in the viscosity sense, where }\vec{\nu}=\frac{\nabla\psi(x^{*})}{|\nabla\psi(x^{*})|}.
    \end{equation*}
    Writing \eqref{eq. main simple} in terms of the level set gives
    \begin{align*}
        \frac{\mathcal{B}'(0)}{g(\Lambda^{2},0)}=&\nabla\cdot(g\nabla\psi)(x^{*})=-\vec{H}(x^{*})\cdot g(|\nabla\psi|^{2},\psi)\nabla\psi(x^{*})+\partial_{\vec{\nu}}(g\partial_{\vec{\nu}}\psi)(x^{*})\\
        \geq&-\frac{\Lambda\cdot g(\Lambda^{2},0)}{\rho}+\partial_{\vec{\nu}}(g\partial_{\vec{\nu}}\psi)(x^{*})\quad\mbox{in the viscosity sense}.
    \end{align*}
    In other words,
    \begin{equation*}
        \partial_{\vec{\nu}}\Big(g(|\nabla\psi|^{2},\psi)\partial_{\vec{\nu}}\psi\Big)(x^{*})\leq\frac{\mathcal{B}'(0)}{g(\Lambda^{2},0)}+\frac{\Lambda\cdot g(\Lambda^{2},0)}{\rho}.
    \end{equation*}
    On the other hand, by letting $t=0$ in Lemma~\ref{lem. barrier A} (A1), one has
    \begin{equation*}
        \partial_{\vec{\nu}}\Big(g(|\nabla\check{\phi}_{h^{*}}|^{2},\check{\phi}_{h^{*}})\partial_{\vec{\nu}}\check{\phi}_{h^{*}}\Big)(x^{*})=\frac{d}{dt}\Big(g\big(|\frac{d}{dt}\check{w}_{\rho}|^{2},\check{w}_{\rho}\big)\frac{d}{dt}\check{w}_{\rho}\Big)(0)>\frac{\mathcal{B}'(0)}{g(\Lambda^{2},0)}+\frac{\Lambda\cdot g(\Lambda^{2},0)}{\rho}.
    \end{equation*}
    Then, we conclude that
    \begin{equation*}
        \partial_{\vec{\nu}}\Big(g(|\nabla\psi|^{2},\psi)\partial_{\vec{\nu}}\psi\Big)(x^{*})<\partial_{\vec{\nu}}\Big(g(|\nabla\check{\phi}_{h^{*}}|^{2},\check{\phi}_{h^{*}})\partial_{\vec{\nu}}\check{\phi}_{h^{*}}\Big)(x^{*}).
    \end{equation*}
    Notice that $\vec{\nu}$ is the normal vector at $x^{*}$ pointing inwards to both $\{\psi>0\}$ and $\{\check{\phi}_{h^{*}}>0\}$ and notice that $\nabla\psi(x^{*})=\nabla\check{\phi}_{h^{*}}(x^{*})=\Lambda\vec{\nu}$, we then obtain that
    \begin{equation*}
        \check{\phi}_{h^{*}}(x^{*}+t\vec{\nu})>\psi(x^{*}+t\vec{\nu})\quad\mbox{for sufficiently small }t>0,
    \end{equation*}
    and thus \eqref{eq. closed condition} is violated. Therefore, we have shown $h^{*}=\rho-\frac{\delta r}{64}$.

    \textbf{Step 5: Ending statement.} Let $x_{1}\in[-\frac{r}{10},\frac{r}{10}]$ and let
    \begin{equation*}
        x\in\{x_{1}\}\times\big[-\frac{\delta r}{64}-2\lambda r,-\frac{\delta r}{64}+2\lambda r\big]\subseteq\check{\mathcal{D}}_{\rho-\frac{\delta r}{64}}.
    \end{equation*}
    It follows that
    \begin{equation*}
        t=\rho-|x-q_{\rho-\frac{\delta r}{64}}|=x_{2}+\frac{\delta r}{64}
    \end{equation*}
    By \eqref{eq. what lemma 3.1 implies in Maximum principle A}, we have
    \begin{equation*}
        \psi(x)\geq\check{\phi}_{h^{*}}(x)\geq\Psi(t)=\Psi\big(x_{2}+\frac{\delta r}{64}\big)\quad\mbox{in }[-\frac{r}{10},\frac{r}{10}]\times\big[-\frac{\delta r}{64}-2\lambda r,-\frac{\delta r}{64}+2\lambda r\big].
    \end{equation*}
    Since \eqref{eq. maximum principle A, constants requirements} ensures
    \begin{equation*}
        [-\frac{r}{10},\frac{r}{10}]\times\big[-\frac{\delta r}{64}-2\lambda r,-\frac{\delta r}{64}+2\lambda r\big]\subseteq[-\lambda r,\lambda r]^{2},
    \end{equation*}
    we have then verified the desired estimate of Lemma~\ref{lem. maximum principle for Harnack} part (A).
\end{proof}

\subsection{Two applications of Lemma~\ref{lem. maximum principle for Harnack}}
Now, we state and prove two corollaries of Lemma~\ref{lem. maximum principle for Harnack}.
\begin{corollary}\label{cor. application A of maximum principle}
    There exist two functions $\overline{\epsilon}(\kappa),\sigma(\kappa)>0$ defined for all $\kappa\in(0,1)$, such that the following fact holds for all $\kappa\in(0,1)$, $\epsilon\leq\overline{\epsilon}(\kappa)$ and $R>0$: Assume that
    \begin{equation*}
        \psi\mbox{ solves \eqref{eq. main simple} and satisfies }\Psi(y-\epsilon R)\leq\psi(z)\leq\Psi(y+\epsilon R)\mbox{ in }[-R,R]^{2}.
    \end{equation*}
    \begin{itemize}
        \item[(A)] \textbf{Lower bound increases.} If
        \begin{equation*}
            \psi(x)\geq\Psi\big(x_{2}-(1-\kappa)\epsilon R\big)\quad\mbox{in }\big[-\frac{9}{10}R,\frac{9}{10}R\big]\times\big[\frac{\lambda}{10}R,\frac{9}{10}R\big],
        \end{equation*}
        with $\lambda$ being the uniform constant produced by Lemma~\ref{lem. maximum principle for Harnack}. Then, one has
        \begin{equation*}
            \psi(x)\geq\Psi\big(x_{2}-(1-\frac{\kappa}{64})\epsilon R\big)\quad\mbox{in }\big[-\frac{\lambda}{4}R,\frac{\lambda}{4}R\big]^{2}.
        \end{equation*}
        \item[(B)] \textbf{Upper bound decreases.} If
        \begin{equation*}
            \psi(x)\leq\Psi\big(x_{2}+(1-\kappa)\epsilon R\big)\quad\mbox{in }\big[-\frac{9}{10}R,\frac{9}{10}R\big]\times\big[\frac{\lambda}{10}R,\frac{9}{10}R\big],
        \end{equation*}
        with $\lambda$ being the uniform constant produced by Lemma~\ref{lem. maximum principle for Harnack}. Then, one has
        \begin{equation*}
            \psi(x)\leq\Psi\big(x_{2}+(1-\frac{\kappa}{64})\epsilon R\big)\quad\mbox{in }\big[-\frac{\lambda}{4}R,\frac{\lambda}{4}R\big]^{2}.
        \end{equation*}
    \end{itemize}
\end{corollary}
\begin{proof}
    We only prove part (A) of Corollary~\ref{cor. application A of maximum principle}, leaving part (B) to the readers.
    
    Let $\overline{\delta},\lambda$ be the constants produced by Lemma~\ref{lem. maximum principle for Harnack}. Let
    \begin{equation*}
        \overline{\epsilon}(\kappa)=\min\Big\{\frac{4\overline{\delta}}{5\kappa},\frac{\lambda}{4},\frac{1}{10}\Big\}.
    \end{equation*}
    Let us denote
    \begin{equation*}
        r=\frac{4R}{5},\quad\delta=\frac{5}{4}\epsilon\cdot\kappa.
    \end{equation*}
    Then, clearly, $\delta r=\kappa\epsilon R$. Let $\widetilde{x}=x-\epsilon R\vec{e}_{2}$, and let
    \begin{equation*}
        \widetilde{\psi}(\widetilde{x})=\psi(\widetilde{x}+\epsilon R\vec{e}_{2})=\psi(x),\quad\mbox{for all }\widetilde{x}\in[-r,r]^{2}.
    \end{equation*}
    Clearly, for $\widetilde{x}\in[-r,r]^{2}$, correspondingly one has $x=\widetilde{x}+\epsilon R\vec{e}_{2}\in[-R,R]^{2}$. Therefore,
    \begin{equation*}
        \widetilde{\psi}(\widetilde{x})=\psi(x)\geq\Psi(x_{2}-\epsilon R)\geq\Psi(\widetilde{x}_{2})\quad\mbox{for }\widetilde{x}\in[-r,r]^{2}.
    \end{equation*}
    Furthermore, for $\widetilde{x}\in[-r,r]\times[\lambda r,r]$, correspondingly one has
    \begin{equation*}
        x=\widetilde{x}+\epsilon R\vec{e}_{2}\in\big[-\frac{9}{10}R,\frac{9}{10}R\big]\times\big[\frac{\lambda}{10}R,\frac{9}{10}R\big].
    \end{equation*}
    Therefore,
    \begin{equation*}
        \widetilde{\psi}(\widetilde{x})=\psi(x)\geq\Psi\big(x_{2}-(1-\kappa)\epsilon R\big)=\Psi(\widetilde{x}_{2}+\kappa\epsilon R)=\Psi(\widetilde{x}_{2}+\delta r)\quad\mbox{for }\widetilde{x}\in[-r,r]^{2}.
    \end{equation*}
    Applying Lemma~\ref{lem. maximum principle for Harnack} to $\widetilde{\psi}$ gives that
    \begin{equation*}
        \widetilde{\psi}(\widetilde{x})\geq\Psi\big(\widetilde{x}_{2}+\frac{\delta r}{64}\big)=\Psi\big(\widetilde{x}_{2}+\frac{\kappa\epsilon R}{64}\big)\quad\mbox{for }\widetilde{x}\in[-\lambda r,\lambda r]^{2}.
    \end{equation*}
    Let $x\in[-\frac{\lambda}{4}R,\frac{\lambda}{4}R\big]^{2}$, then correspondingly one has
    \begin{equation*}
        \widetilde{x}=x-\epsilon R\vec{e}_{2}\in\big[-\frac{\lambda}{4}R,\frac{\lambda}{4}R\big]\times\big[-\frac{\lambda}{2}R,\frac{\lambda}{2}R\big]\subseteq[-\lambda r,\lambda r]^{2}.
    \end{equation*}
    Therefore,
    \begin{equation*}
        \psi(x)=\widetilde{\psi}(\widetilde{x})\geq\Psi\big(\widetilde{x}_{2}+\frac{\kappa\epsilon R}{64}\big)=\Psi\big(x_{2}-(1-\frac{\kappa}{64})\epsilon R\big)\quad\mbox{for }x\in[-\frac{\lambda}{4}R,\frac{\lambda}{4}R\big]^{2}.
    \end{equation*}
    Then, we have finished the proof of part (A).
\end{proof}

\begin{corollary}\label{cor. application B of maximum principle}
    Let $\lambda\in(0,1)$ be the uniform constant produced by Lemma~\ref{lem. maximum principle for Harnack}. There exists a function $\delta_{0}(\beta)>0$ defined for all $\beta>0$, such that the following holds: Given $\beta>0$, let $r>0$, $0<\delta\leq \delta_{0}(\beta)$, and $0<\iota\leq\frac{\lambda\beta}{512}$ be arbitrary.
    \begin{itemize}
        \item[(A)] \textbf{Bending-up case:} Assume that
        \begin{equation*}
            \psi\mbox{ solves \eqref{eq. main simple} and satisfies }\psi(x)\geq\Psi\Big((1+\beta\delta)x_{2}-\iota\delta r\Big)\mbox{ in }B_{r}.
        \end{equation*}
        Then, one has
        \begin{equation*}
            \psi(x)\geq\Psi\big(x_{2}+\frac{\lambda}{512}\beta\delta r\big)\quad\mbox{in }B_{\frac{\lambda}{4}r}.
        \end{equation*}
        \item[(B)] \textbf{Bending-down case:} Assume that
        \begin{equation*}
            \psi\mbox{ solves \eqref{eq. main simple} and satisfies }\psi(z)\leq\Psi\Big((1-\beta\delta)y+\iota\delta r\Big)\mbox{ in }B_{r}.
        \end{equation*}
        Then, one has
        \begin{equation*}
            \psi(x)\leq\Psi\big(x_{2}-\frac{\lambda}{512}\beta\delta r\big)\quad\mbox{in }B_{\frac{\lambda}{4}r}.
        \end{equation*}
    \end{itemize}
\end{corollary}
\begin{proof}
    We only prove part (A) of Corollary~\ref{cor. application B of maximum principle}, leaving part (B) to the readers.
    
    Let $\lambda,\overline{\delta}$ be the uniform constants produced by Lemma~\ref{lem. maximum principle for Harnack}. Let us require
    \begin{equation*}
        \delta\leq \delta_{0}(\beta)=\min\Big\{\frac{2\overline{\delta}}{\lambda\beta},\frac{1}{\beta}\Big\}.
    \end{equation*}
    Clearly, one has $\iota\delta\leq\frac{\lambda}{512}$. Denote
    \begin{equation*}
        \widetilde{r}=\frac{r}{2},\quad\widetilde{\delta}=\frac{\lambda}{2}\beta\delta.
    \end{equation*}
    Let $\widetilde{x}=x-\iota\delta r\vec{e}_{2}$. If $\widetilde{x}\in[-\widetilde{r},\widetilde{r}]^{2}$, then
    \begin{equation*}
        x=\widetilde{x}+\iota\delta r\vec{e}_{2}\in\big[-\frac{3}{5}r,\frac{3}{5}r\big]^{2}\subseteq B_{r}.
    \end{equation*}
    Therefore, the following function $\widetilde{\psi}$ is well-defined:
    \begin{equation*}
        \widetilde{\psi}(\widetilde{x})=\psi(\widetilde{x}+\iota\delta r\vec{e}_{2})=\psi(x),\quad\mbox{for all }\widetilde{x}\in[-\widetilde{r},\widetilde{r}]^{2}.
    \end{equation*}
    
    We claim that
    \begin{equation}\label{eq. claim in cor B of maximum principle}
        \widetilde{\psi}(\widetilde{x})\geq\Psi(\widetilde{x}_{2})\quad\mbox{for }\widetilde{x}\in[-\widetilde{r},\widetilde{r}]^{2}.
    \end{equation}
    To show \eqref{eq. claim in cor B of maximum principle}, we first see that when $\widetilde{x}\in[-\widetilde{r},\widetilde{r}]^{2}$, one has
    \begin{equation*}
        \widetilde{\psi}(\widetilde{x})=\psi(x)\geq\Psi\Big((1+\beta\delta)x_{2}-\iota\delta r\Big)=\Psi(\beta\delta x_{2}+\widetilde{x}_{2}).
    \end{equation*}
    Then, \eqref{eq. claim in cor B of maximum principle} follows from considering the following two cases separately:
    \begin{itemize}
        \item If $x_{2}\geq0$, then clearly $\beta\delta x_{2}+\widetilde{x}_{2}\geq\widetilde{x}_{2}$, so $\widetilde{\psi}(\widetilde{x})\geq\Psi(\beta\delta x_{2}+\widetilde{x}_{2})\geq\Psi(\widetilde{x}_{2})$;
        \item If $x_{2}<0$, then $\widetilde{x}_{2}=x_{2}-\iota\delta r<0$, so $\Psi(\widetilde{x}_{2})=0\leq\widetilde{\psi}(\widetilde{x})$.
    \end{itemize}
    
    Moreover, when $\widetilde{x}\in[-\widetilde{r},\widetilde{r}]\times[\lambda\widetilde{r},\widetilde{r}]$, one has $x_{2}\geq\lambda\widetilde{r}-\iota\delta r\geq\frac{\lambda}{4}r$, so
    \begin{equation*}
        \widetilde{\psi}(\widetilde{x})\geq\Psi(\beta\delta x_{2}+\widetilde{x}_{2})\geq\Psi\big(\widetilde{x}_{2}+\frac{\lambda}{4}\beta\delta r\big)=\Psi(\widetilde{x}_{2}+\widetilde{\delta}\widetilde{r}).
    \end{equation*}
    Combining the above estimate with \eqref{eq. claim in cor B of maximum principle} and applying Lemma~\ref{lem. maximum principle for Harnack} yields that
    \begin{equation*}
        \widetilde{\psi}(\widetilde{x})\geq\Psi\big(\widetilde{x}_{2}+\frac{\widetilde{\delta}\widetilde{r}}{64}\big)=\Psi\big(\widetilde{x}_{2}+\frac{\lambda}{256}\beta\delta r\big)\quad\mbox{for }\widetilde{x}\in[-\lambda\widetilde{r},\lambda\widetilde{r}]^{2}.
    \end{equation*}

    For $x\in B_{\frac{\lambda}{4}r}$, one has
    \begin{equation*}
        \widetilde{x}=x-\iota\delta r\vec{e}_{2}\in B_{\frac{\lambda}{2}r}\subseteq[-\lambda\widetilde{r},\lambda\widetilde{r}]^{2}.
    \end{equation*}
    Then, for $x\in B_{\frac{\lambda}{4}r}$, one has
    \begin{equation*}
        \psi(x)=\widetilde{\psi}(\widetilde{x})\geq\Psi\big(\widetilde{x}_{2}+\frac{\lambda}{256}\beta\delta r\big)=\Psi\big(x_{2}-\iota\delta r+\frac{\lambda}{256}\beta\delta r\big).
    \end{equation*}
    Since $\iota\leq\frac{\lambda\beta}{512}$, we get that
    \begin{equation*}
        \psi(x)\geq\Psi\big(\widetilde{x}_{2}+\frac{\lambda}{512}\beta\delta r\big)\quad\mbox{for }x\in B_{\frac{\lambda}{4}r}.
    \end{equation*}
    Hence, we have finished the proof of part (A).
\end{proof}

\subsection{A variant of Lemma~\ref{lem. maximum principle for Harnack}}
By the proof of Lemma~\ref{lem. maximum principle for Harnack} part (A), we have its variant below, where the gradient assumption on the free boundary is not fully assumed.
\begin{lemma}\label{lem. maximum variant}
    There exist uniform constants $\overline{\delta},\lambda\in(0,1)$, such that the following fact holds for all $r>0$ and $\delta\leq\overline{\delta}$: Assume that $\psi$ satisfies
    \begin{equation*}
        \left\{\begin{aligned}
            &\nabla\cdot\Big(g(|\nabla\psi|^{2},\psi)\nabla\psi\Big)=\frac{\mathcal{B}'(\psi)}{g(|\nabla\psi|^{2},\psi)}&\mbox{in }&\{\psi>0\}\cap[-r,r]^{2},\\
            &|\nabla\psi|=\Lambda&\mbox{on }&\Big([-r,r]\times\big[-\frac{\delta r}{2},r\big]\Big)\cap\partial\{\psi>0\},\\
            &\psi(x)\geq\Psi(x_{2})&\mbox{in }&[-r,r]^{2},\\
            &\psi(x)\geq\Psi(x_{2}+\delta r)&\mbox{in }&[-r,r]\times[\lambda r,r].
        \end{aligned}\right.
    \end{equation*}
    Then, one has
    \begin{equation*}
        \psi(x)\geq\Psi(x_{2}+\frac{\delta r}{64})\quad\mbox{in }[-\lambda r,\lambda r]^{2}.
    \end{equation*}
\end{lemma}
\begin{proof}
    It suffices to notice that in \textbf{Step 4} of the proof of Lemma~\ref{lem. maximum principle for Harnack} part (A), all arguments before and including \eqref{eq. open and close x* final possibility} do not depend on the gradient assumption on $\partial\{\psi=0\}$. Since $h^{*}\geq\rho-\frac{\sigma r}{64}$, we further have that $x^{*}\in[-r,r]\times\big[-\frac{\delta r}{2},r\big]$, so the remaining arguments in \textbf{Step 4} is still valid.
\end{proof}
Similar to Corollary~\ref{cor. application A of maximum principle}, we have the following application. The proof is exactly the same as Corollary~\ref{cor. application A of maximum principle}.
\begin{corollary}\label{cor. application C of maximum principle}
    There exist two functions $\overline{\epsilon}(\kappa),\sigma(\kappa)>0$ defined for all $\kappa\in(0,1)$, such that the following fact holds for all $\kappa\in(0,1)$, $\epsilon\leq\overline{\epsilon}(\kappa)$ and $R>0$: Assume that $\psi$ satisfies
    \begin{equation*}
        \left\{\begin{aligned}
            &\nabla\cdot\Big(g(|\nabla\psi|^{2},\psi)\nabla\psi\Big)=\frac{\mathcal{B}'(\psi)}{g(|\nabla\psi|^{2},\psi)}&\mbox{in }&\{\psi>0\}\cap[-R,R]^{2},\\
            &|\nabla\psi|=\Lambda&\mbox{on }&\Big([-R,R]\times\big[\frac{\epsilon R}{2},R\big]\Big)\cap\partial\{\psi>0\},\\
            &\Psi(x_{2}-\epsilon R)\leq\psi(x)\leq\Psi(x_{2})&\mbox{in }&[-R,R]^{2},\\
            &\psi(x)\geq\Psi\big(x_{2}-(1-\kappa)\epsilon R\big)&\mbox{in }&\big[-\frac{9}{10}R,\frac{9}{10}R\big]\times\big[\frac{\lambda}{10}R,\frac{9}{10}R\big],
        \end{aligned}\right.
    \end{equation*}
    where $\lambda$ is the uniform constant produced by Lemma~\ref{lem. maximum variant} (as well as Lemma~\ref{lem. maximum principle for Harnack}). Then, one has
    \begin{equation*}
        \psi(x)\leq\Psi\big(x_{2}+(1-\frac{\kappa}{64})\epsilon R\big)\quad\mbox{in }\big[-\frac{9}{10}R,\frac{9}{10}R\big]\times\big[\frac{\lambda}{10}R,\frac{9}{10}R\big].
    \end{equation*}
\end{corollary}
\subsection{On the proof of Lemma~\ref{lem. maximum principle for forced decay}}
In this subsection, we give the proof of Lemma~\ref{lem. maximum principle for forced decay}, the method is similar to Lemma~\ref{lem. maximum principle for Harnack}.
\begin{lemma}
    There exist uniform constants $\overline{\delta},\sigma\in(0,1)$, such that the following fact holds for all $r>0$ and $\delta\leq\widetilde{\delta}$: Assume that
    \begin{equation*}
        \psi\mbox{ satisfies \eqref{eq. quasilinear structure} and }\Psi(x_{2}-\delta r)\leq\psi(x)\leq\Psi(x_{2})\mbox{ in }[-r,r]\times[0,r].
    \end{equation*}
    Then, one has $\psi\leq\Psi\big(x_{2}-(1-\sigma)\delta r\big)$ in $[-\sigma r,\sigma r]\times[0,\sigma r]$
\end{lemma}
\begin{proof}[Proof of Lemma~\ref{lem. maximum principle for forced decay}]
    \textbf{Step 1: Geometric set-up.} First, let us choose constants as follows:
    \begin{equation*}
        \lambda=\min\Big\{(256\mathfrak{C}_{6})^{-1/2},\frac{\mathfrak{c}_{5}}{10}\Big\},\quad\overline{\delta}=\min\{\mathfrak{c}_{5},\frac{\lambda}{2},1\},\quad\rho=\frac{r}{\delta}.
    \end{equation*}
    Fix an $x_{1}\in[-\frac{r}{10},\frac{r}{10}]$, then we set
    \begin{equation*}
        q_{h}=x_{1}\vec{e}_{1}+h\vec{e}_{2}\in\mathbb{R}^{2}\mbox{ for each }h\in\mathbb{R},
    \end{equation*}
    Let $m:=2\lambda r$, and let $(w_{\rho,m},t_{-})$ be produced by Lemma~\ref{lem. barrier B}. By the estimate of $t_{-}$ in Lemma~\ref{lem. barrier B}, we have
    \begin{equation*}
        -m\geq t_{-}\geq-m\cdot(1+\frac{\mathfrak{C}_{6}m}{\rho})=-m\cdot(1+2\mathfrak{C}_{6}\lambda\delta)\geq-2m.
    \end{equation*}
    Define
    \begin{equation*}
        \check{\phi}_{h}(x):=w_{\rho,m}\big(\rho-|x-q_{h}|\big),\quad\mbox{for }x\in\check{\mathcal{D}}_{h},
    \end{equation*}
    where the domain $\check{\mathcal{D}}_{h}$ is given as follows:
    \begin{equation*}
        \check{\mathcal{D}}_{h}:=\Big\{x\in\mathbb{R}^{2}:t_{-}\leq\rho-|x-q_{h}|\leq0,\ \angle{\big(x-q_{h},-\vec{e}_{2}\big)}\leq\sin^{-1}(\frac{\delta}{2})\Big\}.
    \end{equation*}
    For $x\in\check{\mathcal{D}}_{h}$, we make the following simplified notations:
    \begin{equation*}
        t=\rho-|x-q_{h}|\in[t_{-},0],\quad\theta=\angle{\big(x-q_{h},-\vec{e}_{2}\big)}\leq\sin^{-1}(\frac{\delta}{2}).
    \end{equation*}
    It is easy to verify that when $|x_{1}|\leq\frac{r}{10}$, one always has
    \begin{equation*}
        \check{\mathcal{D}}_{h}\subseteq(-r,r)^{2},\quad\mbox{for all }\rho+m+\delta r\leq h\leq\rho+m+\frac{r}{4}.
    \end{equation*}
    Moreover, since $\check{\mathcal{D}}_{h}\subseteq\{x_{2}\geq h-\rho+t_{-}\}$, we see that as long as $h\geq\rho+m$, we have
    \begin{equation*}
        \check{\mathcal{D}}_{h}\subseteq\{x_{2}\geq m+t_{-}+\delta r\}\subseteq\{x_{2}\geq-\frac{\mathfrak{C}_{6}m^{2}}{\rho}+\delta r\}\subseteq\{x_{2}>0\}.
    \end{equation*}
    Similar to Lemma~\ref{lem. maximum principle for Harnack}, we initially set $h=\rho+m+\frac{r}{4}$ and later decrement $h$.

    \textbf{Step 2: Initial comparison.} When $h=\rho+m+\frac{r}{4}$, we notice that
    \begin{equation*}
        \check{\mathcal{D}}_{h}\subseteq\{x_{2}\geq t_{-}+\frac{r}{4}\}\subseteq\{x_{2}\geq\frac{r}{4}-2m\}\subseteq\{x_{2}\geq m+2\delta r\}.
    \end{equation*}
    Therefore, one has that for $h=\rho+m+\frac{r}{4}$,
    \begin{equation*}
        \phi_{h}(x)=w_{\rho,m}(t)\leq\Psi(m)<\Psi(x_{2}-\delta r)\leq\psi(x)\quad\mbox{for all }x\in\check{\mathcal{D}}_{h}.
    \end{equation*}

    \textbf{Step 3: Boundary comparison.} When $\rho+m+\delta r\leq h\leq\rho+m+\frac{r}{4}$, we also divide $\partial\check{\mathcal{D}}_{h}$ into three parts like in the proof of Lemma~\ref{lem. maximum principle for Harnack}.
    \begin{itemize}
        \item On outer boundary $\partial_{out}\check{\mathcal{D}}_{h}=\partial\check{\mathcal{D}}_{h}\cap\{t=t_{-}\}$, since we have already shown that $\check{\mathcal{D}}_{h}\subseteq\{x_{2}>0\}$ for all $h\geq\rho+m+\delta r$, then we have $\partial_{out}\check{\mathcal{D}}_{h}\subseteq(-r,r)^{2}\times(0,r)$. Then, we have that $\psi>0$ on $\partial_{out}\check{\mathcal{D}}_{h}$, otherwise it violates the strong maximum principle applied to $\psi$. Therefore, we have see that $\phi_{h}<\psi$ on $\partial_{out}\check{\mathcal{D}}_{h}$ for $\rho+m+\delta r\leq h\leq\rho+m+\frac{r}{4}$.
        \item On inner boundary $\partial_{in}\check{\mathcal{D}}_{h}=\partial\check{\mathcal{D}}_{h}\cap\{t=0\}$ and on side boundary $\partial_{side}\check{\mathcal{D}}_{h}=\partial\check{\mathcal{D}}_{h}\cap\{\theta=\sin^{-1}(\frac{\delta}{2})\}$, one can prove that $\phi_{h}\leq\psi$ for $\rho+m+\delta r\leq h\leq\rho+m+\frac{r}{4}$ similar to Lemma~\ref{lem. maximum principle for Harnack}.
    \end{itemize}
    In conclusion, we see that $\psi\geq\phi_{h}$ on $\partial\check{\mathcal{D}}_{h}$ for $\rho+m+\delta r\leq h\leq\rho+m+\frac{r}{4}$.

    \textbf{Step 4: Ending statement.} Since $\psi$ satisfies \eqref{eq. quasilinear structure} in $(-r,r)\times(0,r)$ without any free boundary, it follows from the second derivative test argument and the open-and-closed argument that $\psi\geq\phi_{h}$ in $\check{\mathcal{D}}_{h}$ when $h=\rho+m+\delta r$. By Lemma~\ref{lem. barrier B}, we have
    \begin{equation*}
        \psi(x)\geq\phi_{h}(x)=w_{\rho,m}(t)\geq\Psi\Big(m+t+\frac{t^{2}}{\mathfrak{C}_{6}\rho}\Big)\quad\mbox{on }\{x_{1}\}\times[m+t_{-}+\delta r,\lambda r].
    \end{equation*}
    On the other hand, on the segment $\{x_{1}\}\times[m+t_{-}+\delta r,\lambda r]$, one has
    \begin{equation*}
        x_{2}=t-\rho+h=t+m+\delta r,\quad\mbox{and hence }t\leq\lambda r-m-\delta r\leq-\lambda r.
    \end{equation*}
    Therefore, we get that
    \begin{equation*}
        \psi(x)\geq\Psi\Big(x_{2}-\delta r+\frac{t^{2}}{\mathfrak{C}_{6}\rho}\Big)\geq\Psi\Big(x_{2}-\big(1-\frac{\lambda^{2}}{\mathfrak{C}_{6}}\big)\delta r\Big).
    \end{equation*}
    By setting $\sigma$ sufficiently small, we have finished the proof of Lemma~\ref{lem. maximum principle for forced decay}.
\end{proof}

\section{Small-perturbation equation}\label{sec. appendix Small-perturbation}
In this section, we consider a small perturbation problem with respect to the one-dimensional solution $\Psi$ satisfying \eqref{eq. 1D solution Psi}. Assume that
\begin{equation}\label{eq. positive region equation}
    \nabla\cdot\Big(g(|\nabla\psi|^{2},\psi)\nabla\psi\Big)=\frac{\mathcal{B}'(\psi)}{g(|\nabla\psi|^{2},\psi)}\mbox{ and }\psi>0\mbox{ in some region }\Omega\subseteq\mathbb{R}^{2}.
\end{equation}
Since $\Psi_{t}>0$ in $\mathbb{R}_{+}$, one finds that $\Psi^{-1}(\psi)$ is well-defined in $\Omega$. Then, we set the horizontal shifting function as
\begin{equation*}
    S(x)=\Psi^{-1}(\psi(x))-x_{2},\quad x\in\Omega.
\end{equation*}
Assume that $\Omega\subseteq[-r,r]^{2}$ and suppose that $\|S\|_{L^{\infty}(\Omega)}=\epsilon r$ for some $\epsilon$ sufficiently small, we then consider the following rescaled function (depending on $r$ and $\epsilon$):
\begin{equation}\label{eq. v(y)=S(x)/epsilon r}
    v(y)=\frac{S(x)}{\epsilon r},\quad\mbox{where }y=(y_{1},y_{2})=(\frac{x_{1}}{r},\frac{x_{2}}{r})=\frac{x}{r}\in\frac{1}{r}\Omega.
\end{equation}
Clearly, $\|v\|_{L^{\infty}(\frac{1}{r}\Omega)}=1$.

\subsection{Derive a small-perturbation equation}
In this subsection, we derive a small-perturbation elliptic equation akin to \cite{S07} for $v$.
    
\textbf{Step 1: Chain rules.} Since $x=r\cdot y$, using the chain rule yields that for the same scalar valued function $f$ or vector valued function $\vec{g}$, one has
\begin{equation*}
    \nabla_{x}f=\frac{1}{r}\nabla_{y}f,\quad\mbox{and }\nabla_{x}\cdot\vec{g}=\frac{1}{r}\nabla_{y}\cdot\vec{g}.
\end{equation*}
Moreover, if $\nabla_{y}v$ is known, then one has
\begin{equation*}
    \nabla_{x}S(x)=\epsilon\cdot\nabla_{y}v(y).
\end{equation*}
Write
\begin{equation*}
    \widetilde{x}_{2}=x_{2}+S(x)=x_{2}+\epsilon r\cdot v(y).
\end{equation*}
If we write $\widetilde{x}_{2}=\widetilde{x}_{2}(x_{1},x_{2})$, then
\begin{equation}\label{eq. tilde x_2 with x_2 chain rule}
    \frac{\partial\widetilde{x}_{2}}{\partial x_{2}}=1+\epsilon\cdot\partial_{y_{2}}v.
\end{equation}
Since $\psi(x)=\Psi(\widetilde{x}_{2})$ where $\widetilde{x}_{2}=x_{2}+S(x)=x_{2}+\epsilon r v(y)$, it then follows from \eqref{eq. speed identity} that
\begin{equation*}
    \nabla_{x}\psi=\Psi_{t}(\widetilde{x}_{2})\cdot(\vec{e}_{2}+\nabla_{x}S)=\sqrt{H(\psi)}\cdot(\vec{e}_{2}+\epsilon\cdot\nabla_{y}v).
\end{equation*}
Substituting this identity in \eqref{eq. positive region equation} yields that
\begin{equation}\label{eq. substitude to get a purturbed equation}
    g\nabla_{x}\cdot\Big(g\cdot\sqrt{H(\psi)}\cdot(\vec{e}_{2}+\epsilon\nabla_{y}v)\Big)=\mathcal{B}'(\psi),
\end{equation}
where
\begin{equation*}
    g=g\big(H(\psi)\cdot|\vec{e}_{2}+\epsilon\cdot\nabla_{y}v|^{2},\psi\big),\quad\psi=\Psi(\widetilde{x}_{2})=\Psi(x_{2}+\epsilon r\cdot v(y)).
\end{equation*}

\textbf{Step 2: Rolle's mean value theorem.} Write the left-hand side of \eqref{eq. substitude to get a purturbed equation} as
\begin{equation}\label{eq. a diff function with parameter epsilon}\begin{aligned}
    \mathcal{F}(\delta,y):=&\frac{1}{r}g\Big(H\big(\Psi(\widetilde{x}_{2})\big)\cdot\big|\vec{e}_{2}+\delta\nabla_{y}v\big|^{2},\Psi(\widetilde{x}_{2})\Big)\\
    \cdot&\nabla_{y}\cdot\Big[g\Big(H\big(\Psi(\widetilde{x}_{2})\big)\cdot\big|\vec{e}_{2}+\delta\nabla_{y}v\big|^{2},\Psi(\widetilde{x}_{2})\Big)\cdot\sqrt{H\big(\Psi(\widetilde{x}_{2})\big)}\cdot\big(\vec{e}_{2}+\delta\nabla_{y}v\big)\Big].
\end{aligned}\end{equation}
Here, we must stress that in the expression of $\mathcal{F}(\delta,y)$, the value $\widetilde{x}_{2}$ is always $x_{2}+\epsilon r v(y)$, instead of changing with $\delta$.
    
Clearly, it follows from \eqref{eq. substitude to get a purturbed equation} that
\begin{equation*}
    \mathcal{F}(\epsilon,y)=(\mathcal{B}'\circ\Psi)(\widetilde{x}_{2}).
\end{equation*}
On the other hand, it follows from \eqref{eq. 1D solution Psi} and \eqref{eq. speed identity} that
\begin{align*}
    (\mathcal{B}'\circ\Psi)(\widetilde{x}_{2})=&\Big[g(|\Psi_{t}|^{2},\Psi)\cdot\Big(g(|\Psi_{t}|^{2},\Psi)\Psi_{t}\Big)_{t}\Big](\widetilde{x}_{2})\\
    =&g\big(H(\Psi(\widetilde{x}_{2})),\Psi(\widetilde{x}_{2})\big)\cdot\frac{d}{d\widetilde{x}_{2}}\Big(g\big(H(\Psi(\widetilde{x}_{2})),\Psi(\widetilde{x}_{2})\big)\sqrt{H\big(\Psi(\widetilde{x}_{2})\big)}\Big)
\end{align*}
Then applying the chain rule \eqref{eq. tilde x_2 with x_2 chain rule} yields that
\begin{align*}
    \mathcal{F}(0,y)=&g\big(H(\Psi(\widetilde{x}_{2})),\Psi(\widetilde{x}_{2})\big)\cdot\nabla_{x}\cdot\Big(g\big(H(\Psi(\widetilde{x}_{2})),\Psi(\widetilde{x}_{2})\big)\sqrt{H\big(\Psi(\widetilde{x}_{2})\big)}\cdot\vec{e}_{2}\Big)\\
    =&g\big(H(\Psi(\widetilde{x}_{2})),\Psi(\widetilde{x}_{2})\big)\cdot\frac{\partial}{\partial x_{2}}\Big(g\big(H(\Psi(\widetilde{x}_{2})),\Psi(\widetilde{x}_{2})\big)\sqrt{H\big(\Psi(\widetilde{x}_{2})\big)}\Big)\\
    =&(\mathcal{B}'\circ\Psi)(\widetilde{x}_{2})\cdot\frac{\partial\widetilde{x}_{2}}{\partial x_{2}}=(\mathcal{B}'\circ\Psi)(\widetilde{x}_{2})\cdot(1+\epsilon\cdot\partial_{y_{2}}v).
\end{align*}
    
Therefore, we get that
\begin{equation*}
    \mathcal{F}(\epsilon,y)\cdot(1+\epsilon\cdot\partial_{y_{2}}v)=\mathcal{F}(0,y)=\mathcal{F}(0,y)\cdot(1+0\cdot\partial_{y_{2}}v).
\end{equation*}
By using the Rolle's mean value theorem, we get that
\begin{equation}\label{eq. Rolle type identity}
    \frac{d}{d\delta}\Big|_{\delta=\xi(y)}\Big(\mathcal{F}(\delta,y)\cdot(1+\delta\cdot\partial_{y_{2}}v)\Big)=0\quad\mbox{for some function }0<\xi(y)<\epsilon.
\end{equation}
In fact, in general $\xi(y)$ also depends on $(\epsilon,r)$, as well as on the $0th-2nd$ order derivatives of $v$. The complete dependence of $\xi$ is $\xi=\xi(D^{2}v,Dv,v,y,\epsilon,r)$. However, in our discussion, we just need to treat $\xi$ as a function from $\frac{1}{r}\Omega$ to $[0,\epsilon]$.
    
\textbf{Step 3: Turn to a non-divergence equation.} In order to apply \eqref{eq. Rolle type identity} without taking the $y$-derivative on $\xi(y)$, we must turn $\mathcal{F}(\delta,y)\cdot(1+\delta\cdot\partial_{y_{2}}v)$ into a non-divergence form. To achieve this, we set
\begin{equation*}
    \mathcal{T}(\delta,y):=H\big(\Psi(\widetilde{x}_{2})\big)\cdot\big|\vec{e}_{2}+\delta\nabla_{y}v\big|^{2},\quad\mathcal{Z}(y)=\Psi(\widetilde{x}_{2}),\quad H=H\big(\Psi(\widetilde{x}_{2})\big)=H(\mathcal{Z}).
\end{equation*}
For the function $g$ (and also its derivatives), we omit their input $(\mathcal{T},\mathcal{Z})$. For example,
\begin{equation*}
    g=g(\mathcal{T},\mathcal{Z}),\quad g_{T}=g_{T}(\mathcal{T},\mathcal{Z}),\quad g_{Tz}=g_{Tz}(\mathcal{T},\mathcal{Z}),\quad\mbox{e.t.c..}
\end{equation*}
Then, we express $\mathcal{F}(\delta,y)\cdot(1+\delta\cdot\partial_{y_{2}}v)$ in a non-divergence way:
\begin{equation}\label{eq. F*(1+delta*v_2) expression}\begin{aligned}
    &\mathcal{F}(\delta,y)\cdot(1+\delta\cdot\partial_{y_{2}}v)=\frac{1+\delta\cdot\partial_{y_{2}}v}{r}g^{2}\cdot\Big(\delta\sqrt{H}\Delta_{y}v+\big\langle\nabla_{y}\sqrt{H},\vec{e}_{2}+\delta\nabla_{y}v\big\rangle\Big)\\
    &+\frac{1+\delta\cdot\partial_{y_{2}}v}{r}gg_{z}\sqrt{H}\big\langle\nabla_{y}\mathcal{Z},\vec{e}_{2}+\delta\nabla_{y}v\big\rangle\\
    &+\frac{1+\delta\cdot\partial_{y_{2}}v}{r}gg_{T}\sqrt{H}\big\langle\big|\vec{e}_{2}+\delta\nabla_{y}v\big|^{2}\cdot\nabla_{y}H,\vec{e}_{2}+\delta\nabla_{y}v\big\rangle\\
    &+\frac{1+\delta\cdot\partial_{y_{2}}v}{r}gg_{T}\sqrt{H}\Big\langle H\cdot\big(2\delta\nabla_{y}\partial_{y_{2}}v+\delta^{2}\nabla_{y}\big|\nabla_{y}v\big|^{2}\big),\vec{e}_{2}+\delta\nabla_{y}v\Big\rangle.
\end{aligned}\end{equation}
    
Now let us denote
\begin{equation*}
    \xi=\xi(y)=\xi(D^{2}v,Dv,v,y,\epsilon,r),\quad\alpha=1+\xi\cdot\partial_{y_{2}}v,\quad\vec{\eta}=\vec{e}_{2}+\xi\nabla_{y}v.
\end{equation*}
Besides, we let
\begin{equation}\label{eq. T' yin_shi_fen_jie}
    \mathcal{T}'=\frac{\partial}{\partial\delta}\mathcal{T}(\xi,y)=2H\cdot(\partial_{y_{2}}v+\xi\cdot|\nabla_{y}v|^{2})=\big\langle2H\vec{\eta},\nabla_{y}v\big\rangle.
\end{equation}
For the function $g$ (and also its derivatives), we write
\begin{equation*}
    g=g(\mathcal{T}(\xi,y),\mathcal{Z}(y)),\quad g_{T}=g_{T}(\mathcal{T}(\xi,y),\mathcal{Z}(y)),\quad\mbox{e.t.c..}
\end{equation*}
Then we get from \eqref{eq. Rolle type identity} and \eqref{eq. F*(1+delta*v_2) expression} the following non-divergence equation:
\begin{equation}\label{eq. super-horrible}\begin{aligned}
    0=&\frac{\alpha g^{2}}{r}\cdot\Big(\sqrt{H}\Delta_{y}v+\big\langle\nabla_{y}\sqrt{H},\nabla_{y}v\big\rangle\Big)\\
    &+\Big(\frac{g^{2}\partial_{y_{2}}v}{r}+\frac{\alpha}{r}\cdot\frac{\partial(g^{2})}{\partial T}\cdot\mathcal{T}'\Big)\cdot\Big(\xi\sqrt{H}\Delta_{y}v+\big\langle\nabla_{y}\sqrt{H},\vec{\eta}\big\rangle\Big)\\
    &+\frac{\alpha gg_{z}\sqrt{H}}{r}\big\langle\nabla_{y}\mathcal{Z},\nabla_{y}v\big\rangle+\Big(\frac{\alpha\sqrt{H}}{r}\cdot\frac{\partial(gg_{z})}{\partial T}\cdot\mathcal{T}'+\frac{gg_{z}\sqrt{H}\partial_{y_{2}}v}{r}\Big)\cdot\big\langle\nabla_{y}\mathcal{Z},\vec{\eta}\big\rangle\\
    &+\frac{gg_{T}\sqrt{H}\partial_{y_{2}}v}{r}\big\langle\big|\vec{\eta}\big|^{2}\nabla_{y}H,\vec{\eta}\big\rangle+\frac{\alpha\sqrt{H}\big|\vec{\eta}\big|^{2}}{r}\big\langle\nabla_{y}H,\vec{\eta}\big\rangle\cdot\frac{\partial(gg_{T})}{\partial T}\cdot\mathcal{T}'\\
    &+\frac{\alpha gg_{T}\sqrt{H}}{r}\big\langle\big|\vec{\eta}\big|^{2}\nabla_{y}H,\nabla_{y}v\big\rangle+\frac{\alpha gg_{T}\sqrt{H}}{r}\cdot\big\langle\nabla_{y}H,\vec{\eta}\big\rangle\cdot2(\partial_{y_{2}}v+\xi\cdot|\nabla_{y}v|^{2})\\
    &+\Big(\frac{gg_{T}H^{3/2}\partial_{y_{2}}v}{r}+\frac{\alpha H^{3/2}}{r}\cdot\frac{\partial(gg_{T})}{\partial T}\cdot\mathcal{T}'\Big)\cdot\big\langle2\xi\nabla_{y}\partial_{y_{2}}v+\xi^{2}\nabla_{y}\big|\nabla_{y}v\big|^{2},\vec{\eta}\big\rangle\\
    &+\frac{\alpha gg_{T}H^{3/2}}{r}\Big(\big\langle2\xi\nabla_{y}\partial_{y_{2}}v+\xi^{2}\nabla_{y}\big|\nabla_{y}v\big|^{2},\nabla_{y}v\big\rangle+\big\langle2\nabla_{y}\partial_{y_{2}}v+2\xi\nabla_{y}\big|\nabla_{y}v\big|^{2},\vec{\eta}\big\rangle\Big).
\end{aligned}\end{equation}
    
\textbf{Step 4: Extract second order derivatives.} We write
\begin{equation}\label{eq. linearization matrix A}\begin{aligned}
    \overleftrightarrow{A}=&\Big(\alpha g^{2}+\xi g^{2}\partial_{y_{2}}v+\alpha\xi\frac{\partial(g^{2})}{\partial T}\cdot\mathcal{T}'\Big)\cdot\sqrt{H}\cdot(\vec{e}_{1}\otimes\vec{e}_{1}+\vec{e}_{2}\otimes\vec{e}_{2})\\
    &+\Big(gg_{T}H^{3/2}\partial_{y_{2}}v+\alpha H^{3/2}\cdot\frac{\partial(gg_{T})}{\partial T}\cdot\mathcal{T}'\Big)\cdot\big(2\xi\vec{e}_{2}+2\xi^{2}\nabla v\big)\otimes\vec{\eta}\\
    &+\alpha gg_{T}H^{3/2}\cdot\Big[\big(2\xi\vec{e}_{2}+2\xi^{2}\nabla v\big)\otimes\nabla v+\big(2\vec{e}_{2}+4\xi\nabla v\big)\otimes\vec{\eta}\Big].
\end{aligned}\end{equation}
By treating $\xi$ as a parameter in the space of measurable functions, we have
\begin{equation*}
    \overleftrightarrow{A}=\overleftrightarrow{A}(\epsilon,r,\xi;\nabla v,v,y).
\end{equation*}
With the matrix $\overleftrightarrow{A}$, it then follows from \eqref{eq. super-horrible} that
\begin{equation}\label{eq. less-horrible}\begin{aligned}
    0=&A^{ij}v_{ij}+\alpha g^{2}\cdot\big\langle\nabla_{y}\sqrt{H},\nabla_{y}v\big\rangle+\Big(g^{2}\partial_{y_{2}}v+\alpha\cdot\frac{\partial(g^{2})}{\partial T}\cdot\mathcal{T}'\Big)\cdot\big\langle\nabla_{y}\sqrt{H},\vec{\eta}\big\rangle\\
    &+\alpha gg_{z}\sqrt{H}\big\langle\nabla_{y}\mathcal{Z},\nabla_{y}v\big\rangle+\Big(\alpha\sqrt{H}\cdot\frac{\partial(gg_{z})}{\partial T}\cdot\mathcal{T}'+gg_{z}\sqrt{H}\partial_{y_{2}}v\Big)\cdot\big\langle\nabla_{y}\mathcal{Z},\vec{\eta}\big\rangle\\
    &+gg_{T}\sqrt{H}\partial_{y_{2}}v\big\langle\big|\vec{\eta}\big|^{2}\nabla_{y}H,\vec{\eta}\big\rangle+\alpha\sqrt{H}\big|\vec{\eta}\big|^{2}\big\langle\nabla_{y}H,\vec{\eta}\big\rangle\cdot\frac{\partial(gg_{T})}{\partial T}\cdot\mathcal{T}'\\
    &+\alpha gg_{T}\sqrt{H}\big\langle\big|\vec{\eta}\big|^{2}\nabla_{y}H,\nabla_{y}v\big\rangle+\alpha gg_{T}\sqrt{H}\cdot\big\langle\nabla_{y}H,\vec{\eta}\big\rangle\cdot2(\partial_{y_{2}}v+\xi\cdot|\nabla_{y}v|^{2}).
\end{aligned}\end{equation}

\textbf{Step 5: Collect first order derivatives.} Recall that $\nabla v$ can be ``factored out'' from $\mathcal{T}'$, as indicated by \eqref{eq. T' yin_shi_fen_jie}. Then, \eqref{eq. less-horrible} implies that $tr(\overleftrightarrow{A}\cdot D^{2}v)+\vec{b}\cdot\nabla v=0$, where the vector function $\vec{b}$ is defined as follows:
\begin{equation}\label{eq. linearization vector b}\begin{aligned}
    \vec{b}=&\alpha g^{2}\nabla\sqrt{H}+\big\langle\nabla\sqrt{H},\vec{\eta}\big\rangle\cdot\Big(g^{2}\cdot\vec{e}_{2}+2\alpha H\frac{\partial(g^{2})}{\partial T}\cdot\vec{\eta}\Big)\\
    &+\alpha gg_{z}\sqrt{H}\nabla\mathcal{Z}+\big\langle\nabla\mathcal{Z},\vec{\eta}\big\rangle\cdot\Big(2\alpha H^{3/2}\frac{\partial(gg_{z})}{\partial T}\cdot\vec{\eta}+gg_{z}\sqrt{H}\cdot\vec{e}_{2}\Big)\\
    &+gg_{T}\sqrt{H}\big|\vec{\eta}\big|^{2}\big\langle\nabla H,\vec{\eta}\big\rangle\cdot\vec{e}_{2}+2\alpha H^{3/2}\big|\vec{\eta}\big|^{2}\big\langle\nabla H,\vec{\eta}\big\rangle\frac{\partial(gg_{T})}{\partial T}\cdot\vec{\eta}\\
    &+\alpha gg_{T}\sqrt{H}\big|\vec{\eta}\big|^{2}\cdot\nabla H+2\alpha gg_{T}\sqrt{H}\big\langle\nabla H,\vec{\eta}\big\rangle\cdot\vec{\eta}.
\end{aligned}\end{equation}
Here, the gradient terms $\nabla H$ and $\nabla\mathcal{Z}$ are understood as taking gradients in the $y$-coordinate. Similarly, the vector $\vec{b}$ has the following dependence:
\begin{equation*}
    \vec{b}=\vec{b}(\epsilon,r,\xi;\nabla v,v,y).
\end{equation*}

\begin{lemma}\label{lem. summarize horrible computations}
    Assume that $\psi(x)$ satisfies \eqref{eq. positive region equation} in $\Omega\subseteq\mathbb{R}^{2}$. Define $v(y)$ as in \eqref{eq. v(y)=S(x)/epsilon r} for $y\in\frac{1}{r}\Omega$. Then, for $\overleftrightarrow{A}(\epsilon,r,\xi;\nabla v,v,y)$ given in \eqref{eq. linearization matrix A} and for $\vec{b}(\epsilon,r,\xi;\nabla v,v,y)$ given in \eqref{eq. linearization vector b}, one has
    \begin{equation}\label{eq. over-simplified horrible equation}
        tr(\overleftrightarrow{A}\cdot D^{2}v)+\vec{b}\cdot\nabla v=0,\quad\mbox{for all }y\in\frac{1}{r}\Omega.
    \end{equation}
    Here, $\xi$ is a parameter for the coefficients $\overleftrightarrow{A}$ and $\vec{b}$, where the parameter space for $\xi$ is
    \begin{equation}\label{eq. parameter space, no regularity/measurability}
        \mathcal{PS}_{\epsilon}=\Big\{\mbox{All functions mapping from }\frac{1}{r}\Omega\mbox{ to }[0,\epsilon]\Big\}.
    \end{equation}
\end{lemma}
\begin{remark}\label{rmk. remark of C1, neglecting parameters}
    If the parameters $(\epsilon,r,\xi)$ are all fixed and no ambiguity is possible, then we can simply write
    \begin{equation*}
        \overleftrightarrow{A}(\nabla v,v,y)=\overleftrightarrow{A}(\epsilon,r,\xi;\nabla v,v,y),\quad\vec{b}(\nabla v,v,y)=\vec{b}(\epsilon,r,\xi;\nabla v,v,y).
    \end{equation*}
\end{remark}

\subsection{Ellipticity of the small-perturbation equation}
Now let us analyze the equation \eqref{eq. over-simplified horrible equation}, especially its ellipticity.
\begin{lemma}\label{eq. ellipticity of linearization}
    Define
    \begin{equation}\label{eq. limiting matrix for mach number}
        \overleftrightarrow{A_{0}}(\epsilon,r;\nabla v,v,y)=\overleftrightarrow{A}(\epsilon,r,0;\nabla v,v,y),
    \end{equation}
    where the parameter $0\in\mathcal{PS}_{\epsilon}$ is the trivial function defined in $\frac{1}{r}\Omega$. Then, one has
    \begin{align*}
        \overleftrightarrow{A_{0}}(\epsilon,r;\nabla v,v,y)=&g\big(\mathcal{T}(0,y),\mathcal{Z}(y)\big)^{2}\sqrt{H\big(\mathcal{Z}(y)\big)}\cdot(\vec{e}_{1}\otimes\vec{e}_{1}+\vec{e}_{2}\otimes\vec{e}_{2})\\
        &+2g\big(\mathcal{T}(0,y),\mathcal{Z}(y)\big)g_{T}\big(\mathcal{T}(0,y),\mathcal{Z}(y)\big)H\big(\mathcal{Z}(y)\big)^{3/2}\cdot(\vec{e}_{2}\otimes\vec{e}_{2}).
    \end{align*}
    Then, there exists an increasing function $\mathcal{M}(\cdot):[0,+\infty)\to(0,+\infty)$, such that for all $\xi$ in the parameter space $\mathcal{PS}_{\epsilon}$ given in \eqref{eq. parameter space, no regularity/measurability}, one has
    \begin{equation*}
        \Big|\overleftrightarrow{A}(\epsilon,r,\xi;\vec{p},z,y)-\overleftrightarrow{A_{0}}(\epsilon,r;\vec{p},z,y)\Big|\leq\epsilon\cdot\mathcal{M}(|\vec{p}|),\quad\mbox{for all }y\in\frac{1}{r}\Omega\mbox{ and }|z|\leq2.
    \end{equation*}
\end{lemma}
\begin{proof}
    Let $\vec{p}=\nabla v(y)$ and $z=v(y)$. If $|\vec{p}|\leq K$ for some fixed $K$ and $|z|\leq2$, one has $|\widetilde{x}_{2}-x_{2}|=O(\epsilon r)$, $|\alpha(y)-1|=O(\epsilon K)$, $|\vec{\eta}(y)-\vec{e}_{2}|=O(\epsilon K)$. Moreover, $|\mathcal{T}(\xi,y)-\mathcal{T}(0,y)|=O(\epsilon K+\epsilon^{2}K^{2})$. Combining these estimates gives the error between $\overleftrightarrow{A}(\epsilon,r,\xi;\vec{p},z,y)$ and $\overleftrightarrow{A_{0}}(\epsilon,r;\vec{p},z,y)$.
\end{proof}
\begin{remark}
    Like in Remark~\ref{rmk. remark of C1, neglecting parameters}, if the parameters $(\epsilon,r)$ are all fixed and no ambiguity is possible, then we can simply write
    \begin{equation*}
        \overleftrightarrow{A_{0}}(\nabla v,v,y)=\overleftrightarrow{A_{0}}(\epsilon,r;\nabla v,v,y).
    \end{equation*}
\end{remark}
\begin{remark}\label{rmk. limiting mach number equation}
    From the definition of $\overleftrightarrow{A_{0}}$, we see that for any fixed $M$, for all $|\vec{p}|\leq M$, $|z|\leq1$, one has
    \begin{align*}
        \lim_{\substack{\epsilon\to0\\r\to\infty\\r\cdot y_{2}\to\infty}}\overleftrightarrow{A_{0}}(\vec{p},z,y)=&g(V^{2},+\infty)^{2}V\cdot(\vec{e}_{1}\otimes\vec{e}_{1}+\vec{e}_{2}\otimes\vec{e}_{2})\\
        &+2g(V^{2},+\infty)g_{T}(V^{2},+\infty)V^{3}\cdot(\vec{e}_{2}\otimes\vec{e}_{2}).
    \end{align*}
    If we use the constant $\mu$ in \eqref{eq. vertical mach deform}, then the above equation can be written as
    \begin{equation*}
        \lim_{\substack{\epsilon\to0\\r\to\infty\\r\cdot y_{2}\to\infty}}\overleftrightarrow{A_{0}}(\vec{p},z,y)=g(V^{2},+\infty)^{2}V\cdot(\vec{e}_{1}\otimes\vec{e}_{1}+\frac{1}{\mu^{2}}\vec{e}_{2}\otimes\vec{e}_{2}).
    \end{equation*}
\end{remark}

\begin{lemma}\label{lem. when does b converge to zero?}
    There exists a function $\underline{r}(\delta,\sigma,M)$, such that for all $0<\delta,\sigma<1<M$, as long as $r\geq\underline{r}(\delta,\sigma,M)$ and $\epsilon\leq\delta/4$, the vector $\vec{b}$ defined in \eqref{eq. linearization vector b} satisfies
    \begin{equation*}
        \big|\vec{b}(\epsilon,r,\xi;\vec{p},z,y)\big|\leq\sigma\quad\mbox{as long as }y\in\frac{1}{r}\Omega\cap\{y_{2}\geq\delta\},\ |\vec{p}|\leq M,\mbox{ and }|z|\leq2.
    \end{equation*}
\end{lemma}
\begin{remark}\label{rmk. after C.3}
    In Lemma~\ref{lem. when does b converge to zero?}, as long as $|v|\leq2$ and $y_{2}\geq\delta$, we must have $\psi(ry)\geq\Psi(ry_{2}-\epsilon r)\geq\Psi(\frac{\delta}{2}r)$ by the assumption $\epsilon\leq\frac{\delta}{4}$. Therefore, as long as $|v|\leq2$, $\Psi^{-1}\big(\psi(y)\big)$ is well-defined for $y\in\frac{1}{r}\Omega\cap\{y_{2}\geq\delta\}$.
\end{remark}
\begin{proof}[Proof of Lemma~\ref{lem. when does b converge to zero?}]
    From \eqref{eq. linearization vector b}, as long as $|z|\leq2$, the vector $\vec{b}$ can be written as
    \begin{equation}\label{eq. vec b 3 term}
        \vec{b}=\mu_{1}(y)\cdot\nabla H+\mu_{2}(y)\cdot gg_{Z}\nabla\mathcal{Z}+\mu_{3}(y)\cdot\frac{\partial(gg_{z})}{\partial T}\nabla\mathcal{Z}.
    \end{equation}
    Here, $\mu_{i}(y)$'s are $\mathbb{R}^{2\times2}$ matrix functions satisfying $|\mu_{i}(y)|\leq C(|\vec{p}|)$ for $i\in\{1,2,3\}$. Therefore, it suffices to estimate $\nabla H$, $gg_{z}\nabla\mathcal{Z}$, and $\frac{\partial(gg_{z})}{\partial T}\nabla\mathcal{Z}$ separately.

    \textbf{Estimate of $\nabla H$.} Recall that $H=H(\Psi(\widetilde{x}_{2}))$, then by the chain rule, we have
    \begin{equation}\label{eq. nabla H chain rule}
        |\nabla H|=\Big|\frac{dH}{dz}(\Psi)\Big|\cdot\Big|\frac{d\Psi}{dt}(\widetilde{x}_{2})\Big|\cdot\Big|\nabla_{y}\widetilde{x}_{2}\Big|\leq r\cdot C(M)\cdot\Big|\frac{dH}{dz}(\Psi)\Big|.
    \end{equation}
    In fact, as long as $|\vec{p}|=|\nabla_{y}v|\leq M$, one has
    \begin{equation*}
        |\nabla_{y}\widetilde{x}_{2}|=|r\vec{e}_{2}+\epsilon r\nabla_{y}v|\leq r\cdot C(M).
    \end{equation*}
    By the discussion in Remark~\ref{rmk. after C.3}, one has $\Psi=\psi(ry)\geq\Psi(\frac{\delta}{2}r)$. By letting $r$ sufficiently large, one obtains from the assumption (A6) and the identity \eqref{eq. H(z) initial definition} that
    \begin{equation*}
        r\cdot C(M)\cdot\Big|\frac{dH}{dz}(\Psi)\Big|\lesssim r\cdot C(M)\cdot\Big|\mathcal{B}'(\frac{\delta}{2}r)\Big|\to0.
    \end{equation*}
    In particular, one has that $|\nabla H|\ll\sigma$ when $r$ is sufficiently large.

    \textbf{Estimate of $gg_{z}\nabla\mathcal{Z}$ and $\frac{\partial(gg_{z})}{\partial T}\nabla\mathcal{Z}$.} Similar to \eqref{eq. nabla H chain rule}, one has
    \begin{equation}\label{eq. nabla Z chain rule}
        |\nabla\mathcal{Z}|=\Big|\frac{d\Psi}{dt}(\widetilde{x}_{2})\Big|\cdot\Big|\nabla_{y}\widetilde{x}_{2}\Big|\leq r\cdot C(M).
    \end{equation}
    By \eqref{eq. compatibility Z} and \eqref{eq. compatibility TZ}, one has that
    \begin{equation}\label{eq. gg_z bounded by B'(Z)}
        |gg_{z}|+\Big|\frac{\partial(gg_{z})}{\partial T}\Big|\lesssim|\mathcal{B}'(\mathcal{Z})|,\quad\mbox{when }gg_{z},\frac{\partial(gg_{z})}{\partial T}\mbox{ are evaluated at }\big(\mathcal{T}(\xi,y),\mathcal{Z}(y)\big).
    \end{equation}
    Notice that for any fixed $\delta$ and any $\epsilon\leq\frac{\delta}{2}$, one has
    \begin{equation*}
        \mathcal{Z}(y)\gtrsim\widetilde{x}_{2}(y)\geq\frac{\delta}{2}r,
    \end{equation*}
    then one obtains from the assumption (A6) that $|\mathcal{B}'(\mathcal{Z})|=o(r^{-1})$. This proves that $|gg_{z}\nabla\mathcal{Z}|,\Big|\frac{\partial(gg_{z})}{\partial T}\nabla\mathcal{Z}\Big|\ll\sigma$ when $r$ is sufficiently small.

    Summarizing all estimates above gives the existence of the function $\underline{r}(\delta,\sigma,M)$.
\end{proof}

Now, we give the interior Harnack principle for the small perturbation equation.
\begin{lemma}\label{lem. interior Harnack}
    There exist two function $r_{0}(h)$ and $\epsilon_{0}(h)$, such that the following fact holds for all $h>0$: For all $\epsilon\leq\epsilon_{0}(h)$ and $r\geq r_{0}(h)$, if $\Omega=[-r,r]\cap[h r,r]$ and $|v|\leq1$ satisfies \eqref{eq. over-simplified horrible equation} (where $\xi\in\mathcal{PS}_{\epsilon}$ is fixed), then
    \begin{equation*}
        v+1\geq c(h)\cdot\big(v(\frac{9}{10}\vec{e}_{2})+1\big)\mbox{ and }1-v\geq c(h)\cdot\big(1-v(\frac{9}{10}\vec{e}_{2})\big)\quad\mbox{in }[-\frac{9}{10},\frac{9}{10}]\times[2\delta,\frac{9}{10}].
    \end{equation*}
\end{lemma}
\begin{proof}
    Notice that for all fixed $M>0$, as long as $|\nabla v|\leq M$, we infer from Lemma~\ref{eq. ellipticity of linearization} and Lemma~\ref{lem. when does b converge to zero?} (by setting $\sigma=1$) that for $r\geq\underline{r}(h,M)$ and $\epsilon\leq\overline{\epsilon}(h,M)$, $\overleftrightarrow{A}(\nabla v,v,y)$ is uniformly elliptic and that $\vec{b}(\nabla v,v,y)$ is uniformly bounded. This satisfies the small-perturbation assumption in \cite{S07}, so the Harnack principle in \cite{S07} follows.
\end{proof}
\section*{Acknowledgment}
The author would like to thank Yan Li, Sitan Lin, Georg Weiss, and Chunjing Xie (authors of \cite{LLWX26}) for fruitful discussion and support.

\noindent \textbf{Conflict of interest.} The author does not have any possible conflicts of interest.
\vspace{2mm}

\noindent \textbf{Data availability statement.}
Data sharing is not applicable to this article, as no data sets were generated or analyzed during the current study.
\vspace{2mm}

\noindent \textbf{AI declaration.} No AI tools have been used in the work or preparation.
\bibliographystyle{abbrv}
\bibliography{main}
\end{document}